\documentclass[11pt]{amsart}
\usepackage[margin=1in]{geometry}
\usepackage{xcolor}
\usepackage{amssymb}
\usepackage{amsthm}
\usepackage{tikz-cd}
\usepackage{amstext}
\usepackage[shortlabels]{enumitem}
\usepackage[normalem]{ulem}
\usepackage{verbatim}
\usepackage{longtable}
\usepackage{bbding}

\makeatletter
\newcommand{\leqnomode}{\tagsleft@true\let\veqno\@@leqno}
\newcommand{\reqnomode}{\tagsleft@false\let\veqno\@@eqno}
\makeatother

\setlist[enumerate,1]{label={\upshape(\roman*)},leftmargin=*}
\numberwithin{equation}{section}

\newtheorem{thm}[equation]{Theorem}
\newtheorem{lem}[equation]{Lemma}
\newtheorem{cor}[equation]{Corollary}
\newtheorem{prop}[equation]{Proposition}
\theoremstyle{definition}
\newtheorem{df}[equation]{Definition}
\newtheorem{ex}[equation]{Example}
\theoremstyle{remark}
\newtheorem{rem}[equation]{Remark}
\newtheorem{clm}[equation]{Claim}
\newtheorem{que}[equation]{Question}

\renewcommand{\emph}{\textbf}

\newcommand{\mg}[1]{}
\newcommand{\al}{\alpha}
\newcommand{\de}{\delta}
\newcommand{\ka}{\kappa}

\newcommand{\sq}{\subseteq}
\newcommand{\wt}{\widetilde}
\newcommand{\wh}{\widehat}
\newcommand{\vS}{\varSigma}

\newcommand{\vL}{\Lambda}
\newcommand{\ov}{\overline}
\newcommand{\B}{\mathcal{B}}

\newcommand{\e}{\emptyset}

\newcommand{\F}{\mathcal{F}}
\newcommand{\Hh}{\mathcal{H}}
\newcommand{\I}{\mathcal{I}}
\newcommand{\J}{\mathcal{J}}
\newcommand{\Kk}{\mathcal{K}}
\newcommand{\LL}{\mathcal{L}}
\newcommand{\la}{\langle}
\newcommand{\M}{\mathcal{M}}
\newcommand{\m}{\medskip}
\newcommand{\n}{\noindent}

\newcommand{\ra}{\rangle}

\newcommand{\restr}{\upharpoonright}
\newcommand{\s}{\sigma}
\newcommand{\sspan}{\operatorname{span}}

\newcommand{\T}{\mathbf{T}}
\newcommand{\NN}{\mathcal{N}}

\newcommand{\sm}{\setminus}

\newcommand{\ts}{\textstyle}

\newcommand{\vY}{\Upsilon}
\newcommand{\ups}{\upsilon}

\def\supp{\mathop{\rm supp}}

\newcommand{\N}{{\mathbb N}}

\newcommand{\R}{{\mathbb R}}
\newcommand{\Q}{{\mathbb Q}}
\newcommand{\C}{{\mathbb C}}
\newcommand{\K}{{\mathbb K}}

\newcommand{\V}{{\mathcal V}}

\def\supp{\mathop{\rm supp}}

\def\ups{\upsilon}

\title{VECTOR LIFTINGS FOR PRODUCTS OF PROBABILITY SPACES AND MEASURABLE MODIFICATIONS OF STOCHASTIC PROCESSES}

\author{M.\ R.\ Burke}%
\address{School of Mathematical and Computational Sciences, University of Prince Edward Island, Charlottetown PE, Canada C1A 4P3}%
\email{burke@upei.ca}%
\thanks{The research of the first author was supported by NSERC}%
\author{N.\ D.\ Macheras}%
\address{Dept of Stats and Insurance Science, University of Piraeus, 80 Karaoli and Dimitriou str., 185 34 Piraeus, Greece}%
\email{macheras@unipi.gr}%
\thanks{The research of the second author was partially supported by the University of Piraeus Research Center.}
\author{W.\ Strauss}%
\address{Universit\"{a}t Stuttgart, Fachbereich Mathematik, Postfach 80 11 40, D-70511 Stuttgart}%
\email{wekarist@gmx.de}%

\date{}

\begin{document}

\begin{abstract}
We investigate the properties of linear primitive liftings $\rho\colon \LL^p(\mu)\to \LL^p(\mu)$ for probability spaces $(X,\vS,\mu)$, which are linear maps selecting a representative from each class for almost everywhere equality. We call them vector liftings. They have the advantage over liftings or linear liftings that they exist for all $p\in[0,\infty]$, not only for $p=\infty$. Their relationship to products is still not clear, but we establish existence of (strong) product vector liftings for products of two factors. The vector liftings which are $2$-marginals with respect to a suitable product yielded a characterization of stochastic processes having a measurable modification modelled on one discovered by Musia{\l}, and led to a proof of the characterization that does not use liftings. The improvement relies on results on extending a measure by null sets that apply when the family of new null sets is an ideal of a $\s$-algebra larger than the domain of the measure, but not necessarily an ideal of the power set. These allow us to reduce or eliminate completeness assumptions on the measures from several of our results.
\end{abstract}

\maketitle

\section{Introduction}
\label{s:intro}

Let $(X,\vS,\mu)$ be a measure space, $0\leq p\leq\infty$. By a vector lifting for $\LL^p(\mu)$, we mean a linear map $\rho\colon \LL^p(\mu)\to \LL^p(\mu)$ mapping each constant function to itself, and satisfying $\rho(f) =_\mu f$ and $f=_\mu g$ implies $\rho(f)=\rho(g)$ for all $f,g\in \LL^p(\mu)$ (not to be confused with linear liftings which also have to be order-preserving).
Vector liftings for $\mathcal{L}^{\infty}(\mu)$ have been known to experts in the theory of liftings at least since the sixties of the last century but were left undeveloped. They were considered trivial in the times of the predominance of functional analysis, where the interest was in positive linear and multiplicative linear liftings instead (compare \cite{it}) and  there seemed to be no evidence for  getting interesting results, when applying it to the topics we deal with in this paper. Not even an extension of the definition to spaces $\mathcal{L}^p(\mu)$, $0\leq p<\infty$ was considered, where the vector lifting is a substitute for the non-existent positive linear lifting.
We investigate the notion of (strong) vector liftings on (topological) measure spaces, focusing on how vector liftings for products relate to vector liftings for the factors.

Many of our results concern extensions of measures by adding null sets. We give a general framework (Definition \ref{d:T}) for doing this by considering triples $(X,\mathcal{A},\I)$, where $\mathcal{A}$ is a $\s$-algebra of subsets gof $X$, and $\I$ is a $\s$-ideal of a $\s$-algebra $\B$ with $\mathcal{A}\sq \B$. We develop in Section \ref{s:ext} properties of extensions of measures by null sets, focusing on products where we add ideals which are skew products. These allow us to prove versions of facts about complete ideals in the absence of completeness.

Section \ref{s:v.lift.meas} establishes basic properties of vector liftings. We also show that for any diffuse probability space $(X,\vS,\mu)$, and $0\leq p<\infty$, there is no selector for the measure classes $\varphi\colon\LL^p(\mu)\to \LL^p(\mu)$ which is order-preserving (Proposition \ref{p:o-p}) generalizing the result from \cite{it} that there is no linear lifting for $\LL^p(\mu)$ when $1\leq p<\infty$.

We investigate vector liftings for product spaces. In particular we prove Theorem \ref{PV30} which implies the following.
If $(X,\mathfrak{S},\vS,\mu)$, $(Y,\mathfrak{T},T,\nu)$ and $(X\times{Y},\mathfrak{S}\times\mathfrak{T},\vY,\ups)$ are topological probability spaces, where $\ups$ extends the product measure $\mu\otimes \nu$ and satisfies Fubini's theorem with respect to the factors (in either order), and if each of $(X,\mathfrak{S},\vS,\mu)$ and $(Y,\mathfrak{T},T,\nu)$ has a strong vector lifting, then for every pair of strong vector liftings $\gamma$ for $\LL^p(\mu)$ and $\eta$ for $\LL^p(\nu)$ there exists a strong vector lifting $\pi$ for $\LL^p(\ups)$ so that $\pi(f\otimes g)=\gamma(f)\eta(g)$ for all $f\in \LL^p(\mu)$, $g\in \LL^p(\nu)$, and the vertical and horizontal sections of $\pi(f)$ all belong to $\LL^p(\nu)$ and $\LL^p(\mu)$, respectively.
As part of the proof, we need to show that for any separately continuous function $h\in \LL^p(\ups)$, if $h$ is equal almost everywhere to a sum $\sum_{i=1}^n f_i\otimes g_i$ where $f_i\in \LL^p(\nu)$, $g_i\in \LL^p(\mu)$, then the functions $f_i$ and $g_i$ in this representation can be taken to be continuous (see Corollary \ref{c:otimes}).

We also examine the notion of marginals on $\LL^p(\mu)$, $p\in[0,\infty]$, see Definition \ref{VMD}. We give some conditions equivalent to the existence of $2$-marginals in an $\LL^p$ setting (Corollary \ref{VM19b}). For the existence of vector marginals we prove in particular the following result (Corollary \ref{VM60}).
Given two complete $\tau$-additive topological probability spaces $(X,\mathfrak{S},\vS,\mu)$, $(Y,\mathfrak{T},T,\nu)$ and their $\tau$-additive product $(X\times{Y},\mathfrak{S}\times\mathfrak{T}, \vY,\ups)$, and two strong vector liftings $\lambda$ and $\eta$ on $\LL^p(\mu)$ and $\LL^p(\nu)$, respectively, there exists a product strong vector lifting $\varphi$ on $\LL^p(\ups_\mathfrak{N})$, where $\mathfrak{N}$ is the $\s$-ideal of all right nil null subsets of $X\times{Y}$ (see Definition \ref{df:dn}) and $\ups_\mathfrak{N}$ is defined as in Lemma \ref{P10}, which is a 2-marginal with respect to $\LL^p(\ups_\mathfrak{N})$.

In Section \ref{s:meas.stoch.proc}, we apply the results of Section \ref{s:v.marg} to the measurable $\LL^p$ modifications of stochastic processes, see Definition \ref{mf10}.
It is seen, that for a primitive lifting $\rho$ for $\LL^p$ of the underlying probability space, that the $\rho$-modification of a measurable $\LL^p$ stochastic process is always measurable is equivalent to $\rho$ being a marginal (Theorem \ref{slc30}) and, in case of modifications respecting continuous functions to $\rho$ being a strong marginal (Corollary \ref{s50AR}). The result does not apply to (linear) liftings unless $p=\infty$, since they typically do not exist.

Cohn \cite{co72}, Theorem 3 (see also Chung and Doob \cite{cd}) has given a necessary and sufficient condition  for a process to have a measurable modification. The result of Cohn has been extended in \cite{bms6}, Theorem 8.8. Hoffmann-J{\o}rgensen  \cite{hj} has given another necessary and sufficient condition, which only depends on the 2-dimensional marginal distributions of the process. Musia{\l} \cite{mu23}, Proposition 4.1 states another characterization of processes possessing a measurable modification. Theorem \ref{t:mm1} is inspired by that statement. See Corollary \ref{c:mm1} for the specialization to probability measures, and Remark \ref{r:mu} for how our result relates to the one in \cite{mu23}. The proof of Theorem \ref{t:mm1} involves no applications of liftings. One of the characterizations applies without completeness assumptions on the measures and the others apply with only minimal completeness assumptions. Remark \ref{r:eta.mod.2} indicates how the result could be obtained with a suitable $2$-marginal.

In Section \ref{s:fubini}, we investigate the following problem: Let $(X,\vS,\mu)$ and $(Y,T,\nu)$ be probability spaces endowed with vector liftings $\gamma$ and $\delta$, respectively. Can we define a vector lifting on the product space by means of a Fubini type formula
\[
(\gamma\odot\delta)(f)(x,y)=\gamma([\delta_\bullet(f)]^y)(x)\;\mbox{if}\;[\delta_\bullet(f)]^y\in\LL^p(\mu),\;\;\mbox{and}\;\;(\gamma\odot\delta)(f)(x,y)=0\;\mbox{otherwise}
\]
where $(\delta_\bullet(f))(x,y):=\delta(f_x)(y)$ if $f_x\in\LL^p(\nu)$ and $(\delta_\bullet(f))(x,y):=0$ otherwise, for all $f\in\LL^p(\mu\wh\otimes\nu)$ and $(x,y)\in{X}\times{Y}$? We single out classes of marginals vector liftings or linear liftings $\gamma$ and $\delta$ which admit a positive solution, see Proposition \ref{G} and Theorem \ref{K}, respectively. Corresponding results for densities and category densities are given in \cite{mms14} and \cite{mms25}, respectively.

\section{Preliminaries}
\label{s:pre}

$\N$, $\R$, and $\C$ stand for the natural, real, and complex numbers, respectively, $\ov{\R}:={\R}\cup\{-\infty,\infty\}$. We write $\mathbb{K}$ to mean $\R$ or $\C$, except in sections 3 and 4, where $\mathbb{K}$ is by default an arbitrary field. For a vector space $X$ over $\mathbb{K}$, by a basis for $X$ we mean a basis in the indexed sense, i.e., a family $\la x_i\ra_{i\in I}$ of elements of $X$ such that for each $x\in X$ there is a unique indexed family $\la a_i\ra_{i\in I}$ of elements of $\mathbb{K}$ having finite support ($a_i=0$ except for finitely many $i$) such that $x=\sum_{i\in I}a_ix_i$. When it is not important to name the index set, we will also write $\la x_i\ra$, $\la a_i\ra$ for these indexed families.

By $\mathcal{P}(X)$ we denote the set of all subsets of a set $X$ ordered by inclusion and we write $M^c:=X\setminus{M}$ for $M\in\mathcal{P}(X)$. $Y^X$ is the space of all functions from $X$ into $Y$. 
For $f,g:X\to\ov{\R}$ we abbreviate $\{x\in X:f(x)<g(x)\}$ as $\{f<g\}$, and similarly for other equality or inequality relations.

Let $X$ be a set. If $\I$ is an ideal of an algebra $\mathcal{A}$ on $X$, we write $\wh{\I}:=\{E\sq{X}: \exists A\in\mathcal{I}, E\sq{A}\}$ for the completion of $\mathcal{I}$. For $\mathcal{C}\sq\mathcal{P}(X)$, $\s(\mathcal{C})$ denotes the $\s$-algebra generated by $\mathcal{C}$.

For a given probability space $(X, \vS, \mu)$ we denote  by $(\vS,\mu)_0$, or by $\vS_0$ for simplicity, the $\s$-ideal $\{E\in\vS:\mu(E)=0\}$. If $A\in\vS$ we write $\vS\cap A$ for the $\vS$-algebra of all $E\cap A$ with $E\in\vS$. We denote by $(\mathfrak{A}_\mu,\bar{\mu})$ the \emph{measure algebra of} a measure space $(X, \vS, \mu)$, where $\mathfrak{A}_\mu = \vS/\vS_0$ is the space of all equivalence classes $A^{\bullet}$ for $A\in\vS$ modulo the $\sigma$-ideal $\vS_0$ and $\bar{\mu}(A^{\bullet}):=\mu(A)$ for $A\in\vS$.
The completion of $(X,\vS,\mu)$ will be written $(X,\wh\vS,\wh\mu)$.
For topological spaces $(X,\mathfrak{T})$ we write $\mathfrak{B}(X)$
for its Borel $\sigma$-field over $X$. We write $C(X,\K)$ for the family of all continuous $\K$-valued functions on $X$. We usually write simply $C(X)$, letting the context determine $\K$.

Let $(X,\vS,\mu)$ be a probability space and ${\mathfrak S}$ a topology on $X$ with ${\mathfrak S}\sq\vS$. We say $(X,{\mathfrak S},\vS,\mu)$ is a \emph{topological probability space}.
This space is \emph{$\tau$-additive} if for each upwards directed $\mathfrak{U}\sq\mathfrak{S}$, $\mu(\bigcup\mathfrak{U})=\sup\{\mu(U):U\in\mathfrak{U}\}$, and it is
\emph{Radon} if $\mathfrak{S}$ is Hausdorff and $\mu$ is a complete measure which is inner regular with respect to the compact sets (and therefore $\tau$-additive).

For $\tau$-additive topological probability spaces $(X,{\mathfrak S},\vS,\mu)$ and $(Y,\mathfrak{T},T,\nu)$, their \emph{$\tau$-additive product} is the unique topological probability space $(X\times Y,\mathfrak{S}\times\mathfrak{T},\Upsilon,\upsilon)$ such that $\upsilon$ is $\tau$-additive, extends the product measure, and is the completion of its restriction to
$\s((\mathfrak{S}\times \mathfrak{T})\cup(\vS\otimes T))$.

For contexts where we wish to adjoint a family $\I\sq\mathcal{P}(X)$ to the null ideal of a measure on $X$, a suitable type of family in our setting is an ideal $\I$ of an algebra larger than the algebra of measurable sets. We introduce the following definition,

\begin{df}\label{d:T}
We write $\T$ for the class of all triples $(X,\mathcal{A},\I)$ where
\begin{enumerate}
\item
$\mathcal{A}$ is an algebra of subsets of $X$.

\item
$\I$ is an ideal of an algebra $\mathcal{B}$ of subsets of $X$ such that $\mathcal{A}\sq \mathcal{B}$.
\end{enumerate}
We write $\T_\s$ for the class of all triples $(X,\mathcal{A},\I)$ in $\T$ for which $\I$ is closed under countable unions and $\mathcal{A}$ is a $\s$-algebra.
\end{df}

\begin{rem}\label{r:iosa}
(a) In \cite{bms2}, we characterized the families $\I$ which are ideals of some algebra of subsets of $X$, and identified the largest algebra $\mathcal{B}_\I$ of which $\I$ is an ideal as follows (\cite{bms2}, Proposition 2.2):
$\I\sq \mathcal{P}(X)$ is an ideal of some algebra of subsets of $X$ if and only if $\e\in\I$ and for $A,B\in \I$ we have $A\cup B\in \I$ and $A\sm B\in \I$; when this holds, there is a largest algebra $\mathcal{B}_\I$ of which $\I$ is an ideal given by $\mathcal{B}_\mathcal{I}:=\{E\sq X:\forall A\in\mathcal{I}\,(E\cap A\in\mathcal{I})\}$.

(b) We can similarly characterize what it means for $\I$ to be a $\s$-ideal of some $\s$-algebra of subsets of $X$: it means that $\I$ is an ideal of some algebra of subsets of $X$, and is closed under countable unions. Note that if $\I$ is closed under countable unions then so is $\mathcal{B}_\I$, as is clear from its definition.
Hence if $\I$ is a $\s$-ideal of an algebra $\mathcal{B}$ then $\I$ is also a $\s$-ideal of a $\s$-algebra $\mathcal{B}'$ (namely $\mathcal{B}'=\mathcal{B}_\I$). Note that in this case, $\I$ is closed under countable intersections since $A_n\in\I$, $n\in\N$, implies $A:=\bigcap_nA_n\in\mathcal{B}'$ and $A\sq A_1$, so $A\in\I$.

(c) In terms of these notions, (ii) could be restated equivalently by saying that $\I$ is an ideal of some algebra of subsets of $X$, and $\mathcal{A}\sq \mathcal{B}_\I$.
\end{rem}

\begin{ex}\label{ex:T}
Let $\mathcal{A}$ be an algebra of subsets of $X$.

(a) When $\I$ is an ideal of $\mathcal{P}(X)$ then $(X,\mathcal{A},\I)\in\T$.

(b) When $\I$ is an ideal of $\mathcal{A}$, then $(X,\mathcal{A},\I)\in\T$.

(c) When $(X,\mathcal{A},\I)\in \T$ and $\mathcal{A}'$ is a subalgebra of $\mathcal{A}$, then $(X, \mathcal{A}', \I)\in\T$.

(d) When $(X,\mathcal{A},\I)\in \T$, $\mathcal{I}\cap \mathcal{A}$ is an ideal of $\mathcal{A}$, and $(X, \mathcal{A}, \I\cap\mathcal{A})\in\T$.

\n These examples have obvious analogs with $\T_\s$ in the place of $\T$.
\end{ex}

\begin{df}\label{CP00}
Let $(X,\mathcal{A},\I)\in\T$. For sets $A,B\in\mathcal{A}$, write $A\sq_{\mathcal I}B$ if $A\setminus B\in\mathcal{I}$, $A=_{\mathcal I}B$ if $A\,\triangle\,B\in \I$. For $\mathcal{A}$-measurable functions $f,g\colon X\to\ \ov{\R}$, write $f\leq_{\mathcal I}g$ if $\{f>g\}\in\mathcal{I}$, and $f=_{\mathcal I}g$ if $\{f\not=g\}\in\mathcal{I}$.
\end{df}

When $(X,\mathcal{A},\mathcal{I})$ is $(X,\vS,\vS_0)$ for a measure space $(X,\vS,\mu)$, we write the subscript as $\mu$ instead of $\vS_0$, i.e., $A\sq_\mu B$, $f=_\mu g$, and so on.

\begin{df}\label{P01}
Let $(X,\vS,\I)\in\T$. We make the following definitions.
\begin{enumerate}[(a)]
\item
$\LL^0(\K,\vS) := \{f\in \K^X: f\ \text{is $\vS$-measurable}\}$.

\item
$\LL^\infty(\K,\vS,\I) := \{f\in\LL^0(\K,\vS):\|f\|_\infty<\infty\}$, where $\|f\|_\infty:=\inf\{M\geq{0}: |f|\leq_\I M\}$.

\item
$\mathcal{N}(\K,\vS,\I) := \{f\in\LL^0(\K,\vS):f=_\mathcal{I} 0\}$.
\end{enumerate}

\n In (b) and (c), we get equivalent definitions if we replace $\I$ by the ideal $\I\cap\vS$ of $\vS$. In (b), by standard arguments we have the formulas $\|af\|_\infty = |a|\|f\|_\infty$, $\|f+g\|_\infty\leq \|f\|_\infty + \|g\|_\infty$, $\|fg\|_\infty\leq \|f\|_\infty\|g\|_\infty$, and thus $\LL^\infty(\K,\vS,\I)$ is a subalgebra of $\LL^0(\K,\vS)$.

We call $\mathcal{N}(\K,\vS,\I)$ the \emph{null ideal} for both $\LL^0(\K,\vS)$ and $\LL^\infty(\K,\vS,\I)$. Note that a function $f\in \LL^0(\K,\vS)$ satisfying $f=_\I0$ belongs to $\LL^\infty(\K,\vS,\I)$.
When $\mathbb{K}$ is clear from the context, we omit it. So $\LL^0(\vS)$ means $\LL^0(\K,\vS)$, and similarly for $\LL^\infty(\vS,\I)$, $\mathcal{N}(\vS,\I)$.

For a measure space $(X,\vS,\mu)$, with $\I=\vS_0$ the sets defined above are the standard notions of $\LL^p(\mu)$ when $p=0,\infty$. For the case $0<p<\infty$, we make the standard definition
\begin{enumerate}[(a)]
\setcounter{enumi}{3}
\item
$\LL^p(\mathbb{K},\mu) = \LL^p(\mu) := \{f\in \LL^0(\mu): \int |f|^p\,d\mu<\infty\}$.
\end{enumerate}
The null ideal for $\LL^p(\mu)$ for all $p\in[0,\infty]$ is $\mathcal{N}(\mu) :=\{f\in\LL^0(\mu):f=_\mu 0\}$.
If $r_p\colon \LL^p(\mu)\to {L}^p(\mu)$ is the quotient map, then when the value of $p$ is clear from the context, we write $f^{\bullet}$ for $r_p(f) = f+\mathcal{N}(\mu)\in L^p(\mu)$, and for $V\sq \LL^p(\mu)$ we write $V^\bullet = r_p[V] = \{f^\bullet:f\in V\}$.

\end{df}

\begin{df}\label{3ff}
For a given set $X$, ${\mathcal D}\sq{\mathcal P}(X)$, and $\xi\in{\mathcal P}(X)^{\mathcal D}$ we consider the following properties holding for all $A,B\in\mathcal{D}$.
\begin{enumerate}
\item[(N)]
$\emptyset\in{\mathcal D}$ and $\xi(\emptyset)=\emptyset$,

\item[(E)]
$X\in{\mathcal D}$ and $\xi(X)=X$,

\item[(M)]
$A\sq B$ implies $\xi(A)\sq\xi(B)$,

\item[(V)]
$A^c\in\mathcal{D}$ and $\xi(A)\cap\xi(A^c)=\emptyset$,

\item[($\vartheta$)]
$A\cap B\in{\mathcal D}$ and $\xi(A\cap B)=\xi(A)\cap\xi(B)$.

\item[(C)]
$A^c\in\mathcal{D}$ and $\xi(A^c)=[\xi(A)]^c$.

\item[$(\Lambda)$]
$\xi$ satisfies $(\vartheta)$ and $(C)$.
\end{enumerate}
If $X$ is a topological space, we say $\xi$ is \emph{strong}, if
$U\sq\xi(U)$ for every open set $U$ belonging to $\mathcal{D}$.
For $(X,{\mathcal D},{\mathcal I})\in \T$ we consider for $\xi:\mathcal{D}\to\mathcal{P}(X)$ the following properties holding for all $A,B\in\mathcal{D}$.
\begin{enumerate}
\item[($I{1}$)]
$\xi(A)\in\mathcal{D}$ and $\xi(A)=_{\I}A$,

\item[($I2$)]
$A=_{\mathcal{I}}B$ implies $\xi(A)=\xi(B)$,
\end{enumerate}
When $(X,\mathcal{D},\I)=(X,\vS,\vS_0)$ for a measure space $(X,\Sigma,\mu)$, write $(L1)$, $(L2)$ instead of $(I1)$, $(I2)$. Also, write $\Lambda(\mu)$ (resp.\ $\vartheta(\mu)$) for the set of all $\xi\in{\mathcal P}(X)^\vS$ satisfying the conditions $(N)$, $(E)$, and $(\Lambda)$ (resp.\ $(N)$, $(E)$, and $(\vartheta)$).
\end{df}

Analogously to the properties $(I1)$ and $(I2)$ for set maps, for maps on functions we define the following properties.

\begin{df}\label{E5}
(a) Let $(X,\vS,\mu)$ be a measure space, $\mathcal{M}\sq\LL^0(\mu)$ a linear subspace containing the constant functions.
We consider for $\xi\colon \mathcal{M}\to \K^X$ the following conditions holding for all $f,g\in \mathcal{M}$.
\begin{enumerate}
\item[($l1$)]
$\xi(f)\in\mathcal{M}$ and $\xi(f)=_\mu f$.

\item[($l2$)]
$f=_\mu g$ implies $\xi(f)=\xi(g)$.
\end{enumerate}
We say $\xi$ is a \emph{primitive lifting} for $\mathcal{M}$ if $\xi$ fixes the constant functions and satisfies $(l1)$ and $(l2)$. We say $\xi$ is a \emph{vector lifting} if $\xi$ is a linear primitive lifting.
We denote by $P(\mu,\mathcal{M})$, or just $P(\mu)$, the set of all primitive liftings for $\mathcal{M}$, and by $V(\mu,\mathcal{M})$, or just $V(\mu)$, the set of all vector liftings (when $\mathcal{M}$ is a subspace).
When $\mathcal{M} = \LL^p(\mu)$, we write $P^p(\mu)$ instead of $P(\mu,\LL^p(\mu))$ and we write $V^p(\mu)$ instead of $V(\mu,\LL^p(\mu))$.
If it is necessary to specify $\K$, we write $P^p_\K(\mu)$, $V^p_\K(\mu)$.
If $X$ is also a topological space, we say $\xi$ is \emph{strong} if $\xi(f)=f$ for each $f\in C(X)\cap \mathcal{M}$.

(b)
When $\K=\R$, a vector lifting which is order-preserving ($f\leq g$ implies $\xi(f)\leq \xi(g)$) is called a \emph{linear lifting} for $\mathcal{M}$.
If $\mathcal{M}$ is closed under multiplication, a vector lifting which is multiplicative ($\xi(fg)=\xi(f)\xi(g)$) is called a \emph{lifting} for $\mathcal{M}$.
Denote by $\mathfrak{G}(\mu)$ and $\Lambda^\infty(\mu)$ the family of all linear liftings for $\LL^\infty(\mu)$ and of all liftings for $\LL^\infty(\mu)$, respectively. It is well-known and easy to check that $\Lambda^\infty(\mu)\sq \mathfrak{G}(\mu)$.
\end{df}

There is a bijective correspondence between liftings $\rho\in\Lambda(\mu)$ and liftings $\rho'\in \Lambda^\infty(\mu)$
determined by the relation $\rho'(\chi_A)=\chi_{\rho(A)}$ for all $A\in\vS$ (see \cite{it}, pages 35--36). We generally denote them both by the same symbol, writing for example $\rho(\chi_A)=\chi_{\rho(A)}$.

\begin{df}\label{df1}
(Cf.\ \cite{bms5}, page 6)
Let $X$ and $Y$ be sets. For $E\sq X\times Y$, write $E_x:=\{y\in Y:(x,y)\in E\}$ and $E^y:=\{x\in X:(x,y)\in E\}$  ($x\in X$, $y\in Y$).

(a) When discussing families of subsets of $X\times Y$ constructed from families $\mathcal{S}\sq\mathcal{P}(X)$ and $\mathcal{T}\sq\mathcal{P}(Y)$, by $\mathcal{S}\,\dot{\times}\,Y$ we will mean $\{S\times Y:S\in\mathcal{S}\}$, and similarly $X\,\dot{\times}\,\mathcal{T}:= \{X\times T:T\in\mathcal{T}\}$.

(b) The downward closures of
$\mathcal{S}\,\dot{\times}\,Y$ and $X\,\dot{\times}\,\mathcal{T}$ in $\mathcal{P}(X\times Y)$ will be denoted
\begin{align*}
\mathcal{S}\,\wh{\times}\,Y & := \{F\sq X\times Y: F\sq S\times Y\ \text{for some}\ S\in\mathcal{S}\}, \\
X\,\wh{\times}\,\mathcal{T} & := \{F\sq X\times Y: F\sq X\times T\ \text{for some}\ T\in\mathcal{T}\}.
\end{align*}
These are ideals when $\mathcal{S}$ and $\mathcal{T}$ are ideals.

(c) For $\mathcal{T}=(\mathcal{T}_x)_{x\in X}$ a collection of families $\mathcal{T}_x\sq\mathcal{P}(Y)$ indexed by $X$, define
the \emph{skew product}
\begin{align*}
\mathcal{S}\ltimes\mathcal{T} & :=\{E\sq X\times Y:\exists N_E\in\mathcal{S}\;\forall x\notin N_E\;(E_x\in\mathcal{T}_x)\}.
\end{align*}
If the indexed family is constant, $\mathcal{T}_x = \mathcal{U}$ for all $x\in X$, we write the skew product as:%
\footnote{In this general set-theoretic setting, the notation $\mathcal{S}\ltimes\mathcal{U}$ would be ambiguous if $\mathcal{U}$ were both a family of subsets of $Y$ and a collection indexed by $X$ of families of subsets of $Y$. In our intended uses the context always makes the intended meaning clear.}
\begin{align*}
\mathcal{S}\ltimes\mathcal{U} & :=\{E\sq X\times Y:\exists N_E\in\mathcal{S}\;\forall x\notin N_E\;(E_x\in\mathcal{U})\}.
\end{align*}
In this setting, interchanging the roles of the two coordinates, we also define
\begin{align*}
\mathcal{S}\rtimes\mathcal{U} & :=\{E\sq X\times Y:\exists M_E\in\mathcal{U}\;\forall y\notin M_E\;(E^y\in\mathcal{S})\}.
\end{align*}
\end{df}

The following simple observation is useful. We leave its straightforward verification to the reader.

\begin{prop}\label{p:skew}
Let $\mathcal{S}\sq\mathcal{P}(X)$ be nonempty,
$\mathcal{T}=(\mathcal{T}_x)_{x\in X}$, where $\mathcal{T}_x\sq\mathcal{P}(Y)$ for each $x\in X$.
\begin{enumerate}
\item
If $\mathcal{S}$ is upwards directed and each $\mathcal{T}_x$ is an ideal of a subalgebra $\B_x$ of $\mathcal{P}(Y)$ then $\mathcal{S}\ltimes\mathcal{T}$ is an ideal of the subalgebra $\mathcal{S}\ltimes\mathcal{B}$ of $\mathcal{P}(X\times Y)$, where $\mathcal{B}=(\mathcal{B}_x)_{x\in X}$.

\item
If $\mathcal{S}$ is upwards $\s$-directed and each $\mathcal{T}_x$ is a $\s$-ideal of a $\s$-subalgebra $\B_x$ of $\mathcal{P}(Y)$ then $\mathcal{S}\ltimes\mathcal{T}$ is a $\s$-ideal of the $\s$-subalgebra $\mathcal{S}\ltimes\mathcal{B}$ of $\mathcal{P}(X\times Y)$, where $\mathcal{B}=(\mathcal{B}_x)_{x\in X}$.
\end{enumerate}
\end{prop}

($\mathcal{S}$ is upwards ($\s$-)directed if every finite (countable) subset of $\mathcal{S}$ has an upper bound in $(\mathcal{S},\sq)$.)

\m

We make the following standard definition.

\begin{df}\label{df:ext}
Given probability spaces $(X,\vS,\mu)$ and $(X,\vS',\mu')$, we say that $(X,\vS',\mu')$ is an \textbf{extension of $(X,\vS,\mu)$ by null sets}, or just that $\mu'$ is an \textbf{extension of $\mu$ by null sets}, if $\vS\sq\vS'$, $\mu'\restr\vS=\mu$, and for all $E'\in\vS'$ there exists an $E\in\vS$ such that $\mu'(E'\,\triangle\,E)=0$.
Equivalently, $\mu'$ is an extension of $\mu$ by null sets if and only if $\mu'$ extends $\mu$ and $\vS'=\s(\vS\cup\vS'_0)$.
\end{df}

The equivalence in the definition is standard. It is also
Corollary \ref{c:P10i}.

\begin{df}\label{df:dn}
(Cf. \cite{bms5}, Definition 2.8)
Let $(X,\vS,\mu)$ and $(Y,T_x,\nu_x)_{x\in X}$ be probability spaces. Let $T:=(T_x)_{x\in X}$, $T_0:=(T_{x,0})_{x\in X}$ ($T_{x,0}=$ the null ideal of $\nu_x$), and $\nu:=(\nu_x)_{x\in X}$, $\wh{\nu}:=(\wh{\nu}_x)_{x\in X}$.

(a) $\mathfrak{N}(\mu,\nu):=\vS_0\ltimes T_0$ is the family of \emph{right nil null sets}. It is a $\sigma$-ideal of $\vS_0\ltimes T$.

When all $(Y,T_x,\nu_x)$ are equal to a single space $(Y,T,\nu)$ ($T_0$ now denoting the null ideal of $\nu$), the roles of the two coordinates can be interchanged, giving the $\sigma$-ideals $\ov{\mathfrak{N}}(\mu,\nu):=\vS_0\rtimes T_0$ (\emph{left nil null sets}) and
$\mathfrak{N}_2(\mu,\nu):=\mathfrak{N}(\mu,\nu)\cap\ov{\mathfrak{N}}(\mu,\nu)$ (\emph{(two-sided) nil null sets}) of the $\s$-algebras $\vS\rtimes T_0$, $\vS_0\ltimes T\cap \vS\rtimes T_0$, respectively.

(b) The families $\mathfrak{N}(\mu,\wh{\nu})$, $\ov{\mathfrak{N}}(\wh{\mu},\nu)$, $\mathfrak{N}_2(\wh{\mu},\wh{\nu})$ we denote simply by $\mathfrak{N}$, $\ov{\mathfrak{N}}$, $\mathfrak{N}_2$, respectively. They are $\sigma$-ideals of $\mathcal{P}(X\times Y)$, the latter two being defined only when all spaces $(Y,T_x,\nu_x)$ are equal to a single space $(Y,T,\nu)$.

(c)
$\mu\ltimes T$ is the family of all $E\sq{X\times{Y}}$ for which there exists an $N\in\vS_0$ with $E_x\in T_x$ when $x\in N^c$, and $x\mapsto \nu_x(E_x)$ is $\vS$-measurable on $N^c$. It is a Dynkin class (\cite{fr4}, 136A, 452A) on which we define $(\mu\ltimes\nu)(E):=\int_{N^c}\nu(E_x)\,d\mu(x)$ for every $E\in\mu\ltimes T$.

In the case of a single space $(Y,T,\nu)$, we can interchange the roles of the two coordinates and we get similarly a Dynkin class $\vS\rtimes\nu$ on which we define $(\mu\rtimes\nu)(E)$ in the obvious way.

(d) For a $\s$-ideal $\I\sq\vS_0$, $\mu\ltimes_\I T$ has the same definition as $\mu\ltimes T$, but with $N\in\vS_0$ replaced by $N\in\I$.
\end{df}

\begin{df}\label{df:dn:2}
Given a triple $(X,\vS,\mu)$, $(Y,T,\nu)$, $(X\times Y,\vY,\upsilon)$ of probability spaces,
we define the following conditions, where $(X\times Y,\vS\otimes T,\mu\otimes\nu)$ is the ordinary product space, $\vS\otimes T$ being the $\s$-algebra generated by the rectangles $A\times B$ ($A\in\vS$, $B\in T$), and $\I\sq\vS_0$ is a $\s$-ideal.
\begin{enumerate}
\item[${[}P_0{]}$:]
$\vS\otimes{T}\sq\vY$

\item[${[}P_1{]}$:]
$[P_0]$ and $\upsilon\restr\vS\otimes{T}=\mu\otimes\nu$, i.e., $\ups$ extends the product measure $\mu\otimes\nu$

\item[${[}P_2{]}$:]
$[P_1]$ and $\forall E'\in\vY$ $\exists E\in\vS\otimes T$ ($\ups(E'\,\triangle\,E)=0$), i.e., $\ups$ extends $\mu\otimes\nu$ by null sets.

\item[${[}C{]}$:]
$\vY\sq\mu\ltimes T$ and $\upsilon=\mu\ltimes\nu\restr\vY$

\item[${[}C{]_\I}$:]
$\vY\sq\mu\ltimes_\I T$ and $\upsilon=\mu\ltimes\nu\restr\vY$

\item[${[}\ov{C}{]}$:]
$\vY\sq\Sigma\rtimes{\nu}$ and $\upsilon=\mu\rtimes\nu\restr\vY$
\end{enumerate}
The properties $[C]$ and $[C]_\I$ we define also if $(Y,T,\nu)$ is replaced by a family $(Y,T_x,\nu_x)_{x\in X}$, using the same definitions as above, taking $T:=(T_x)_{x\in X}$, $\nu:=(\nu_x)_{x\in X}$. $[C]$ is the same as $[C]_{\vS_0}$.
\end{df}

\begin{rem}\label{r:df:dn:2}
(a) Note that $[P_0]+[C]$ implies $[P_1]$.

(b) $[P_0]$, $[P_1]$, $[P_2]$, $[C]$, $[C]_\I$, $[\ov{C}]$ are properties of the triple $(X,\vS,\mu)$, $(Y,T,\nu)$, $(X\times Y,\vY,\upsilon)$, but when the factors are clear from the context, we will sometimes say that they are properties of $(X\times Y,\vY,\upsilon)$ or just $\ups$. Similarly for $[C]$ and $[C]_\I$ when  $(Y,T,\nu)$ is replaced by a family $(Y,T_x,\nu_x)_{x\in X}$.

(c) If we regard $\nu_x$ as a measure on $X\times Y$ supported by the fiber $\{x\}\times Y$ of the projection map $\pi\colon X\times Y\to X$, then in the language of \cite{fr4}, 452E, property $[C]$ in the case of a family $(Y,T_x,\nu_x)_{x\in X}$ says that $\nu=(\nu_x)_{x\in X}$ is a \emph{disintegration} of $\ups$ over $X$ which is strongly consistent with $\pi$.%
\footnote{Note that in \cite{fr4}, functions being integrated are only assumed to be defined on a co-negligible set. See \cite{fr1}, 122K.}
\end{rem}

The following proposition states that it is enough to verify $[C]_\I$ on the sets in a $\pi$-system generating $\vY$. Cf.\ \cite{fr4}, 452A (a).

\begin{prop}\label{p:C.I}
Let $(X,\vS,\mu)$, $(Y,T_x,\nu_x)_{x\in X}$, and $(X\times Y,\vY,\upsilon)$ be probability spaces. Let $\mathcal{W}\sq\vY$ be a $\pi$-system with $\s(\mathcal{W})=\vY$. Let $\I\sq\vS_0$ be a $\s$-ideal and assume that $[C]_\I$ holds for sets $E\in \mathcal{W}$, i.e., for each $E\in \mathcal{W}$, there is an $I\in \I$ such that for all $x\in I^c$, $E_x\in T_x$, $x\mapsto \nu_x(E_x)$ is $\vS$-measurable on $I^c$, and $\ups(E) = \int_{I^c}\nu_x(E_x)\,d\mu(x)$. Then $[C]_\I$ holds.
\end{prop}

\begin{proof}
The sets $E$ satisfying $[C]_\I$ form a Dynkin system.
\end{proof}

\begin{ex}\label{ex:prod.r.c.p.}
Suppose $(X,\vS,\mu)$, $(Y,T,\nu_x)_{x\in X}$, $(X\times Y,\vS\otimes T,\ups)$ are probability spaces so that $(\nu_x)_{x\in X}$ is a \emph{product regular conditional probability} (\emph{product r.c.p.} for short) on $T$ for $\ups$ with respect to $\mu$ (see \cite{mms3}, Definition 1.1), i.e., the following two conditions are satisfied.
\begin{itemize}
\item[(D1)]
for every $B\in T$, the map $x\mapsto \nu_x(B)$ is $\vS$-measurable;
\item[(D2)]
$\ups(A\times{B})=\int_A \nu_x(B)\,d\mu(x)$ for $A\in\vS$ and $B\in T$.
\end{itemize}
(Note that by (D2), $\ups(A\times{Y})=\mu(A)$ for every $A\in\vS$.)
Then $(X,\vS,\mu)$, $(Y,T,\nu_x)_{x\in X}$, $(X\times Y,\vS\otimes T,\ups)$ satisfies $[C]_\I$ with $\I=\{\e\}$.

\begin{proof}
By (D1) and (D2), we get that $[C]_\I$ holds for the rectangles $A\times B$, $A\in\vS$, $B\in T$. Since these form a $\pi$-system generating $\vS\otimes T$, $[C]_\I$ holds by Proposition \ref{p:C.I}.
\end{proof}
\end{ex}

The following two propositions use a standard machine for converting Fubini type properties for sets into similar properties for functions. (Cf.\ \cite{fr4}, 452A, 452B, 452F.) We will frequently require Proposition \ref{p:C.for.fcts} (iv) and (iv)$^\perp$ and will use them without reference,%
\footnote{Many of our results on products $X\times Y$ have another version obtained by interchanging the role of the two factors. Borrowing a notational device from \cite{mms2}, we sometimes denote this other version by putting the symbol $^\perp$ after the number of the theorem.}
saying only that they hold ``by $[C]$'' and ``by $[\ov{C}]$,'' respectively (the latter applying only in the case of a single space $(Y,T,\nu)$).

\begin{prop}\label{p:C.for.fcts:0}
Let $X$ be a set equipped with $\I$, an upward $\s$-directed family of subsets of $X$. Let $(Y,T_x)_{x\in X}$, and $(X\times Y,\vY)$ be measurable spaces such that $\vY\sq\I\ltimes T$, where $T=(T_x)_{x\in X}$. Let $h\colon X\times Y\to Z$ be a $\vY$-measurable function, where $Z$ is a separable metrizable space.
\begin{enumerate}
\item
There is an $I\in \I$ such that $h_x$ is $T_x$-measurable when $x\in I^c$.

\item
When $Z=\ov{\R}$ or $\C$, if $ J=( J_x)_{x\in X}$, $ J_x$ a $\s$-ideal of $T_x$, and $\Kk$ is a $\s$-ideal of $\vY$ satisfying $\Kk\sq\I\ltimes J$, we can also require
$\|h_x\|_\infty \leq \|h\|_\infty$ for $x\in I^c$, where the first norm is computed with respect to $ J_x$, and the second with respect to $\Kk$.
\end{enumerate}
\end{prop}

\begin{proof}
(i) A $\vY$-measurable simple function $h$ has this property since there is an $I\in\I$ such that $x\in I^c$ implies $[h^{-1}(F)]_x\in T_x$ for each of the (finitely many) preimage sets $h^{-1}(F)$. Since every $\vY$-measurable function $h$ is a pointwise limit $h=\lim_{n\to\infty} h_n$ of $\vY$-measurable simple functions $h_n$,%
\footnote{Let $d$ be a totally bounded metric for $Z$ and let $\B_n$ be a partition of $Z$ into Borel sets of diameter at most $1/n$. For $E\in \B_n$, choose any point $x_E\in E$ and define $h_n(p)=x_E$ if and only if $p\in h^{-1}(E)$. Then the $h_n$ are $\vY$-measurable simple functions and $h=\lim_{n\to\infty} h_n$.}
with say $[h_n]_x$ $T_x$-measurable for $x\in I_n^c$, $I_n\in\I$, we get that for $x\in I^c$, where $I\in\I$ contains $\bigcup_nI_n$, $h_x = \lim_{n\to\infty} [h_n]_x$ is a pointwise limit of $T_x$-measurable functions and hence $T_x$-measurable.

(ii) For some $E\in \Kk$, $|h(x,y)|\leq \|h\|_\infty$ when $(x,y)\in E^c$. By (i) and $\Kk\sq\I\ltimes J$, there is an $I\in \I$ such that when $x\in I^c$, $h_x$ is $T_x$-measurable and $E_x\in J_x$ and therefore $\|h_x\|_\infty \leq \|h\|_\infty$.
\end{proof}

\begin{df}\label{df:[T]}
For $\mathcal{S}\sq \mathcal{P}(X\times Y)$, and $T=(T_x)_{x\in X}$, where $T_x\sq\mathcal{P}(Y)$, let $\mathcal{S}[T]$ be defined by
\[
\mathcal{S}[T] := \mathcal{S}\cap(\{\e\}\ltimes T) = \{E\in\mathcal{S} : E_x\in T_x\ \text{\rm for all $x\in X$}\}.
\]
\end{df}

\begin{cor}\label{c:C.for.fcts:0}
Let $X$ be a set, $(Y,T_x)$, $x\in X$, and $(X\times Y,\vY)$ measurable spaces, $Z$ a separable metrizable space, $h\colon X\times Y\to Z$. Let $T=(T_x)_{x\in X}$.
Then $h$ is $\vY[T]$-measurable if and only if $h$ is $\vY$-measurable and $h_x$ is $T_x$-measurable for all $x\in X$.
\end{cor}

\begin{proof}
$(\Rightarrow)$ Apply Proposition \ref{p:C.for.fcts:0} (i) taking $\vY$ to be $\vY[T]$ and taking $\I=\{\e\}$.
$(\Leftarrow)$ For $E\sq Z$ Borel, we have $h^{-1}(E)\in\vY$ and for all $x\in X$, $[h^{-1}(E)]_x = h_x^{-1}(E) \in T_x$. Thus, $h^{-1}(E)\in\vY[T]$.
\end{proof}

\begin{prop}\label{p:C.for.fcts}
Let $(X,\vS,\mu)$, $(Y,T_x,\nu_x)_{x\in X}$, and $(X\times Y,\vY,\upsilon)$ be probability spaces. Let $\I\sq\vS_0$ be a $\s$-ideal such that $[C]_\I$ holds. Let $h\colon X\times Y\to \ov{\R}$ or $h\colon X\times Y\to \C$ be a $\vY$-measurable function. Let $0\leq p\leq \infty$.
\begin{enumerate}
\item
If $h(X\times Y)\sq[0,\infty]$, there is an $I\in \I$ such that each section $h_x$ is $T_x$-measurable for $x\in I^c$, $x\mapsto \int h_x\,d\nu$ is $\vS$-measurable on $I^c$, and $\int h\,d\ups = \int_{I^c} \int h_x\,d\nu_x  \,d\mu(x)$.

\item
There are $I\in \I$ and $N\in\vS_0$ with $I\sq N$ such that $h_x$ is $T_x$-measurable when $x\in I^c$, $x\mapsto \|h_x\|_\infty$ is $\vS$-measurable on $I^c$, and $\|h_x\|_\infty \leq \|h\|_\infty$ when $x\in N^c$.

\item
If $h$ is $\ups$-integrable, then there is an $N\in \vS_0$ such that $h_x$ is $\nu_x$-integrable for $x\in N^c$, $x\mapsto \int h_x\,d\nu_x$ is $\mu$-integrable on $N^c$, and $\int h\,d\ups = \int_{N^c} \int h_x\,d\nu_x\,d\mu(x)$.

\item
If $h\in \LL^p(\ups)$, then
there is an $N\in \vS_0$ such that $h_x\in \LL^p(\nu_x)$ when $x\in N^c$. If $p=0$, we can take $N\in \I$.

\item
$[C]$ holds for the spaces $(X,\vS,\mu)$, $(Y,\wh{T}_x,\wh{\nu}_x)_{x\in X}$, and $(X\times Y,\wh{\vY},\wh{\upsilon})$.
\end{enumerate}
\end{prop}

\begin{rem}
From $[C]_\I$ we get $[C]$, so $\vY_0\sq \vS_0\ltimes T_0$ holds. But note than in general we will not have $\vY_0\sq \I\ltimes T_0$---consider the usual Borel measure on the unit square with $\I=\{\e\}$.
\end{rem}

\begin{proof}
(i) The family $\F$ of nonnegative $\vY$-measurable functions $h\colon X\times Y\to \ov{\R}$ such that for some $I\in \I$, the section $h_x$ is $T_x$-measurable for $x\in I^c$, $x\mapsto \int h_x\,d\nu_x$ is $\vS$-measurable on $I^c$, and $\int h\,d\ups = \int_{I^c} \int h_x\,d\nu_x  \,d\mu(x)$ contains the characteristic functions of elements of $\vY$ by assumption. Since $\F$ is closed under taking linear combinations with nonnegative coefficients and limits of increasing sequences, it contains all nonnegative $\vY$-measurable functions.

(ii) The first part follows by Proposition \ref{p:C.for.fcts:0} (i) and gives  $I\in \I$. For the second part, for each $r\in \Q$, let $E_r=\{|h|>r\}$. Then for $x\in X$, $[E_r]_x = \{|h_x|>r\}$. By $[C]_\I$, there is an $I_r\in \I$, $I\sq I_r$, such that for $x\in I_r^c$, $[E_r]_x\in T_x$ and $x\mapsto \nu_x[E_r]_x$ is $\vS$-measurable on $I_r^c$. Let $I_0=\bigcup_rI_r\in \I$. On $I_0^c$, each of the maps $x\mapsto \nu_x[E_r]_x$ is $\vS$-measurable, and therefore for any $c\in \R$, $\{x\in I_0^c:\|h_x\|_\infty > c\} = \{x\in I_0^c:\nu_x[E_r]_x>0$ for some $r>c\}\in\vS$. Thus, $x\mapsto \|h_x\|_\infty$ is $\vS$-measurable on $I_0^c$. The last part follows by Proposition \ref{p:C.for.fcts:0} (ii), taking $\I=\vS_0$, $J=T_0:=(T_{x,0})_{x\in X}$, $T_{x,0}$ the null ideal of $\nu_x$, $\Kk=\vY_0$.

(iii) If $h$ is real-valued and $\ups$-integrable, then corresponding to the positive and negative parts $h^+$, $h^-$, since $h^\pm\in\F$, we get $I\in \I$, which we can take to be the same for both parts, as in the definition of $\F$. Since $\int h^\pm\,d\ups = \int_{I^c} \int h_x^\pm\,d\nu_x\,d\mu(x)$ is finite, by enlarging $I$ to a set $N\in\vS_0$, we get that the integrals $\int h_x^\pm\,d\nu_x$ are finite when $x\in N^c$. We then get that the sections $h_x$ are $\nu_x$-integrable for $x\in N^c$, $x\mapsto \int h_x\,d\nu_x = \int h_x^+\,d\nu_x - \int h_x^-\,d\nu_x$ is $\mu$-integrable on $N^c$, and $\int h\,d\ups = \int h^+\,d\ups - \int h^-\,d\ups = \int_{N^c} \int h_x^+\,d\nu_x\,d\mu(x) - \int_{N^c} \int h^-_x\,d\nu_x\,d\mu(x) = \int_{N^c} \int h_x\,d\nu_x\,d\mu(x)$. If $h$ is complex-valued, apply the real-valued case to the real and imaginary parts.

(iv) The cases $p=0$ and $p=\infty$ are covered by (ii). The case $0 < p < \infty$ is covered by (iii) applied to $|h|^p$.

(v) First consider $E\in \wh{\vY}$ with $\wh{\ups}(E)=0$. There is an $F\in \vY$ such that $E\sq F$ and $\ups(F)=0$. By $\vY_0\sq \vS_0\ltimes T_0$, there is an $N\in\vS_0$ such that when $x\in N^c$, $\nu_x(F_x)=0$ and therefore $\wh{\nu}_x(E_x)=0$.

For an arbitrary $E\in \wh{\vY}$, there is a $G\in\vY$ such that $P:= E\,\triangle\,G$ satisfies $\wh{\ups}(P)=0$. By $[C]$ and the previous paragraph, there is an $N\in\vS_0$ such that when $x\in N^c$, $G_x\in T_x$, $\wh{\nu}_x(P_x)=0$, and $x\mapsto\nu_x(G_x)$ is $\vS$-measurable on $N^c$
with $\int_{N^c} \nu_x(G_x)\,d\mu(x) = \ups(G)$.
It follows that when $x\in N^c$, $E_x = G_x\,\triangle\,P_x\in \wh{T}_x$,
$x\mapsto\wh{\nu}_x(E_x) = \nu_x(G_x)$ is $\vS$-measurable on $N^c$, and $\int_{N^c} \wh{\nu}_x(E_x)\,d\mu(x) = \int_{N^c} \nu_x(G_x)\,d\mu(x) = \ups(G) = \wh{\ups}(E)$.
\end{proof}

\begin{cor}\label{c:sec}
Let $(X,\vS,\mu)$, $(Y,T_x,\nu_x)_{x\in X}$, $(X\times Y,\vY,\upsilon)$ be probability spaces. For functions $h\colon X\times Y\to \R$, and $N\in\vS_0$, $M\in T_0$, define $h_N,h_M\colon X\times Y\to\R$ by
\begin{gather*}
\text{$h_N(x,y)=h(x,y)$ when $x\in N^c$, $h_N(x,y)=0$ when $x\in N$}, \\
\text{$h^M(x,y)=h(x,y)$ when $y\in M^c$, $h^M(x,y)=0$ when $y\in M$}.
\end{gather*}
\begin{enumerate}
\item
Assume $\vS_0\,\dot{\times}\, Y\sq \vY$. If $h\in \LL^0(\ups)$ satisfies that for some $N\in\vS_0$, $h_x\in\LL^p(\nu_x)$ when $x\notin N$, then $h_N\in \LL^0(\ups)$, $h_N =_\ups h$, for all $x\in X$, $[h_N]_x\in \LL^p(\nu_x)$, and for all $y\in Y$, $[h_N]^y =_\mu [h]^y$.

\item
Assume we have a single space $(Y,T,\nu)$, and that in {\rm(i)} we have moreover that $X\,\dot{\times}\, T_0\sq \vY$ and for some $M\in T_0$, $h^y\in\LL^p(\mu)$ when $y\notin M$. Then $(h_N)^M\in \LL^0(\ups)$, $(h_N)^M =_\ups h$ and for all $x\in X$ and $y\in Y$, $[(h_N)^M]_x\in \LL^p(\nu)$ and $[(h_N)^M]^y\in\LL^p(\mu)$.
\end{enumerate}
\end{cor}

\begin{rem}\label{r:c:sec}
If we have $[C]$ in (i), or $[C]$ and $[\ov{C}]$ in (ii), these applies in particular to $h\in\LL^p(\ups)$ by Proposition \ref{p:C.for.fcts} (iv) and (iv)$^\perp$, and then the functions $h_N$ in (i), and $(h_N)^M$ in (ii), belong to $\LL^p(\ups)$ as well since they are $=_\ups h$.
\end{rem}

\begin{proof}
(i) is mostly obvious, we just need to note that we have
$[h_N]^y(x) = [h]^y(x)$ when $x\in N^c$ and $[h_N]^y(x)=0$ when $x\in N$, so $[h_N]^y =_\mu [h]^y$ for all $y\in Y$.

(ii) From the last part of (i) we get that $[h_N]^y\in\LL^p(\mu)$ when $y\notin M$. Then apply (i)$^\perp$ to $h_N$ getting
$(h_N)^M\in \LL^0(\ups)$, $(h_N)^M =_\ups h_N =_\ups h$ and for all $x\in X$ and $y\in Y$, $[(h_N)^M]_x =_\nu [h_N]_x\in \LL^p(\nu)$ and $[(h_N)^M]^y\in\LL^p(\mu)$.
\end{proof}

\section{Extensions by null sets}
\label{s:ext}

The proofs of the next two lemmas follow standard ideas and we leave most of them as exercises for the reader (cf. \cite{hp}, I and II, and \cite{el}, II, Aufgabe 6.2). The proofs are largely a matter of checking that standard arguments go through with the assumption that $(X,\vS,\I)\in\T$, instead of the stronger assumption that $\I$ is an ideal of $\mathcal{P}(X)$.

\begin{lem}\label{P10i}
Let $(X,\vS,\I)\in\T$.
\begin{enumerate}
\item
The algebra generated by $\vS\cup\mathcal{I}$ is
$\vS_\mathcal{I} := \{A\,\triangle\,I:A\in\vS,\,I\in\mathcal{I}\}$.
$\I$ is an ideal of $\vS_\I$. If $(X,\vS,\I)\in\T_\s$ then $\s(\vS\cup \mathcal{I}) = \vS_\mathcal{I}$.

\item
Suppose $(X,\vS,\I)\in\T_\s$.
Let $f\colon X\to \ov{\R}$ be $\vS_\I$-measurable. Then there is a function $g\colon X\to \ov{\R}$ which is $\vS$-measurable and is such that $f=_\I g$.
The same statement holds with $\ov{\R}$ replaced by any standard Borel space.
\end{enumerate}
\end{lem}

\begin{rem}\label{r:notT}
(a) Since $\vS$ and $\I$ are closed under symmetric difference, and the symmetric difference operation is commutative and associative, $\vS_\I$ is simply the closure of $\vS\cup\I$ under symmetric difference. Example \ref{ex:P10i.1} shows that it is not true in general for a $\s$-algebra $\vS$ of subsets of $X$ and a $\s$-ideal $\I$ of some algebra on $X$ that $\s(\vS\cup\I) = \vS_\I$. Nevertheless, if $\I'$ is the ideal of $\s(\vS\cup\I)$ generated by $\I$, then $(X,\vS,\I')\in\T_\s$, so $\s(\vS\cup\I) = \s(\vS\cup\I') = \vS_{\I'}=\{A\,\triangle\,I:A\in\vS,\,I\in\I'\}$.

(b) The assumption $(X,\vS,\I)\in\T_\s$ is not a necessary condition for the equality $\s(\vS\cup\I) = \vS_\I$, i.e., the latter does not imply that $\I$ is an ideal of $\vS_\I$, as shown by Example \ref{ex:P10i.2}.
\end{rem}

\begin{proof}
(i) We begin by checking that $\vS_\mathcal{I}$ is the algebra of subsets of $X$ generated by $\vS\cup\mathcal{I}$. Clearly an algebra containing $\vS\cup\I$ must contain $\vS_\I$ as well, and by taking either $I$ or $A$ to be empty in $A\,\triangle\,I$, we see that $\vS\cup\mathcal{I}\sq\vS_\mathcal{I}$. There remains to check that $\vS_\mathcal{I}$ is closed under complementation and unions. This follows from
\[
(A\,\triangle\,I)^c = X\,\triangle\,(A\,\triangle\,I) = (X\,\triangle\,A)\,\triangle\,I = A^c\,\triangle\,I
\]
and from the fact that for $A_1,A_2\in\vS$, $I_1,I_2\in\mathcal{I}$, letting $E:=(A_1\,\triangle\,I_1)\cup(A_2\,\triangle\,I_2)$, we have
\[
I:=E\,\triangle\,(A_1\cup A_2) \sq(I_1\cup I_2).
\]
Since $A_1$ and $A_2$ belong to $\vS\sq \mathcal{B}_\I$, we get $I\in\mathcal{B}_\I$, so $I\in\mathcal{I}$ since $\I$ is an ideal of $\mathcal{B}_\I$. Therefore $E = (A_1\cup A_2)\,\triangle\,I\in\vS_\mathcal{I}$.
Since $\vS_\mathcal{I}$ is a subalgebra of $\mathcal{B}_\I$ and $\mathcal{I}$ is an ideal of $\mathcal{B}_\I$, it follows readily that $\mathcal{I}$ is an ideal of $\vS_\mathcal{I}$. We leave for the reader the verification that $\vS_\mathcal{I}$ is a $\s$-algebra when $(X,\vS,\I)\in\T_\s$.

We sketch a proof for (ii). For $r\in\Q$, write $\{f<r\}=A_r\,\triangle\,I_r$ where $A_r\in\vS$ and $I_r\in\I$. The function $g$ defined by $g(x)=\inf\{r\in\Q:x\in A_r\}$ is $\vS$-measurable because $\{g<r\}=\bigcup_{s<r}A_s$. We have $f=_\I g$ because $f(x)=g(x)$ when $x\notin\bigcup_{r\in\Q}I_r$.
Now suppose $f$ takes values in a standard Borel space $Y$. Since $Y$ is Borel-isomorphic to a Borel subset of $\ov{\R}$ (\cite{ke}, Proposition 12.1 and Corollary 15.2), we may assume $Y$ is a Borel subset of $\ov{\R}$. Then we get $g$ as above.
The set $E := g^{-1}(\ov{R}\sm Y)$ belongs to $\vS$ since $g$ is $\vS$-measurable, and belongs to $\I$ since $f=_\I g$. By redefining $g$ on $E$ to be identically equal to a fixed value $y_0\in Y$,%
\footnote{If $Y=\e$ then $X=\e$ since $f\colon X\to Y$, so $g$ and $f$ are both the empty function.}
we arrange that $g$ takes values in $Y$.
\end{proof}

\begin{prop}\label{p:skew2}
Let $(X,\vS)$, $(X\times Y,\vY)$ be measurable spaces, $\I$ an upwards $\s$-directed family of subsets of $X$. Given $(Y,T_x, J_x)\in \T_\s$, $x\in X$, write $ J=( J_x)_{x\in X}$, $T_ J=((T_x)_{ J_x}:x\in X)$. If $\vY\sq \I\ltimes T_ J$ then $(X\times Y,\vY,\I\ltimes J)\in\T_\s$.
\end{prop}

\begin{proof}
$\I\ltimes J$ is a $\s$-ideal of the $\s$-algebra $\I\ltimes T_ J$ by (ii) of Proposition \ref{p:skew}.
\end{proof}

The following example will be useful in section \ref{s:fubini}.

\begin{ex}\label{ex:atom.factor}
Let $(X,\vS,\mu)$, $(Y,T,\nu)$, $(X\times Y,\vY,\ups)$ be probability spaces so that $[C]$ and $[\ov{C}]$ hold and $X\,\dot{\times}\,T\sq \vY$. Assume that $X$ is an atom, i.e., $\mu$ is $\{0,1\}$-valued.  Write $\pi_Y\colon X\times Y\to Y$ for the projection onto $Y$. Then%
\footnote{Part (i) is essentially \cite{bms4}, Proposition 5.2 (iii). The statement there assumes $\ups$, $\mu$ and $\nu$ are complete but no completeness is needed. For clarity, we repeat the argument.}

\begin{enumerate}
\item
$(X\,\dot{\times}\,T)_{\vY_0} = \vY$

\item
For each $f\in \LL^0(\ups)$, there are $w\in\LL^0(\nu)$ and $N_f\in T_0$ such that $f=_\ups g$, where $g=w\circ\pi_Y$, and $y\notin N_f$ implies $f^y\in \LL^0(\mu)$ and $f^y=_\mu w(y)$.

\item
For each $E\in\vY$, there are $A\in T$ and $N_E\in T_0$ such that $E =_\ups X\times A$ and $y\notin N_E$ implies $E^y\in \vS$ and $\mu(E^y)=\chi_{A}(y)$.
\end{enumerate}

\begin{proof}
(i) The inclusion $(X\,\dot{\times}\,T)_{\vY_0} \sq \vY$ is clear. For the reverse inclusion, let $E\in\vY$. By $[\ov{C}]$, there is an $N_E\in T_0$ such that $y\in N_E^c$ implies $E^y\in\vS$, and the function $y\mapsto\mu(E^y)$ is $T$-measurable on $N_E^c$. The set $A:=\{y\in N_E^c:\mu(E^y)=1\}$ belongs to $T$ and hence $X\times A\in\vY$ since $X\,\dot{\times}\,T\sq \vY$. When $y\notin A\cup N_E$, we have $\mu(E^y)=0$, so by $[\ov{C}]$, $E =_\ups X\times A$. It follows that $E\in (X\,\dot{\times}\,T)_{\vY_0}$.

(ii) We have $(X\times Y,X\,\dot{\times}\,T,\vY_0)\in T_\s$. Since $f$ is $\vY$-measurable and (i) holds, by Lemma \ref{P10i} (ii) there is a $(X\,\dot{\times}\,T)$-measurable function $g\colon X\times Y\to\R$ such that $f=_\ups g$.
Using (i) of Proposition \ref{p:C.for.fcts:0} with $\I=\{\e\}$ and $X\,\dot{\times}\,T$ in the role of $\vY$, we see that $g_x\in\LL^0(\nu)$ for all $x\in X$. Using (i)$^\perp$ of the same proposition with $\I=\{\e\}$ and $X\,\dot{\times}\,T$ in the role of $\vY$ again, and taking $\{\e,X\}$ in the role of $T$, we see that for all $y\in Y$, $g^y$ is $\{\e,X\}$-measurable, so $g^y$ is constant with a value $w(y)$. Then $w\in\LL^0(\nu)$ since $g_x=w$ for all $x\in X$.

(iii) In (ii), take $f=\chi_E$. In the proof of (ii), Lemma \ref{P10i} (ii) allows us to take $g$ to be $\{0,1\}$-valued since $f$ is $\{0,1\}$-valued. Then $w$ is $\{0,1\}$-valued, so $w=\chi_A$ for some $A\in T$ and  $g(x,y)=w(y)=\chi_A(y)=\chi_{X\times A}(x,y)$. The rest of (iii) follows from  corresponding parts of (ii).
\end{proof}
\end{ex}

\begin{cor}\label{c:P10i}
Let $(X,\vS,\mu)$ and $(X,\vS',\mu')$ be probability spaces. Then $\mu'$ is an extension of $\mu$ by null sets if and only if $\mu'$ extends $\mu$ and $\vS'=\s(\vS\cup\vS'_0)$.
\end{cor}

\begin{proof}
The forward implication is easy to see. For the converse, note that with $\I:=\vS'_0$, we have $(X,\vS,\I)\in\T_\s$ and use the fact from Lemma \ref{P10i} (i) that $\s(\vS\cup\I)=\vS_{\I}$.
\end{proof}

\begin{ex}\label{ex:P10i.1}
There are a probability space $(X,\vS,\mu)$ on a finite set $X$, and $\I$, an ideal of some algebra on $X$, such that $\mu$ has an extension to a measure $\mu'$ whose null ideal contains $\I$ but $\vS_\I$ is not an algebra of subsets of $X$.

Consider the probability space $(X,\vS,\mu)$ given by
\[
X=\{1,2,3,4\},\ \vS = \{\e,\{1,2\},\{3,4\},X\},\ \mu\{1,2\} = \mu\{3,4\}=1/2.
\]
Let $\I=\{\e,\{1,3\}\}$. The algebra generated by $\vS\cup\I$ is $\vS'=\mathcal{P}(X)$, and $\mu$ extends to $\mu'$ on $\vS'$ having $\I$ in its null ideal by setting $\mu'\{1\} = \mu'\{3\} = 0$, $\mu'\{2\} = \mu'\{4\} = 1/2$.
However, the number of elements of $\vS_\I$, $A\,\triangle\,I$ with $A\in\vS$ and $I\in\I$, is at most $|\vS|\cdot |\I| = 8$,%
\footnote{It is easily checked that the elements $A\,\triangle\,I$ consist  precisely of the sets of even cardinality.}
so $\vS'\not=\vS_\I$.
It follows that $(X,\vS,\I)\notin\T$, which we can also see directly  since $\mathcal{B}_\I$ consists of the four sets disjoint from $\{1,3\}$ and their complements, so $\vS\not\sq\mathcal{B}_\I$.
\end{ex}

\begin{ex}\label{ex:P10i.2}
There are a probability space $(X,\vS,\mu)$ on a finite set $X$, and $\I$, an ideal of some algebra on $X$, such that $\mu$ has an extension to a measure $\mu'$ with domain $\vS_\I$ whose null ideal contains $\I$ but $(X,\vS,\I)\notin\T$.

We begin with the trivial probability space $(X,\vS,\mu)$ having $X=\{a,b,c,d,e\}$, $\vS=\{\e,X\}$. Define
\[
\I=\{\e,\{a,b,c\},\{a,b\},\{c\}\},\ \  \J=\{\e,\{b,c,d\},\{b\},\{c,d\}\}.
\]
It is readily seen that each of $\I$ and $\J$ is closed under union and difference, and hence is an ideal of some algebra on $X$. Trivially $\vS\sq \mathcal{B}_\I$, so by Lemma \ref{P10i} (i), $\vS_\I$ is the algebra generated by $\vS\cup\I$. Its atoms are $\{a,b\},\{c\},\{d,e\}$, and $\mu$ extends to a probability measure $\mu'$ on this algebra having $\I$ in its null ideal by setting $\mu'\{d,e\}=1$.  This space $(X,\vS_\I,\mu')$, together with the ideal $\J$, is our example, as we now verify.

The closure of $\vS\cup\I\cup\J$ under symmetric difference is $\mathcal{P}(X)$.%
\footnote{
Since disjoint unions are symmetric differences, it is enough to show why the closure $\mathcal{S}$ of $\vS\cup\I\cup\J$ under symmetric difference contain the relevant singletons. We have $\{c\},\{b\}\in\mathcal{S}$, and then $\{a,b\}\,\triangle\,\{b\}=\{a\}\in \mathcal{S}$ and $\{c,d\}\,\triangle\,\{c\}=\{d\}\in \mathcal{S}$. Then $\{e\}=X\,\triangle\,\{a,b,c,d\} = X\,\triangle\,\{a\}\,\triangle\,\{b\}\,\triangle\,\{c\}\,\triangle\,\{d\}\in\mathcal{S}$.}
Hence the algebra $\vS''$ generated by the union $\vS_\I\cup \J$, as well as the closure $(\vS_\I)_\J$ of the union under symmetric difference, are both equal to $\mathcal{P}(X)$. $\mu''$ extends to a measure on $\vS''$ having $\J$ in its null ideal by setting $\mu''\{e\}=1$. However, $\vS_\I\not\sq\mathcal{B}_\J$ since taking $I=\{a,b,c\}\in\I$ and $J=\{b,c,d\}\in\J$, we see that $I\cap J = \{b,c\}$ which does not belong to $\J$.
Note that $\I$ and $\J$ are interchanged by the bijection $a\leftrightarrow d$, $b\leftrightarrow c$, $e\leftrightarrow e$, so we could have interchanged the roles of $\I$ and $\J$.
\end{ex}

\begin{lem}\label{P10}
Let $(X,\vS,\I)\in\T$ and let $(X,\vS,\mu)$ be a finitely additive measure space $($with $\mu$ valued in $[0,\infty])$ such that $\vS\cap\mathcal{I}\sq \vS_0$. We get a well-defined finitely additive measure $\mu_\I$ on $\vS_\mathcal{I}$ extending $\mu$ by setting $\mu_\I(A\,\triangle\,I)=\mu A$ for $A\in\vS$, $I\in\I$.
If $(X,\vS,\I)\in\T_\s$ and $\mu$ is a measure, then $\mu_\I$ is a measure.

The null ideal $(\vS_\I)_0$ of $\mu_\I$ has the following properties.
\begin{enumerate}
\item
$(\vS_\I)_0 =\{A\,\triangle\, I:A\in\vS_0,\,I\in\mathcal{I}\}$.

\item
$(\vS_\I)_0 = \mathcal{I}$ if and only if $\vS\cap\mathcal{I} = \vS_0$ if and only if $\vS_0\sq\I$.

\item
$(\vS_\I)_0 = \{A\cup I:A\in\vS_0,\,I\in\mathcal{I}\}$ if and only if $A\sm I\in\vS$ whenever $A\in \vS_0$, $I\in \I$.
\end{enumerate}
The measure algebras of $\mu$ and $\mu_\I$ are isomorphic $($by mapping $E^\bullet\in\mathfrak{A}_\mu$ to $E^\bullet\in\mathfrak{A}_{\mu_\I}$ for $E\in \vS)$.
\end{lem}

Note that $\mu_\I$ is an extension of $\mu$ by null sets since it extends $\mu$, $\I$ is contained in its null ideal, and its domain is $\vS_\I$.
As pointed out earlier, it is possible that $\I$ is an ideal of some algebra on $X$ and $\mu$ can be extended to a measure having $\I$ in its null ideal but $(X,\vS,\I)\not\in T$. (See Example \ref{ex:P10i.2} and cf.\ also Corollary \ref{c:ext}.) We reserve the notation $\mu_\I$ for the case where $(X,\vS,\I)\in T$. We state this for the countably additive setting as a definition.

\begin{df}\label{d:mu_I}
Let $(X,\vS,\mu)$ be a measure space. When $(X,\vS,\I)\in T_\s$ and $\vS\cap\I\sq\vS_0$ we define a measure $\mu_\I$ on $\vS_\I:=\{A\,\triangle\,I:A\in\vS,\,I\in\I\}$ by $\mu_\I(A\,\triangle\,I)=\mu A$ for $A\in\vS$, $I\in\I$.
\end{df}

\begin{proof}[Proof of Lemma \ref{P10}]
We check that $\mu_\I$ is a finitely additive measure and leave the rest as an exercise for the reader. If $A\,\triangle\,I = B\,\triangle\,J$ (where $A,B\in\vS$, $I,J\in\mathcal{I}$) then composing both sides with $B\,\triangle\,I$ gives $A\,\triangle\,B = I\,\triangle\,J\in\vS\cap\mathcal{I}\sq\vS_0$, so $\mu(A\,\triangle\,B)=0$ and hence $\mu(A)=\mu(A\cap B)=\mu(B)$. Thus $\mu_\I$ is well-defined. When $A\in\vS$ we have $\mu_\I A = \mu_\I(A\,\triangle\,\e) = \mu A$, so $\mu_\I$ extends $\mu$. To prove additivity of $\mu_\I$, suppose  $A_1\,\triangle\,I_1$ and $A_2\,\triangle\,I_2$ are disjoint, where $A_1,A_2\in\vS$, $I_1,I_2\in\mathcal{I}$. Then $A_1\sm I_1$ and $A_2\sm \,I_2$ are disjoint and therefore
\[
A_1\cap A_2 \sq I_1\cup I_2,
\]
which gives $A_1\cap A_2\in \vS\cap \mathcal{I}$ and hence $\mu(A_1\cap A_2)=0$. As in the proof of Lemma \ref{P10i}, for some $I\in\mathcal{I}$ we have $E:=\bigcup_n A_n\,\triangle\,I_n=(\bigcup_n A_n)\,\triangle\,I$.
It follows that
\[
\mu_\I\Bigl(\bigcup_n A_n\,\triangle\,I_n\Bigr) = \mu_\I\Bigl(\Bigl(\bigcup_n A_n\Bigr)\,\triangle\,I\Bigr) = \mu\Bigl(\bigcup_n A_n\Bigr)
= \sum_n\mu(A_n) = \sum_n\mu_\I(A_n\,\triangle\,I_n) \qedhere
\]
\end{proof}

\begin{cor}\label{c:comp}
Let $(X,\vS,\mu)$ be a measure space and let $\I$ be a $\s$-ideal of $\mathcal{P}(X)$ with $\vS\cap\I\sq\vS_0$. If either $\mu$ is complete or $\vS_0\sq \I$ then $\mu_\I$ is complete.
\end{cor}

\begin{cor}\label{c:ext.i}
Let $(X,\vS',\mu')$ be an extension of a probability space $(X,\vS,\mu)$ by null sets. Then letting $\I=\vS'_0$ we have $(X,\vS,\I)\in\T_\s$ and $(X,\vS',\mu') = (X,\vS_\I,\mu_\I)$.
\end{cor}

\begin{proof}
By Corollary \ref{c:P10i}, $\vS'=\s(\vS\cup\I)=\vS_\I$.
\end{proof}

\begin{cor}\label{P10ii}
Let $(X,\vS,\mu)$ be a measure space, $(X,\vS',\mu')$ an extension by null sets. For any $p\in[0,\infty]$, $\LL^p(\mu')=\{f+g: f\in\LL^p(\mu),\; g\in\mathcal{N}(\mu')\}$.
\end{cor}

When $p=0$ or $p=\infty$, a similar formula holds for any $(X,\vS,\I)\in\T_\s$, namely,
$\LL^0(\vS_\mathcal{I})=\{f+g: f\in\LL^0(\vS),\; g\in\mathcal{N}(\vS_\I,\I)\}$ and
$\LL^\infty(\vS_\mathcal{I},\mathcal{I})=\{f+g: f\in\LL^\infty(\vS,\mathcal{I}),\; g\in\mathcal{N}(\vS_\I,\I)\}$. The proofs are similar.

\begin{proof}
The inclusion from right to left is clear. For the other direction, suppose $h\in \LL^p(\mu')$. Get $f\in\LL^0(\vS)$ from Lemma \ref{P10i} (ii) (taking $\I=\vS'_0$) with $f=_{\mu'} h$. When $p=0$, just take $g=h-f$. When $p=\infty$, additionally note that $\|h\|_\infty = \|f\|_\infty$. When $0<p<\infty$, if $h\in \LL^p(\mu')$ then $\int|h|^p\,d\mu' =\int|f|^p\,d\mu' = \int|f|^p\,d\mu$, so $f\in  \LL^p(\mu)$.
\end{proof}

For the following corollary, cf. \cite{ro}, Proposition 39 page 325.

\begin{cor}\label{c:ext}
Let $(X,\vS,\mu)$ be a measure space.
Let $\F\sq\mathcal{P}(X)$. There is a measure space $(X,\vS',\mu')$ with $\vS\sq\vS'$ and $\mu'$ extending $\mu$ such that $\F\sq\vS'_0$ if and only if $E\sq\bigcup_{n=1}^\infty F_n$ for $E\in\vS$ and $\{F_n:n\in\N\}\sq \F$ implies $\mu(E)=0$.
When this holds, if $\I$ is the $\s$-ideal of $\mathcal{S}:=\s(\vS\cup\F)$ generated by $\F$, then $\mathcal{S}=\vS_\I$ and $\mu'\restr\mathcal{S} = \mu_{\I}$.
\end{cor}

Thus we can take $\mu'$ to be an extension of $\mu$ by null sets, and every such extension extends $\mu_\I$.

\begin{proof}
If the extension $\mu'$ exists then clearly the condition is satisfied. Conversely, if the condition is satisfied then let $\Kk$ be the $\s$-ideal of $\mathcal{P}(X)$ generated by $\F$, i.e., $\Kk$ is the collection of all subsets of $X$ covered by a sequence of elements of $\F$. The assumption gives $\Kk\cap\vS\sq\vS_0$ and therefore the extension $\mu'$ exists, specifically we can take $\mu'=\mu_\Kk$.

For the last sentence, since $\F\sq\vS_0'$, we have $\I\sq\vS'_0$. We also have $(X,\vS,\I)\in \T_\s$ since $\I$ is a $\s$-ideal of $\mathcal{S}$ and $\vS\sq\mathcal{S}$. Using Lemma \ref{P10i} (i) we get $\mathcal{S}=\s(\vS\cup\I) = \vS_\I$, and given $A\,\triangle\,I\in \vS_\I$, where $A\in\vS$, $I\in\I$, we have $\mu'(A\,\triangle\,I) = \mu'(A)=\mu(A)=\mu_\I(A\,\triangle\,I)$.
\end{proof}

From the characterization of an ideal-of-some-algebra in Remark \ref{r:iosa} (a), it is clear that given any collection $\mathcal{C}$ of subsets of $X$, there is a smallest ideal-of-some-algebra containing $\mathcal{C}$, namely the intersection of all ideals-of-some-algebra containing $\mathcal{C}$. In particular, given two ideals-of-some-algebra $\I$ and $\J$ on $X$, there is a smallest ideal-of-some-algebra $\I\vee \J$ containing $\I\cup\J$.
Similarly, there is a smallest $\s$-ideal-of-some-algebra $\I\vee_\s \J$ containing $\I\cup\J$.

\begin{prop}\label{p:struct}
Let $\I,\J$ each be an ideal of some algebra on $X$. Writing
\begin{align*}
\mathcal{A}_1 & = \{I\cap J:I\in\I,\,J\in\J\}, \\
\mathcal{A}_2 & = \{I\sm J:I\in\I,\,J\in\J\}, \\
\mathcal{A}_3 & = \{J\sm I:I\in\I,\,J\in\J\},
\end{align*}
we have that $\I\vee\J$ consists of the finite unions $\bigcup_{i=1}^nE_i$ of sets $E_i\in \mathcal{A}_1\cup \mathcal{A}_2\cup \mathcal{A}_3$.
\end{prop}

\begin{proof}
Let $\mathcal{S}$ be the collection of all finite unions $\bigcup_{i=1}^nE_i$ of sets $E_i\in \mathcal{A}_1\cup \mathcal{A}_2\cup \mathcal{A}_3$. Clearly $\mathcal{S}$ is contained in any ideal-of-some-algebra which contains $\I$ and $\J$ since any ideal-of-some-algebra is closed under differences, finite unions, and intersections. Hence $\mathcal{S}\sq \I\vee\J$. $\mathcal{S}$ contains $\I\sq \mathcal{A}_2$ and $\J\sq \mathcal{A}_3$ (consider sets $I\sm \e$, $J\sm \e$). There remains to show that $\mathcal{S}$ is an ideal-of-some-algebra. It is clearly closed under finite unions. For closure under differences, first note the following.

\begin{clm}
$A \sm B\in\mathcal{S}$ when $A,B\in\mathcal{A}_1\cup \mathcal{A}_2\cup \mathcal{A}_3$.
\end{clm}

In what follows, $I_1,I_2\in\I$, $J_1,J_2\in\J$, and we write intersections as products. If $A$ and $B$ both belong to $\mathcal{A}_1$, say $A=I_1J_1$ and $B=I_2J_2$, then
\[
A\sm B = (I_1J_1) \sm (I_2J_2) = I_1J_1(I_2^c\cup J_2^c) = I_1I_2^cJ_1 \cup I_1J_1J_2^c\in \mathcal{S}
\]
since it is the union of two elements of $\mathcal{A}_1$. If $A=I_1J_1\in \mathcal{A}_1$, $B=I_2\sm J_2\in \mathcal{A}_2$, then
\[
A\sm B = (I_1J_1) \sm (I_2\cap J_2^c) = I_1J_1(I_2^c\cup J_2) = I_1I_2^cJ_1 \cup I_1J_1J_2 \in \mathcal{S}
\]
since it is again the union of two elements of $\mathcal{A}_1$. Similarly if $A\in \mathcal{A}_1$, $B\in \mathcal{A}_3$.

\smallskip

\n If $A=I_1\sm J_1\in \mathcal{A}_2$, $B=I_2J_2\in \mathcal{A}_1$, then
\[
A\sm B = (I_1J_1^c) \sm (I_2J_2) = I_1J_1^c(I_2^c\cup J_2^c) = I_1I_2^cJ_1^c \cup I_1J_1^cJ_2^c = I_1I_2^cJ_1^c \cup I_1(J_1\cup J_2)^c \in \mathcal{S}
\]
since it is the union of two elements of $\mathcal{A}_2$. Similarly if $A\in \mathcal{A}_3$, $B\in \mathcal{A}_1$.

\smallskip

\n If $A=I_1\sm J_1\in \mathcal{A}_2$, $B=I_2\sm J_2\in \mathcal{A}_2$, then
\[
(I_1J^c_1) \sm (I_2J_2^c) = I_1J_1^c(I_2^c\cup J_2) = I_1I_2^cJ_1^c \cup I_1J_1^cJ_2 \in \mathcal{S}
\]
since it is the union of an element of $\mathcal{A}_2$ and an element of $\mathcal{A}_1$. Similarly if $A,B\in\mathcal{A}_3$.

\smallskip

\n If $A=I_1\sm J_1\in \mathcal{A}_2$, $B=J_2\sm I_2\in \mathcal{A}_3$, then
\[
(I_1J^c_1) \sm (I_2^cJ_2) = I_1J_1^c(I_2\cup J_2^c) = I_1I_2J_1^c \cup I_1J_1^cJ_2^c = I_1I_2J_1^c \cup I_1(J_1\cup J_2)^c \in \mathcal{S}
\]
since it is the union of two elements of $\mathcal{A}_2$. Similarly if $A\in\mathcal{A}_3$, $B\in\mathcal{A}_2$.

\m

From the claim, we then get that $A \sm B\in\mathcal{S}$ when $A\in\mathcal{S}$, $B\in\mathcal{A}_1\cup \mathcal{A}_2\cup \mathcal{A}_3$ since writing $A = \bigcup_i A_i$ gives $A\sm B = \bigcup_i (A_i\sm B)\in \mathcal{S}$. Then for $A,B\in\mathcal{S}$, writing $B = \bigcup_{j=1}^mB_j$ we have by induction on $m$ that $A\sm B = (A\sm(\bigcup_{j=1}^{m-1}B_j))\sm B_m$ belongs to $\mathcal{S}$.
\end{proof}

\begin{prop}\label{p:join}
Let $X$ be a set, $\vS$ an algebra of subsets of $X$, and let each of $\I$, $\J$ be an ideal-of-some-algebra on $X$.
\begin{enumerate}
\item
If $\J$ is the $\s$-ideal-of-some-algebra generated by $\I$, then $\mathcal{B}_\I\sq \mathcal{B}_\J$.

\item
$\mathcal{B}_\I\cap\mathcal{B}_\J\sq \mathcal{B}_{\I\vee\J}\sq \mathcal{B}_{\I\vee_\s\J}$.

\item
If $(X,\vS,\I)$, $(X,\vS,\J)\in \T$, then $(X,\vS,\I\vee\J)\in \T$.

\item
If $(X,\vS,\I)$, $(X,\vS,\J)\in \T_\s$, then $(X,\vS,\I\vee_\s\J)\in \T_\s$.

\item
If $(X,\vS,\I)\in\T$ and $(X,\vS_\I,\J)\in\T$ then $(\vS_\I)_\J = \vS_{\I\vee\J}$.

\item
If $(X,\vS,\I)\in\T_\s$ and $(X,\vS_\I,\J)\in\T_\s$ then $(\vS_\I)_\J = \vS_{\I\vee_\s\J}$
\end{enumerate}
\end{prop}

\begin{proof}
For any $E\sq X$ and any collection $\mathcal{S}$ of subsets of $X$, write
\[
\mathcal{C}(E,\mathcal{S}) = \{K\sq X: E\cap K\in \mathcal{S}\}
\]
and note that $\mathcal{S}_1\sq \mathcal{S}_2$ implies $\mathcal{C}(E,\mathcal{S}_1) \sq \mathcal{C}(E,\mathcal{S}_2)$.
From the formulas $E\cap\bigcup_iK_i=\bigcup_iE\cap K_i$ and $E\cap(K_1\sm K_2) = (E\cap K_1)\sm (E\cap K_2)$, we see that if $\mathcal{S}$ is an ideal of some algebra then so is $\mathcal{C}(E,\mathcal{S})$, and if $\mathcal{S}$ is a $\s$-ideal of some algebra then so is $\mathcal{C}(E,\mathcal{S})$. The property that $E\in\mathcal{B}_\I$ can be written $\I\sq\mathcal{C}(E,\I)$.

(i) If $E\in\mathcal{B}_{\I}$ then $\I\sq\mathcal{C}(E,\I)\sq\mathcal{C}(E,\J)$. Since $\mathcal{C}(E,\J)$ is a $\s$-ideal of some algebra on $X$, we get $\J\sq\mathcal{C}(E,\J)$, so $E\in\mathcal{B}_{\J}$.

(ii) Let $E\in\mathcal{B}_\I\cap\mathcal{B}_\J$. Then $\I\sq\mathcal{C}(E,\I)$ and $\J\sq\mathcal{C}(E,\J)$, so $\I\cup\J\sq\mathcal{C}(E,\I\vee\J)$. Since $\mathcal{C}(E,\I\vee\J)$ is an ideal of some algebra on $X$, we get $\I\vee\J\sq\mathcal{C}(E,\I\vee\J)$, so $E\in\mathcal{B}_{\I\vee\J}$. The second inclusion follows from (i).

(iii), (iv) From $\vS\sq\mathcal{B}_{\I}$ and $\vS\sq\mathcal{B}_{\J}$ we get $\vS\sq\mathcal{B}_{\I}\cap \mathcal{B}_{\J}\sq\mathcal{B}_{\I\vee\J}\sq\mathcal{B}_{\I\vee_\s\J}$.

(v) By (iii) and Lemma \ref{P10i} (i), $\vS_{\I\vee\J}$ is the algebra of subsets of $X$ generated by $\vS\cup(\I\vee\J)$. Since $(\vS_\I)_\J$ is an algebra containing $\I\cup\J$, it is in particular an ideal-of-some-algebra containing $\I\cup\J$, so it contains $\I\vee\J$ and hence $\vS_{\I\vee\J}\sq (\vS_\I)_\J$. Similarly, since $\vS_{\I\vee\J}$ is an algebra containing $\I\cup\J$, it contains $\vS_\I$ and hence also contains $(\vS_\I)_\J$.

(vi) Same proof as for (v), replacing $\I\vee\J$ by $\I\vee_\s\J$ and saying $\s$-algebra and $\s$-ideal-of-some-algebra instead of algebra and ideal-of-some-algebra.
\end{proof}

With regards to (i), we note that in general the inclusion $\I\sq\J$ between two ideals of some algebra on $X$ does not imply any inclusion relationship between $\mathcal{B}_\I$ and $\mathcal{B}_\J$.

\begin{ex}
(a) If $X=\{a,b,c,d\}$, $\I=\{\e,\{a,b\}\}$, and $\J$ is the algebra with atoms $\{a\}$, $\{b\}$, $\{c,d\}$ (and so in particular is an ideal of some algebra), then $\I\sq\J$ but $\mathcal{B}_\I$ and $\mathcal{B}_\J$ are not comparable as can be seen by considering the singletons $\{a\}$ and $\{c\}$. (b) If $\I$ is any ideal of some algebra with $\mathcal{B}_\I\not=\mathcal{P}(X)$ and $\J$ is its completion, then $\I\sq\J$ and $\mathcal{B}_\I$ is a proper subalgebra of $\mathcal{B}_\J=\mathcal{P}(X)$. (c) If $\J$ is any ideal of some algebra with $\mathcal{B}_\J\not=\mathcal{P}(X)$ and $\I=\{\e\}$, then $\I\sq\J$ and $\mathcal{B}_\J$ is a proper subalgebra of $\mathcal{B}_\I=\mathcal{P}(X)$.
\end{ex}

The next example shows that the second inclusion in (ii) can be proper,
and so in particular we can have $\I\vee\J\not=\I\vee_\s\J$, even when $\I$ and $\J$ are $\s$-ideals of some algebra. When $\I$ and $\J$ are complete $\s$-ideals, then $\I\vee\J = \{A\cup B:A\in\I,\,B\in\J\}$ is a $\s$-ideal. The same fact and the same formula for $\I\vee\J$ hold if $\I,\J$ are $\s$-ideals-of-some-algebra with $\I\sq\mathcal{P}(E)$ and $\I\sq\mathcal{P}(F)$ for some disjoint sets $E$ and $F$.

\begin{ex}
Let $X=\N\times \N$, $\I=\{A\times \N:A\sq\N\}$, and $\J=\{\N\times B:B\sq\N\}$. Then $\I$ and $\J$ are $\s$-algebras of subsets of $X$, and hence each is a $\s$-ideal-of-some-algebra on $X$. Since $X\in\I$, the ideal-of-some-algebra $\I\vee\J$ generated by $\I\cup\J$ is also the subalgebra of $\mathcal{P}(X)$ generated by $\I\cup\J$. Each rectangle $A\times B$ belongs to this algebra, and the elements of $\I\cup\J$ are all rectangles, so $\I\vee\J$ is the algebra of all finite unions of rectangles. On the other hand, since $X$ is countable and $\I\vee\J$ contains all singletons, $\I\vee_\s\J=\mathcal{P}(X)$. For the diagonal $D=\{(n,n):n\in\N\}$ we have $D\in\mathcal{B}_{\I\vee_\s\J} = \mathcal{P}(X)$ but $D\notin\mathcal{B}_{\I\vee\J}$ because $X\in \I\vee\J$ while $D\cap X=D$ does not belong to $\I\vee\J$ since it is not a finite union of rectangles.
\end{ex}

The following proposition is a version of Proposition 2.3 of \cite{bms5} in our present setting.

\begin{prop}\label{p:aug2}%
Let $(X,\vS,\mu)$ be a measure space. Suppose that $(X,\vS,\I)\in \T_\s$, $(X,\vS_\I,\J)\in\T_\s$, $\vS\cap\mathcal{I}\sq\vS_0$ and $\vS_\I\cap \mathcal{J}\sq (\vS_\I)_0$ so that the measures $\mu_\I$ and $(\mu_\I)_\mathcal{J}$ are defined. We have
\begin{enumerate}
\item
The domain of $(\mu_\I)_\J$ is
$(\vS_\I)_\J = \{A\,\triangle\,I\,\triangle\,J:A\in\vS_0,\,I\in\mathcal{I},\,J\in\J\}$.

\item
$(\mu_\I)_\mathcal{J}(A\,\triangle\, I\,\triangle\, J)=\mu(A)$ for $A\in\vS$, $I\in \I$, $J\in\J$.

\item
$(X,\vS,\I\vee_\s\J)\in \T_\s$ and $\vS\cap(\I\vee_\s\J)\sq\vS_0$, so that $\mu_{\I\vee_\s\J}$ is defined.

\item
$(\mu_\I)_\J = \mu_{\I\vee_\s\J}$.

\item
$\vS\cap\mathcal{J}\sq\vS_0$ and $\vS_\mathcal{J}\cap \mathcal{I}\sq (\vS_\mathcal{J})_0$.

\item
If $(X,\vS_\J,\I)\in\T_\s$ then $(\mu_\mathcal{J})_\mathcal{I}$ is also defined and
$(\mu_\I)_\mathcal{J}=(\mu_\mathcal{J})_\mathcal{I}$.
\end{enumerate}
\end{prop}

Note that $\vS_\I\cap \mathcal{J}\sq (\vS_\I)_0$ does not follow from the weaker assumption that $\vS\cap\mathcal{J}\sq\vS_0$ as shown by the probability space $(X,\vS,\mu)$ with $X=\{a,b\}$, $\vS=\{\emptyset,X\}$ taking $\mathcal{I}=\{\emptyset,\{a\}\}$, $\mathcal{J}=\{\emptyset,\{b\}\}$. The following example shows that the assumption that $(X,\vS_\J,\I)\in\T$ in (vi) can fail.

\begin{ex}\label{ex:not.symm}
There are a probability space $(X,\vS,\mu)$ on a finite set $X$, and $\I$, $\J$, ideals-of-some-algebra on $X$, such that the hypotheses of Proposition \ref{p:aug2} hold but they fail if we interchange $\I$ and $\J$.
Specifically, $(X,\vS_\J,\I)\notin\T$.

Consider the trivial probability space $(X,\vS,\mu)$ given by $X=\{1,2,3,4\}$, $\vS = \{\e,X\}$. Let $\I=\{\e,\{1,2\}\}$, $\J=\mathcal{P}(\{1,3\})$.
We have $\vS\sq\mathcal{B}_\I$ and $\vS_\I\sq\mathcal{B}_\J=\mathcal{P}(X)$. Also $\I\cap\vS=\{\e\}=\vS_0$ and
$\J\cap\vS_\I = \J\cap\{\e,\{1,2\},\{3,4\},X\} = \{\e\} \sq (\vS_\I)_0 = \{\e,\{1,2\}\}$.
However, $\mathcal{B}_\I$ consists of all sets which are disjoint from $\{1,2\}$ and their complements, so $\{1,3\}\notin\mathcal{B}_\I$ and therefore $\vS_\J\not\sq \mathcal{B}_\I$.
\end{ex}

\begin{proof}
By Definition \ref{d:mu_I}, $\mu_\I$ and $(\mu_\I)_\J$ are defined and (i), (ii) hold.

(iii) Existence of $\mu_{\I\vee_\s\J}$ follows from  Lemma \ref{P10} again, using Proposition \ref{p:join} (iv) and  $\vS\cap(\I\vee_\s\J)\sq\vS_0$ which follows from the fact that $\I$ and $\J$, and therefore also $\I\vee_\s\J$, are contained in the null ideal of $(\mu_\I)_\J$ which extends $\mu$.

(iv) The equality of the domains is from Proposition \ref{p:join} (vi). Then (ii) gives $(\mu_\I)_\J(A\,\triangle\,I\,\triangle\,J) = \mu(A) = \mu_{\I\vee_\s\J}(A\,\triangle\,I\,\triangle\,J)$ since $I\,\triangle\, J\in \I\vee_\s\J$ for $I\in\I$ and $J\in \J$.

(v) This part is as in Proposition 2.3 of \cite{bms5}. To verify $\vS\cap\mathcal{J}\sq\vS_0$, let $A\in \vS\cap\mathcal{J}$. Then $A\in \vS_\I\cap \mathcal{J}\sq (\vS_\I)_0$, so $\mu(A)=\mu_\I(A)=0$, hence $A\in\vS_0$. To verify $\vS_\mathcal{J}\cap \mathcal{I}\sq (\vS_\mathcal{J})_0$, let $A\in \vS$, $J\in\mathcal{J}$ satisfy  $A\,\triangle\,J =: I\in \mathcal{I}$. Then $A\,\triangle\,I=J \in \vS_\I\cap\mathcal{J}\sq(\vS_\I)_0$, so $\mu_\J(A\,\triangle\,J) = \mu(A) = \mu_\I(A\,\triangle\,I)=0$, and thus $A\,\triangle\,J \in (\vS_\J)_0$.

(vi) Clear from (v), as well as (ii) (or (iv)) and the same formula with $\I$ and $\J$ interchanged.
\end{proof}

\begin{df}\label{df:[C]}%
Given sets $X$ and $Y$, for $\I\sq \mathcal{P}(X)$ define
$\I\,\wh{\times}\,Y:=\{F: F\sq I\times Y\ \text{for some}\ I\in \I\}$.
Similarly, for $\mathcal{J}\sq\mathcal{P}(Y)$, define $X\,\wh{\times}\,\mathcal{J}:=\{F: F\sq X\times J\ \text{for some}\ J\in \mathcal{J}\}$.
\end{df}

Propositions \ref{p:[C]} and \ref{p:IJ[C]} are building on Proposition 2.10 and Remark 2.12 of \cite{bms5}.

\begin{prop}\label{p:[C]}%
Let $(X,\vS,\mu)$, $(Y,T,\nu)$, $(X\times Y,\vY,\upsilon)$ be probability spaces, and let $(X,\vS,\I)$ and $(X\times Y,\vY,\I')$ belong to $\T_\s$, where  $\I'\sq \I\,\wh{\times}\,Y$, and $\vS\cap\I\sq\vS_0$.
\begin{enumerate}
\item
If the triple $(X,\vS,\mu)$, $(Y,T,\nu)$, $(X\times Y,\vY,\upsilon)$ satisfies $[C]$ then $\vY\cap \I'\sq\vY_0$, so that $\ups_{\I'}$ is defined, and the triple $(X,\vS_\I,\mu_\I)$, $(Y,T,\nu)$, $(X\times Y,\vY_{\I'},\upsilon_{\I'})$ satisfies $[C]$.

\item
If the triple $(X,\vS,\mu)$, $(Y,T,\nu)$, $(X\times Y,\vY,\upsilon)$ satisfies $[\ov{C}]$ then $\vY\cap \I'\sq\vY_0$. If also for each $F\in\I'$ there is $Q\in T_0$ with $F^y\in\I$ when $y\notin Q$, then the triple $(X,\vS_\I,\mu_\I)$, $(Y,T,\nu)$, $(X\times Y,\vY_{\I'},\upsilon_{\I'})$ satisfies $[\ov{C}]$.
\end{enumerate}
\end{prop}

Taking $\I=\vS_0$ we get the following corollary which generalizes the observation in Remark 2.9 of \cite{bms6} that the first part holds when $\I'=\mathfrak{N}$.

\begin{cor}\label{c:[C]}
Given probability spaces $(X,\vS,\mu)$, $(Y,T,\nu)$, $(X\times Y,\vY,\upsilon)$, if $[C]$ holds for $\ups$, then it holds also for $\upsilon_{\I'}$ for any $\I'\sq \vS_0\,\wh{\times}\,Y$ with $(X\times Y,\vY,\I')\in\T_\s$.
The same is true with $[C]$ replaced by $[\ov{C}]$, if for each $F\in\I$ there is $Q\in T_0$ with $F^y\in\vS_0$ when $y\notin Q$.
\end{cor}

To appreciate the need for the assumption on the $F^y$, consider the following example.

\begin{ex}\label{ex:[C]}
The property $[\ov{C}]$ fails for the triple $(X,\vS,\mu)$, $(Y,T,\nu)$, $(X\times Y,\vY_\mathfrak{N},\upsilon_\mathfrak{N})$, where $(X,\vS,\mu)$ and $(Y,T,\nu)$ are the Lebesgue measure space on $[0,1]$ and $(X\times Y,\vY,\upsilon)$ is the Lebesgue measure space on $[0,1]^2$.

To see this, build a set $E\sq X\times Y$ so that $E$ is the graph of a function (and hence belongs to $\mathfrak{N}$) but for each $y\in Y$, $E^y$ is a Bernstein set (and hence is not measurable). For example, see Exercise 5 page 28 of \cite{kh} which produces a family of disjoint Bernstein sets $E^y$, $y\in[0,1]$, then take $E=\bigcup_y E^y\times\{y\}$.%
\footnote{To do this directly, let $\la(K_\al,y_\al):\al<\mathfrak{c}\ra$ enumerate all pairs $(K,y)$ consisting of an uncountable closed set $K\sq[0,1]$ and a point $y\in [0,1]$. Recursively choose two distinct points $x_\al,x'_\al\in K_\al\sm(\{x_\beta:\beta<\al\}\cup\{x'_\beta:\beta<\al\})$ and let $E=\{(x_\al,y_\al):\al<\mathfrak{c}\}$.}
\end{ex}

\begin{proof}[Proof of Proposition \ref{p:[C]}]
(i) For $\vY\cap \I'\sq\vY_0$, let $F\in \vY\cap \I'$.  We have $F\sq I\times Y$ for some $I\in \I$.
For some $P\in \vS_0$, for $x\notin P$ we have $F_x\in T$, and $x\mapsto \nu(F_x)$ is $\vS$-measurable on $P^c$. The set $A=\{x\in P^c: \nu(F_x)>0\}$ belongs to $\vS$ and is contained in $I$. Thus $A\in\vS\cap\I\sq\vS_0$ and therefore $\ups(F) = \int_{P^c} \nu(F_x)\,d\mu(x)=0$, so $F\in \vY_0$.  The existence of $\ups_{\I'}$ then follows from Lemma \ref{P10}.

To prove $[C]$, let $L\in \vY_{\I'}$. Write $L = A\,\triangle\,F$ where $A\in\vY$ and $F\in\I'$, $F\sq I\times Y$ for some $I\in \I$.
For some $P\in \vS_0$, for $x\notin P$ we have $A_x\in T$, and $x\mapsto \nu(A_x)$ is $\vS$-measurable on $P^c$. Then $P\cup I\in(\vS_\I)_0$ and when $x\notin P\cup I$ we have $L_x = A_x\,\triangle\,F_x = A_x\in T$ and $x\mapsto \nu(L_x) = \nu(A_x)$ is $\vS_\I$-measurable on $(P\cup I)^c$. Moreover, $\int_{(P\cup I)^c} \nu(L_x)\,d\mu_\I(x) = \int_{(P\cup I)^c} \nu(A_x)\,d\mu_\I(x) =
\int_{P^c} \nu(A_x)\,d\mu_\I(x) = \int_{P^c} \nu(A_x)\,d\mu(x) = \ups(A) = \ups_{\I'}(L)$.

(ii) For $\vY\cap \I'\sq\vY_0$, let $F\in \vY\cap \I'$.  We have $F\sq I\times Y$ for some $I\in \I$.
For some $Q\in T_0$, for $y\notin Q$ we have $F^y\in \vS$ and $F^y\in\I$.%
\footnote{The assumption on the $F^y$ is not needed here because $F^y\in \vS$ and $F^y\sq I$ together imply $F^y\in\I$. The assumption is used in an essential way later in the proof.}
Thus $F^y\in \vS\cap\I\sq\vS_0$. It follows by $[\ov{C}]$ that $\ups(F) = 0$.

For $[\ov{C}]$, let $L\in \vY_{\I'}$. Write $L = A\,\triangle\,F$ where $A\in\vY$, $F\in\I'$, $F\sq I\times Y$ for some $I\in \I$.
For some $Q\in T_0$, for $y\notin Q$ we have $A^y\in \vS$, $y\mapsto \mu(A^y)$ is $T$-measurable on $Q^c$, and $F^y\in\I$. Then when $y\notin Q$ we have $L^y = A^y\,\triangle\,F^y \in \vS_\I$, and $y\mapsto \mu_\I(L^y) = \mu_\I(A^y\,\triangle\,F^y) = \mu_\I(A^y) = \mu(A^y)$ is $T$-measurable on $Q^c$. Moreover, $\int_{Q^c} \mu_\I(L^y)\,d\nu(y) = \int_{Q^c} \mu(A^y)\,d\nu(y) =
\ups(A) = \ups_{\I'}(L)$. Thus the triple $(X,\vS_\I,\mu_\I)$, $(Y,T,\nu)$, $(X\times Y,\vY_{\I'},\upsilon_{\I'})$ satisfies $[\ov{C}]$.
\end{proof}

\begin{prop}\label{p:IJ[C]}%
Suppose the triple $(X,\vS,\mu)$, $(Y,T,\nu)$, $(X\times Y,\vY,\upsilon)$
satisfies $[C]$ and $[\ov{C}]$, and we are given the following.
\begin{enumerate}[\rm(a)]
\item
$(X,\vS,\I)$ and $(X\times Y,\vY,\I')$ in $\T_\s$, where  $\I'\sq \I\,\wh{\times}\,Y$, and $\vS\cap\I\sq\vS_0$.

\item
$(Y,T,\mathcal{J})$ and $(X\times Y,\vY,\J')$ in $\T_\s$, where $\J'\sq X\,\wh{\times}\,\J$, and $T\cap \mathcal{J}\sq T_0$.

\item
$F\cap G\in \I'\cap\J'$ when $F\in\I'$, $G\in\J'$.
\end{enumerate}
Then $(X\times Y,\vY_{\I'},\J')\in \T_\s$ and $\vY_{\I'}\cap \mathcal{J}'\sq(\vY_{\I'})_0$ so that $(\ups_{\I'})_{\mathcal{J}'}$ is defined, and symmetrically $(\ups_{\mathcal{J}'})_{\I'}$ is defined and $(\ups_{\I'})_{\mathcal{J}'} = (\ups_{\mathcal{J}'})_{\I'}$. If
\begin{enumerate}[\rm(a)]
\setcounter{enumi}{3}
\item
for each $F\in\I'$ there is $Q\in T_0$ with $F^y\in\I$ when $y\notin Q$, and

\item
for each $G\in\J'$ there is $P\in \vS_0$ with $G_x\in\J$ when $x\notin P$,
\end{enumerate}
then the triple $(X,\vS_\I,\mu_\I)$, $(Y,T_{\J},\nu_{\J})$, $(X\times Y,(\vY_{\I'})_{\J'},(\ups_{\I'})_{\J'})$ satisfies $[C]$ and $[\ov{C}]$.
\end{prop}

\begin{proof}
By Proposition \ref{p:[C]} (i), $\upsilon_{\I'}$ is defined.

For $(X\times Y,\vY_{\I'},\J')\in \T_\s$, let $A\in\vY$, $F\in\I'$, $G\in\J'$. We want to show that $(A\,\triangle\,F)\cap G\in \J'$. This follows from
$(A\,\triangle\,F)\cap G = (A\cap G)\,\triangle\,(F\cap G)$ and the fact that $(X\times Y,\vY,\J')\in\T_\s$ and assumption (c) holds.

To verify $\vY_{\I'}\cap \mathcal{J}'\sq(\vY_{\I'})_0$, suppose we have $A\in \vY$, $F\in\I'$, $F\sq I\times Y$ for some $I\in \I$ and suppose that $A\,\triangle\,F\sq X\times J$ for some $J\in \mathcal{J}$. Thus, $A\sq (I\times Y)\cup (X\times J)$. By $[C]$, for some $P\in \vS_0$, $x\notin P$ implies $[A]_x\in T$, and $x\mapsto \nu([A]_x)$ is $\vS$-measurable on $P^c$. When $x\notin P\cup I$, we have $[A]_x\sq J$, and hence (since $[A]_x\in T$) $[A]_x\in \J$, so $[A]_x\in T\cap\mathcal{J}\sq T_0$. Thus, the set $\{x\in P^c: \nu([A]_x)>0\}$ is contained in $I$ and therefore belongs to $\vS\cap\mathcal{I}\sq \vS_0$. Thus, from $[C]$ we get $\ups(A)=0$ and hence $\ups_{\I'}(A\,\triangle\,F)=\ups(A)=0$.

The equality $(\ups_{\I'})_{\mathcal{J}'} = (\ups_{\mathcal{J}'})_{\I'}$ then follows from Proposition \ref{p:aug2}.

Now suppose (d) and (e) hold. From Proposition \ref{p:[C]} (i), we get that
\begin{equation}\label{eq:T1}
(X,\vS_\I,\mu_\I),\ (Y,T,\nu),\ (X\times Y,\vY_{\I'},\upsilon_{\I'})
\end{equation}
satisfies $[C]$. Using this fact we can apply Proposition \ref{p:[C]} (ii)$^\perp$ to (\ref{eq:T1}), the other required assumptions being that $(Y,T,\J)$ and $(X\times Y,\vY_{\I'},\J')$ belong to $\T_\s$, $T\cap \J\sq T_0$, and for each $G\in\J'$ there is $P\in (\vS_\I)_0$ with $G_x\in\J$ when $x\notin P$. These are covered by (b) and (e), or were already proven. We get the conclusion that
\begin{equation}\label{eq:T2}
(X,\vS_\I,\mu_\I),\ (Y,T_\J,\nu_\J),\ (X\times Y,(\vY_{\I'})_{\J'},(\ups_{\I'})_{\J'})
\end{equation}
satisfies $[C]$.
Similarly, from Proposition \ref{p:[C]} (ii), we get that (\ref{eq:T1}) satisfies $[\ov{C}]$. Using this fact, we can apply Proposition \ref{p:[C]} (i)$^\perp$ to this triple, the other required assumptions being that $(Y,T,\J)$ and $(X\times Y,\vY_{\I'},\J')$ belong to $\T_\s$, and $T\cap \J\sq T_0$. These are covered by (b), or were already proven. We conclude that (\ref{eq:T2}) satisfies $[\ov{C}]$.
\end{proof}

\begin{prop}\label{p:VMLa:2}
Let $(X,\vS,\mu)$, $(Y,T_x,\nu_x)_{x\in X}$, and $(X\times{Y},\vY,\ups)$ be probability spaces. Write $T=(T_x)_{x\in X}$, $T_0=(T_{x,0})_{x\in X}$, where $T_{x,0}$ is the null ideal of $\nu_x$. Let $\I\sq\vS_0$ be a $\s$-ideal of $\vS$.
\begin{enumerate}
\item
If $[C]$ holds and $\vY\sq\I\ltimes T$ then $\ups_{\I\ltimes T_0}$ satisfies $[C]$.

\item
If $[C]_\I$ holds then $\ups_{\I\ltimes T_0}$ satisfies $[C]_\I$.
\end{enumerate}
\end{prop}

\begin{rem}
It follows that the same conclusion holds for any smaller ideal of $\I\ltimes T$ in the place of $\I\ltimes T_0$.
\end{rem}

\begin{proof}
(i) Assume that $[C]$ holds and $\vY\sq\I\ltimes T$.
We have $(X\times Y,\vY,\I\ltimes T_0)\in\T_\s$, since $\I\ltimes T_0$ is a $\s$-ideal of $\I\ltimes T$ and $\vY\sq\I\ltimes T$.
We also have
\[
(\I\ltimes T_0)\cap\vY \sq (\vS_0\ltimes T_0)\cap\vY = \vY_0.
\]
Thus, $\ups_{\I\ltimes T_0}$ is defined. To prove that $\ups_{\I\ltimes T_0}$ satisfies $[C]$, let $A\in\vY$, $F\in\I\ltimes T_0$.
By $[C]$ for $\ups$, by the definition of $\I\ltimes T_0$, and by $\I\sq\vS_0$, there is a $P\in\vS_0$ such that for $x\in P^c$ we have $A_x\in T_x$ and $F_x\in T_{x,0}$, the map $x\mapsto \nu_x(A_x)$ is $\vS$-measurable on $P^c$, and $\ups(A) = \int_{P^c}\nu_x(A_x)\,d\mu(x)$. Then for $x\in P^c$, $[A\,\triangle\,F]_x = A_x\,\triangle\,F_x\in T_x$, $x\mapsto \nu_x([A\,\triangle\,F]_x) = \nu_x(A_x)$ is $\vS$-measurable on $P^c$, and $\ups_{\I\ltimes T_0}(A\,\triangle\,F) = \ups(A) = \int_{P^c}\nu_x(A_x)\,d\mu(x) = \int_{P^c}\nu_x([A\,\triangle\,F]_x)\,d\mu(x)$.

(ii) By $[C]_\I$, the assumptions of (i) are satisfied, so $\ups_{\I\ltimes T_0}$ is defined. The proof of $[C]_\I$ proceeds as in (i), this time choosing $P\in\I$.
\end{proof}

\begin{ex}
If $[C]$ holds for the triple $(X,\vS,\mu)$, $(Y,T,\nu)$, $(X\times Y,\vY,\upsilon)$, then it is easy to see that $\vY\cap\mathfrak{N}=\vY_0$ and $\vY\cap\mathfrak{N}_2\sq \vY_0$ (with equality if $[\ov{C}]$ also holds). Hence, for $\Kk = \mathfrak{N}_2,\mathfrak{N}$ the measures $\ups_\Kk$ are defined by Lemma \ref{P10}. The spaces $(X\times Y,\vY_\Kk,\ups_\Kk)$ also satisfy $[C]$, and $(X\times Y,\vY_{\mathfrak{N}_2},\ups_{\mathfrak{N}_2})$ satisfies $[\ov{C}]$ if $(X\times Y,\vY,\ups)$ does (by Proposition \ref{p:VMLa:2} and \ref{p:VMLa:2}$^\perp$).
In particular, the probability measure $\mu\otimes_{\mathcal{I}}\nu:=(\mu\otimes\nu)_{\mathcal{I}}$ is the unique extension of $\mu\otimes\nu$ onto $\vS\otimes_{\mathcal{I}}T:=(\vS\otimes T)_{\mathcal{I}}$ having null ideal $\mathcal{I}$, where
\[
\vS\otimes T\sq\vS\,\wh{\otimes}\,T\sq\vS\otimes_{\mathfrak{N}_2}T\sq\vS\otimes _{\mathfrak{N}}T\sq\mu\ltimes T,
\]
\[
\mu\ltimes\nu\restr\vS\otimes_{\mathfrak{N}}T=\mu\otimes_{\mathfrak{N}}\nu,\ \mu\otimes_{\mathfrak{N}}\nu\restr
\vS\otimes _{\mathfrak{N}_2}T=\mu\otimes_{\mathfrak{N}_2}\nu,\ \mu\otimes_{\mathfrak{N}_2}\nu\restr\vS\,\wh{\otimes}\,T=\mu\,\wh{\otimes}\,\nu,
\]
$(X\times Y,\vS\otimes_{\mathcal{I}}T,\mu\otimes_{\mathcal{I}}\nu)$ is a complete probability space for
$\mathcal{I}=\mathfrak{N}_2,\mathfrak{N}$, and the measure algebras of
$\mu\otimes_{\mathfrak{N}}\nu$, $\mu\otimes_{\mathfrak{N}_2}\nu$, $\mu\,\wh{\otimes}\,\nu$, and $\mu{\otimes}\nu$ are Boolean isomorphic.
\end{ex}

\section{Vector liftings}
\label{s:v.lift.meas}

The definition of vector lifting in Definition \ref{E5} can be given in purely linear algebraic terms, and we begin by taking that point of view.
As is well-known, every surjective linear map $r\colon X\to Z$ of vector spaces has a linear right inverse which can be obtained by choosing a basis $\la z_i\ra$ for $Z$, selecting any $x_i\in r^{-1}(z_i)$ and extending the map $z_i\mapsto x_i$ by linearity to get a right inverse $s\colon Z\to X$. The right inverse can be given equivalently as the composition $sr\colon X\to X$, $s$ being recoverable from $sr$ by the formula $s(r(x)) = (sr)(x)$. We are interested in the case where $Z=X/Y$ is a quotient of $X$ by a subspace $Y$.

\begin{df}\label{AVdf}
Given a vector space $X$ and a linear subspace $Y$ of $X$ write $x=_Yz$ if $x-z\in Y$, $x^{\bullet}:=\{z\in X:x=_Yz\}$,
$X/Y:=\{x^{\bullet}:x\in X\}$, and let $r_Y:x\in X\mapsto x^{\bullet}\in X/Y$ be the quotient map.
A linear map $\xi:X\to X$ is called a \emph{vector lifting} for $X$ mod $Y$, if for all $x,x_1,x_2\in X$, $\xi(x)=_Yx$, and if $x_1=_Yx_2$ then $\xi(x_1)=\xi(x_2)$.
$W(X,Y)$ denotes the set of all vector liftings for $X$ mod $Y$.
Given, in addition, a linear subspace $C$ of $X$ a $\xi\in W(X,Y)$ is called \emph{$C$-strong}, if $\xi(x)=x$ for every $x\in C$. $W(X,Y,C)$ denotes the set of all $C$-strong $\xi\in W(X,Y)$.
\end{df}

\begin{rem}\label{AVR}
As recalled above, a vector lifting $\xi$ can be given equivalently as a linear map $s\colon X/Y\to X$ satisfying $s(a)\in a$ for each $a\in X/Y$, the two versions being definable from each other by the formula $\xi(x)=s(x^\bullet)$. In category theoretic terms, such an $s$ is a section of $r_Y$ (which is a retraction). (See \cite{se}, Chapter III, Definitions 9.5.2.) We have $\xi = s\circ r_Y$.
\end{rem}

Since the sections $s$ are linear, they are determined by their values on a basis for $X/Y$. It will be convenient to state this as a proposition. As it is standard linear algebra, we omit the proof.

\begin{prop}\label{p:exist.of.v.lift}
Let $X$ be a vector space, $Y$ a subspace of $X$. For each family $\la x_i\ra$ of elements of $X$ such that $\la x_i^\bullet\ra$ is a basis for $X/Y$, there is a unique vector lifting $\xi\colon X\to X$ such that $\xi(x_i)=x_i$ for each index $i$. Conversely, each vector lifting has this form for some basis $\la x_i^\bullet\ra$ of $X/Y$.
\end{prop}

\begin{prop}\label{AVP}
Given a vector space $X$, and linear subspaces $Y$ and $C$ of $X$, we have that $W(X,Y,C)\not=\emptyset$ if and only the condition
$(N_{C,Y})\;C\cap Y=\{0\}$ is satisfied.
\end{prop}

\begin{proof}
$\xi\in W(X,Y,C)$ implies for $x\in C\cap Y$ that
$x=\xi(x)=\xi(0)=0$, so $(N_{C,Y})$ holds.

Conversely, if $(N_{C,Y})$ is satisfied then $C$ and $C^\bullet = C/Y$ are linearly isomorphic. Given a basis $\la x_i\ra_{i\in I}$ for $C$, the basis $\la x_i^\bullet\ra_{i\in I}$ for $C^\bullet$ can be extended to a basis $\la x_i^\bullet\ra_{i\in J}$ for $X/Y$. The unique lifting $\xi$ satisfying $\xi(x_i)=x_i$ for all $i$ is $C$-strong.
\end{proof}

We now specialize to subspaces of $\LL^0(\mu)$, with $\K=\R$ or $\C$.
Let $(X,\vS,\mu)$ be a measure space, $\M$ a linear subspace of $\LL^0(\mu)$ containing the family $C$ of constant functions. Letting $\NN(\mu)$ be the null ideal of $\LL^0(\mu)$, the family $V(\M) = V(\mu,\M)$ of vector liftings for $\mathcal{M}$ as defined in Definition \ref{E5} are the elements of $W(\M,\M\cap\NN(\mu),C)$.
Hence, $\gamma:\M \to \M$ is a vector lifting for $\M$ if $\gamma$ is linear, satisfies $(l1)$ and $(l2)$, and $\gamma(1)=1$.

\begin{rem}\label{r:empty}
Note that when $X$ is empty, $\M=\{\e\}$ contains only the empty function which serves as every constant function. The unique function $\gamma\colon \M\to \M$ is a vector lifting. Of course this case is no interest.
\end{rem}

As mentioned in the introduction, if $\mu$ is a diffuse measure, i.e., $\mu$ has no atoms, then there are no linear liftings for $\LL^p(\R,\mu)$ when $1\leq p<\infty$ (\cite{it}, Chap. IV, Theorem 6). The proof in \cite{it} uses the continuity of positive linear functionals on $L^p(\R,\mu)$, but the argument can be modified to avoid linearity as the following proposition shows.

\begin{prop}\label{p:o-p}
Let $(X,\vS,\mu)$ be any diffuse probability space, and let $0\leq p<\infty$. There is no order-preserving selector for $\LL^p(\R,\mu)$, i.e., there is no function $\varphi\colon\LL^p(\R,\mu)\to \LL^p(\R,\mu)$ satisfying for all $f,f_1,f_2\in\LL^p(\R,\mu)$ that $\varphi(f)=_\mu f$, and $f_1\leq_\mu f_2$ implies $\varphi(f_1) \leq \varphi(f_2)$.
\end{prop}

\begin{proof}
For $0\leq p<q\leq\infty$, we have the inclusion $\LL^q(\mu)\sq \LL^p(\mu)$ (\cite{fol}, Proposition 6.12). An order-preserving selector $\varphi$ for $\LL^p(\mu)$ as in the statement would yield one for $\LL^q(\mu)$ by restriction. Hence to prove nonexistence we may assume $p\geq 1$.

Suppose $\varphi$ were such a function.
Fix for each $n$ a partition $\mathcal{S}_n$ of $X$ into $n$ sets of measure $1/n$ (\cite{fr2} 215D) and set $\mathcal{S} = \bigcup_n\mathcal{S}_n$. Choose $x\in X$ so
that $\varphi(f)(x) = f(x)$ for each of the countably many functions $f$ of the form
\[
f = \sum_{n=1}^N a_nb_n^{1/p} f_n,
\]
where $N\in\N$, $a_n,b_n$ are positive rational numbers, and $f_n=\chi_{E_n}$ with $E_n\in\mathcal{S}$.
For each $n$, choose $E_n\in\mathcal{S}$ such that $x\in E_n$ and
$\mu(E_n)\leq 1/n^p$. Let $f_n=\chi_{E_n}$. Then
\[
\|f_n\|_p = \mu(E_n)^{1/p}\leq \frac1n.
\]
In $L^p(\R,\mu)$, the series $\sum_{n=1}^\infty n^{-2}f_n/\|f_n\|_p$ is absolutely summable and hence converges. Let $g\in\LL^p(\R,\mu)$ be any function in the class of the limit. Define $g_N = \sum_{n=1}^N n^{-2}f_n/\|f_n\|_p$ and write $g = g_N + h_N$. The function $h_N$ belongs to the class
\[
\sum_{n=N+1}^\infty \frac1{n^2}\frac1{\|f_n\|_p}f_n
\]
in $L^p(\R,\mu)$ and hence is almost everywhere nonnegative
since the positive cone is closed in the norm topology (see for example \cite{sc2}, II 5.2 (ii))
and the partial sums of this series are nonnegative. It follows that $g\geq g_N$ almost everywhere. By our choice of $x$, $\varphi(g_N)(x)=g_N(x)$ for all $N$.
We have for any $N\in \N$,
\[
\varphi(g)(x) \geq \varphi(g_N)(x) = g_N(x) = \sum_{n=1}^N
\frac{1}{n^2}\frac{1}{\|f_n\|_p}f_n(x) \geq \sum_{n=1}^N \frac{1}{n}\to\infty\
  \text{as}\ N\to\infty,
\]
contradiction.
\end{proof}

In contrast, vector liftings always exist for any probability space, except in trivial cases.

\begin{prop}\label{p:vect.lift}
A vector lifting for $\M$ exists if and only if either $X=\e$ or $\mu(X)>0$.
\end{prop}

\begin{proof}
If $X$ is nonempty but $\mu(X)=0$ then no vector lifting $\gamma$ exists, because the constant functions all belong to the same measure class, so they cannot all be fixed.
Conversely, if $X=\e$ then a vector lifting exists as pointed out in Remark \ref{r:empty}, and if $\mu(X)>0$, and $C$ is the subspace of constant functions in $\M$, then a vector lifting for $\M$ is a member of $W(\M,\M\cap\mathcal{N}(\mu),C)$ which exists by Proposition \ref{AVP} since nonzero constant functions are not in $\mathcal{N}(\mu)$.
\end{proof}

The following simple observation will be useful.

\begin{rem}\label{r:PV15:1}
Let $(X,\vS,\mu)$ be a probability space and let $\M$ be a subspace of $\LL^0(\mu)$.
If $\langle f_i^{\bullet}\rangle_{i\in I}$ with each $f_i\in \M$ is a linear basis for $\M^\bullet$, and we extend it to a basis for $L^p(\mu)$, $\langle f_i^{\bullet}\rangle_{i\in J}$, where $I\sq J$, then the vector lifting $\pi\in V^p(\mu)$ satisfying $\pi(f_i)=f_i$ for $i\in J$ also satisfies $\pi(f)\in \M$ for all $f\in \M$.
(Writing $f^\bullet = \sum_i a_i f_i^\bullet$, $\la a_i\ra$ of finite support in $\mathbb{K}$, we have $a_i=0$ for $i\notin I$, so $\pi(f) = \sum_ia_if_i\in \M$.)
\end{rem}

For ease of reference later, we give some examples of the use of this remark.

\begin{ex}\label{ex:PV15:1}
(a) Letting $(X,\vS,\mu)$ be the Lebesgue measure space on $X=[0,1]$, by the Lusin theorem, each class in $L^0(\mu)$ contains a function of Baire class $1$ (i.e., a pointwise limit of continuous functions). Since the Baire class $1$ functions form a subspace, it follows that there is a vector lifting for $\LL^0(\mu)$ whose values are Baire class $1$ functions.

(b) Let $(X,\vS,\mu)$, $(Y,T,\nu)$, $(X\times Y,\vY,\ups)$ be probability spaces satisfying $[C]$ and $\vS_0\,\dot{\times}\, Y\sq \vY$.
Every element $h\in \LL^p(\ups)$ is equal modulo $\ups$ to an element $h'\in \LL^p(\ups)$ satisfying $[h']_x\in \LL^p(\nu)$ for all $x\in X$ (Corollary \ref{c:sec} (i)). Since the family of such elements $h'$ is a subspace, it follows that there is a vector lifting $\pi\in V^p(\ups)$ whose values satisfy $[\pi(f)]_x\in\LL^p(\nu)$ for $x\in X$.

(c) If in (b) we also have $[\ov{C}]$ and $X\,\dot{\times}\,T_0\sq \vY$, then we get a vector lifting $\pi\in V^p(\ups)$ whose values satisfy $[\pi(f)]_x\in\LL^p(\nu)$ for $x\in X$ and $[\pi(f)]^y\in\LL^p(\nu)$ for $y\in Y$ (using Corollary \ref{c:sec} (ii)).

(d) Let $(X,\mathfrak{S},\vS,\mu)$ be a topological probability space.
If $\langle f_i^{\bullet}\rangle_{i\in I}$ is a basis for $C^p(X)^\bullet$, with each $f_i$ continuous, which we extend to a basis $\langle f_i^{\bullet}\rangle_{i\in J}$, $I\sq J$, for $L^p(\mu)$, then the vector lifting $\pi\in V^p(\mu)$ given by $\pi(f_i)=f_i$ for $i\in J$ satisfies that $\pi(f)$ is continuous for each $f\in C^p(X)$. Under the condition $(N_{C,\mu})$ defined next, $\langle f_i\rangle_{i\in I}$ is a basis for $C^p(X)$ and we get $\pi(f)=f$ for each $f\in C^p(X)$.
\end{ex}

\begin{df}\label{Vdf2}
Given a topological measure space $(X,\mathfrak{S},\vS,\mu)$ we put
\[
C^p(X,\K):=C(X,{\K})\cap\LL^p(\K,\mu).
\]
As with the notation for $\LL^p(\K,\mu)$, we will omit the mention of $\K$, writing simply $C^p(X)$. In principle, when $\K=\C$, we could consider either $C^p(X,\R)$ or $C^p(X,\C)$ inside of $\LL^p(\C,\mu)$, but by $C^p(X)$ we will always mean $C^p(X,\C)$ in this context. Note that $C^0(X)=C(X)$.

Recall from Definition \ref{E5} that we say $\eta\in V^p(\mu)$ is strong when $\eta(f)=f$ for every $f\in C^p(X)$. We write
$V^p(\mathfrak{S},\mu)$ for the set of all strong $\eta\in V^p(\mu)$ and $(N_{C,\mu})$ instead of $(N_{C^p(X),\mathcal{N}(\mu)})$. (We have $\mathcal{N}(\mu)\cap C^p(X) = \mathcal{N}(\mu)\cap C(X)$ for any $p$, so the notation does not need to mention $p$.)
\[
\mathcal{N}(\mu)\cap C(X)=\{0\}.\leqno(N_{C,\mu})
\]
\end{df}

\begin{rem}\label{V20E}
Let $(X,\mathfrak{S},\vS,\mu)$ be a topological measure space.
\begin{enumerate}
\item
The property $[N_{\mathfrak{S},\vS_0}]\;\mathfrak{S}\cap\vS_0=\{\emptyset\}$ implies the condition $(N_{C,\mu})$.

\begin{proof}
For $h\in\mathcal{N}(\mu)\cap C(X)$, it follows $\{h\not=0\}\in\mathfrak{S}\cap\vS_0=\{\emptyset\}$, i.e. $h=0$.
\end{proof}

\item
If $(X,\mathfrak{S})$ is Tychonoff, then $(N_{C,\mu})$ implies $[N_{\mathfrak{S},\vS_0}]$.

\begin{proof}
For $G\in\mathfrak{S}\cap\vS_0$, if $G \not=\emptyset$ and $x_0\in G$ then by the Tychonoff property, we can choose $0\leq h\in C(X)$ with $0\leq h\leq 1$, $h(x_0)=1$ and $h(x)=0$ for every $x\in G^c$. It follows
$h^{-1}[(0,\infty)]\sq G\in\vS_0$ so $h\in\mathcal{N}(\mu)$ and by $(N_{C,\mu})$ get $h=0$, contradicting $h(x_0)=1$. Thus, $G=\emptyset$.
\end{proof}

\item
If $(N_{C,\mu})$ holds, then for $f\in C(X)$, $\sup\{|f(x)|:x\in X\}\leq \|f\|_\infty$. In particular, $\|f\|_\infty<\infty$ implies that $f$ is bounded.

\begin{proof}
We may assume that $f$ is real-valued and that $M:=\|f\|_\infty<\infty$. We have $(f-M)^+\in C(X)$ and $(f-M)^+ =_\mu 0$, so $(f-M)^+ = 0$ by $(N_{C,\mu})$. This gives $f\leq M$. Similarly, $f\geq -M$.
\end{proof}
\end{enumerate}
\end{rem}

\begin{lem}\label{fct}
Let $(X,\vS,\mu)$ be a measure space, $(X,\vS',\mu')$ an extension by null sets. Corresponding to $\eta\in V^p(\mu)$, we have $\eta'\in V^p(\mu')$ extending $\eta$ defined by $\eta'(f+g):=\eta(f)$
for $f\in\LL^p(\mu)$, $g\in\mathcal{N}(\mu')$.
If moreover $(X,\mathfrak{S},\vS,\mu)$ is a topological measure space and $\eta\in V^p(\mathfrak{S},\mu)$, then
$\eta'\in V^p(\mathfrak{S},\mu')$.
\end{lem}

\begin{proof}
For the sufficiency of defining $\eta'$ on elements $f+g$ as stated, see Corollary \ref{P10ii}. Let $f,f_1,f_2\in\LL^p(\mu)$ and  $g,g_1,g_2\in\mathcal{N}(\mu')$.

If $f_1+g_1=_{\mu'}f_2+g_2$ then $f_1-f_2=_{\mu'}g_2-g_1$. We have $\{f_1\not=f_2\}\in\vS$ and $\mu'\{f_1\not=f_2\}=0$, so $\{f_1\not=f_2\}\in\vS_0$ since $\mu'$ extends $\mu$. Thus, $\eta(f_1)=\eta(f_2)$, showing that $\eta'$ is well-defined and satisfies $(l2)$.
We also have $\eta'(f+g) = \eta(f)=_{\mu'}f
=_{\mu'}f+g$, showing that $\eta'$ satisfies $(l1)$.

For linearity, we have $\eta'((f_1+g_1)+(f_2+g_2)) = \eta'\bigl((f_1+f_2)+(g_1+g_2)\bigr) = \eta(f_1+f_2)=\eta(f_1)+\eta(f_2) = \eta'(f_1+g_1)+\eta'(f_2+g_2)$, since $g_1+g_2\in\mathcal{N}(\mu')$, and for $a\in{\K}$, $\eta'(a(f+g))=\eta'(af+ag)=\eta(af)=a\eta(f)=a\eta'(f+g)$, because $ag\in\mathcal{N}(\mu')$.
Also, $\eta'(1)=\eta(1)=1$, so $\eta'\in V^p(\mu')$.
In the last sentence, for $h\in C^p(X)$ we get $\eta'(h)=\eta(h)=h$.
\end{proof}

\begin{prop}\label{V30}
Given a topological measure space $(X,\mathfrak{S},\vS,\mu)$, we have that
$V^p(\mathfrak{S},\mu)\not=\emptyset$ if and only if
$(X,\mathfrak{S},\vS,\mu)$ satisfies the condition $(N_{C,\mu})$.
\end{prop}

\begin{proof} By Proposition \ref{AVP}, taking $X$ to be $\LL^p(\mu)$, $Y$ to be $\mathcal{N}(\mu)$, and $C$ to be $C^p(X)$. \end{proof}

\begin{rem}\label{V10R}
By Example \ref{ex:PV15:1} (a), there is a vector lifting for Lebesgue measure on $[0,1]$ whose values are Baire class 1 functions.
In contrast, it is consistent with ZFC that there are no linear liftings for Lebesgue measure on $[0,1]$, or for any power $[0,1]^\ka$, whose values are Borel functions (see \cite{bs}).
\end{rem}

\begin{rem}\label{CM119a}
\rm{Let $(X,\mathfrak{S},\vS,\mu)$ be a compact Radon probability space with $\supp(\mu)=X$ and $\Lambda_\mathfrak{S}(\mu)=\emptyset$ (see \cite{fr4} 453N). It follows by Remark \ref{V20E} (i), that $(X,\mathfrak{S},\vS,\mu)$ satisfies condition $(N^p_{C,\mu})$, implying along with Proposition \ref{V30} that $V^p(\mathfrak{S},\mu)\not=\emptyset$. So, it is, in general, impossible to convert  a given vector lifting $\eta\in{V}^\infty(\mu)$ into a lifting $\rho$ on $\LL^\infty(\mu)$ (in the sense of \cite{it} Chapter III) by any existing method.}
\end{rem}

For a topological probability space $(X,\mathfrak{T},\vS,\mu)$, we might try to get a multiplicative selector $\rho$ for the classes of nonnegative measurable functions $X\to[0,\infty]$ which is strong (i.e., fixes the continuous functions). On the family of continuous functions itself, we will have $\rho(fg) = fg = \rho(f)\rho(g)$ as long as $fg$ is continuous, but continuity of $fg$ can fail because of the discontinuity of multiplication on $[0,\infty]$. The next example shows that even for functions that are finite-valued except at one point, there might not be a multiplicative selector.

\begin{ex}\label{sf66}
Let $\lambda$ be Lebesgue measure on $X:=[-1/2,1/2]$ equipped with the usual topology. Fix $p\in (0,\infty)$. For $c = 1,2$, define $f,g_c\colon X\to[0,\infty]$ by $f(x): = |x|^{1/(2p)}$, $g_c(x):=c/f(x)$ $(x\not=0)$, $g_c(0):=\infty$. Note that $f$ is bounded and $\int |g_c|^p\,d\lambda<\infty$. Also, $f$ and $g_c$ are continuous and $fg_c$ is continuous except at $0$.

There is no function $\rho$ into $[0,\infty]^X$ defined on a family of functions including $f$, $g_c$ and $fg_c$ $(c = 1,2)$, such that $\rho(h)=h$ for $h$ continuous, $\rho(fg_c) = \rho(f)\rho(g_c)$ and the following weak version of $(l2)$ holds.
\begin{enumerate}
\item[$(l2)^*$]
$f_1=f_2$ except at one point implies $\rho(f_1)=\rho(f_2)$.
\end{enumerate}
If $\rho$ had all of the listed properties, then
$c=\rho(c) = \rho(fg_c) = \rho(f)\rho(g_c) = fg_c$. Evaluating at $0$ we get $c=0\cdot \infty$ for $c=1,2$, a contradiction no matter how we define $0\cdot\infty$.
\end{ex}

Removing the requirement that $\rho$ be strong, we have the following example which modifies an example of J. von Neumann defined originally on $X=\R$.

\begin{ex}\label{vN}
(Cf. J.\ von Neumann \cite{ne}.) Let $\lambda$ be Lebesgue measure on $X:=[-1/2,1/2]$. Suppose $\rho:\LL^0(\lambda)\to \LL^0(\lambda)$ (here $\LL^0(\lambda) = \LL^0(\R,\lambda)$) is linear and satisfies $(l1)$ and $(l2)^*$ (see Example \ref{sf66}). Define
\begin{align*}
f(x) & := x,\ \ e(x):=1\ \ (x\in X) \\
g_t(x) & := 1/(x-t)\ \ (x,t\in X,\,x\not=t),\ \ g_t(t):=0.
\end{align*}
A) For $\lambda$-almost all $t\in X$, either $\rho(fg_t)\not=\rho(f)\rho(g_t)$ or $\rho(eg_t)\not= \rho(e)\rho(g_t)$.%
\footnote{Of course $eg_t=g_t$ but we are emphasizing the failure of the multiplicative property.}

\smallskip

By $(l1)$, $S:=\{t\in X: \rho(f)(t)=t$ and $\rho(e)(t)=1\}$ has measure one. Fix $t\in S$.  Notice first that we have $f(x)g_t(x) - te(x)g_t(x) = (f(x)-te(x))g_t(x) = e(x)$ when $x\not=t$. By linearity of $\rho$ and $(l2)^*$, $\rho(fg_t) - t\rho(eg_t) = \rho(e)$. If we had $\rho(fg_t) = \rho(f)\rho(g_t)$ and $\rho(eg_t) = \rho(e)\rho(g_t)$ we would get $\rho(f)\rho(g_t) - t\rho(e)\rho(g_t) = \rho(e)$. Evaluating at $t$ and taking into account $t\in S$ gives a contradiction.

\smallskip

\n B) If we require also that $\rho(e)=e$, then by A), $\rho(fg_t)\not=\rho(f)\rho(g_t)$ for $\lambda$-almost all $t$.
\end{ex}

\begin{ex}\label{ve11}
Given a measure space $(X,\vS,\mu)$ with $\mu(X)>0$, let $\gamma:\LL^0(\mu)\to\LL^0(\mu)$ be an ${\R}$-vector lifting existing by Proposition \ref{p:vect.lift} and define
\[
\gamma_{\flat}(A):=\{\gamma(\chi_A)\geq 1\}\quad\mbox{for}\quad A\in\vS.
\]
$\gamma_{\flat}:\vS\to\vS$ satisfies $(Lj)$ for $j=1,2$, $(N)$, and $(E)$.

Since $\gamma(\chi_A)+\gamma(\chi_{A^c})=1$ for $x\in\gamma_{\flat}(A)$ it follows $\gamma(\chi_{A^c})(x)\leq 0$, so $x\in[\gamma_{\flat}(A^c)]^c$, i.e. $\gamma_{\flat}(A)\sq
[\gamma_{\flat}(A^c)]^c$, implying
$\gamma_{\flat}(A)\cap\gamma_{\flat}(A^c)=\emptyset$, i.e.
$\gamma_{\flat}$ satisfies $(V)$.
\end{ex}

\begin{rem}\label{ve12}
Write $(X,\LL,\lambda)$ for the Lebesgue measure space on $X:=[-1/2,1/2]$. Let
$\xi$ be an ${\R}$-vector lifting for $\LL^0(\lambda)$. There does not exist a $\gamma:\LL\to\LL$ satisfying $(N)$ and $(E)$, such that $\gamma_0=\xi$, where $\gamma_0(f)(x) = \sup\{r\in\Q: x\in\gamma\{r<f\}\}$.
(See \cite{bms6} for the basic properties of $\gamma_0$.)

\begin{proof}
Assume, that there exists a $\gamma:\LL\to\LL$ satisfying $(N)$ and $(E)$, such that $\gamma_0=\xi$ and let $E,F\in\LL$ satisfy $E\sq F$. If $G:=F\setminus E$ we get
$\xi(\chi_E)+\xi(\chi_G)=\xi(\chi_F)$ and by \cite{bms6}, Proposition 3.4 (ii) get $\chi_{\gamma(E)}+\chi_{\gamma(G)}=\chi_{\gamma(F)}$.

$x\in\gamma(E)\cap\gamma(G)$ would imply
$2=\chi_{\gamma(E)}(x)+\chi_{\gamma(G)}(x)=\chi_{\gamma(F)}(x) = 1$, contradiction, i.e. $\gamma(E)\cap\gamma(G)=\emptyset$, so
$\gamma(E)\sq\gamma(F)$, i.e. $\gamma$ satisfies $(M)$ and $\xi=\gamma_0$ satisfies $(m)$, that is, for every $f,g\in\LL^0(\lambda)$, $f\leq{g}$ implies $\xi(f)\leq\xi(g)$, i.e. $\xi$ is an order-preserving vector lifting for $\lambda$, contradicting Proposition \ref{p:o-p}.
\end{proof}
\end{rem}

\section{Product vector liftings}
\label{s:prod.lift}

\begin{df}\label{PVD1}
Given sets $X,Y$ and functions $f:X\to{\K}$ and $g:Y\to{\K}$, define
$(f\otimes{g})(x,y):=f(x)\cdot g(y)$ for $(x,y)\in X\times Y$.
\end{df}

For later reference, we make the following observation concerning the nature of $f\otimes g$ when $f$ and $g$ are measurable real- or complex-valued functions on probability spaces.

\begin{rem}\label{PVS}
Fix $p\in[0, \infty]$, $(X,\vS,\mu)$, $(Y,T,\nu)$ probability spaces, $f\in \mathbb{K}^X$, $g\in \mathbb{K}^Y$, $h:=f\otimes g$.
\begin{enumerate}
\item
If $f\in\LL^p(\mu)$ and $g\in \LL^p(\nu)$ then
$h \in\LL^p(\mu\otimes\nu)$.

\item
If $f_1,f_2\in\LL^p(\mu)$, $g_1,g_2\in \LL^p(\nu)$,
$f_1 =_\mu f_2$ and $g_1 =_\nu g_2$ then $f_1\otimes g_1 =_{\mu\otimes\nu} f_2\otimes g_2$.
\end{enumerate}
Let $(X\times Y,\vY,\ups)$ be a third probability space so that $[P_0]$, $[C]$ and $[\ov{C}]$ hold.
\begin{enumerate}
\setcounter{enumi}{2}
\item
If $h\in\LL^p(\ups)$ then either $f=_{\wh{\mu}} 0$, or $g=_{\wh{\nu}} 0$, or $f\in\LL^p(\mu)$ and $g\in \LL^p(\nu)$.

\item
If $h \in\LL^0(\ups)$ then $h=_\ups 0$ if and only if $f=_{\wh{\mu}} 0$ or $g=_{\wh{\nu}} 0$.
\end{enumerate}

If $\ups=\mu\otimes\nu$ then in (iii) we can write $f=0$, $g=0$ instead of $f=_{\wh{\mu}} 0$, $g=_{\wh{\nu}} 0$ using the proof below with the set $M$ taken to be empty.

\begin{proof}
(i) The measurability of $f\otimes g$ is clear since functions of one variable are also $\mu\otimes\nu$-measurable as functions of two variables, and products of measurable functions are measurable. The statement thus holds for $p=0$.
For $p=\infty$, we have $a,b>0$, $N\in\vS_0$ and $M\in T_0$ such that $|f(x)|\leq a$, $|g(y)|\leq b$ when $x\notin N$, $y\notin M$. Then $|f(x)g(y)|\leq ab$ when $(x,y)\notin N\times Y\cup X\times M$, so $f\otimes g\in \LL^\infty(X\times Y,\mu\otimes \nu)$.
For $0<p<\infty$, we have by the Tonelli Theorem, $\int |f(x)g(y)|^p\,d\mu(x)\,d\nu(y) = \int |f(x)|^p\,d\mu(x)\int |g(y)|^p\,d\nu(y) <\infty$, so $f\otimes g\in \LL^p(X\times Y,\mu\otimes \nu)$.

(ii) There exist $N\in\vS_0$ and $M\in T_0$ such that $f_1(x)=f_2(x)$ when $x\notin N$ and $g_1(y)=g_2(y)$ when $y\notin M$. Then $f_1(x)g_1(y)=f_2(x)g_2(y)$ when $(x,y)\notin N\times Y \cup X\times M$.

(iii) Suppose we have neither $f =_{\wh{\mu}} 0$ nor $g =_{\wh{\nu}} 0$. By $[C]$, there is an $M\in \vS_0$ such that $h_x\in \LL^p(\nu)$ when $x\notin M$. Choose $x\notin M$ such that $f(x)\not=0$. Then $g=(1/f(x))h_x\in \LL^p(\nu)$. Similarly, using $[\ov{C}]$ we get $f\in \LL^p(\mu)$.

(iv) $(\Rightarrow)$ We prove the contrapositive. Suppose we have neither $f =_{\wh{\mu}} 0$ nor $g =_{\wh{\nu}} 0$. By (iii), $f\in \LL^0(\mu)$ and $g\in \LL^0(\nu)$. Then the set $\{f\not=0\}\times \{g\not=0\}\in \vS\otimes T\sq \vY$ has positive $\mu\otimes\nu$-measure, and hence positive $\ups$-measure, so $h = f\otimes g =_\ups 0$ fails.

$(\Leftarrow)$ Say $f=_{\wh{\mu}} 0$. We have $f(x)=0$ when $x\in M^c$ for some $M\in\vS_0$. Then $h(x,y)=0$ when $(x,y)\notin M\times Y$. By $[P_0]$, $M\times Y\in\vY$, and by $[C]$, $\ups(M\times Y)=0$, so $h=_\ups 0$.
\end{proof}
\end{rem}

\begin{prop}\label{p:lin.comb.prod}
Let $h\colon X\times Y\to \mathbb{K}$.
Write
\[
\V=\sspan\{h_x:x\in X\},\ \ \Hh = \sspan\{h^y:y\in Y\},
\]
the spans of the collections of vertical and horizontal sections of $h$, respectively. The following statements are equivalent.
\begin{enumerate}[\rm(a)]
\item
$h = \sum_{i=1}^n f_i\otimes g_i$ for some $f_i\in \Hh$, $g_i\in\V$, $i=1,\dots,n$,

\smallskip

\item
$h = \sum_{i=1}^n f_i\otimes g_i$ for some $f_i\in \mathbb{K}^X$, $g_i\in \mathbb{K}^Y$, $i=1,\dots,n$,

\smallskip

\item
$\V$ is finite-dimensional.
\end{enumerate}
\end{prop}

Note that by the symmetry of (a), we could say in (c) that $\Hh$ is finite-dimensional. The equivalence of the two forms of (c) is basic linear algebra though since $\dim \Hh = \dim \V$ when either of these dimensions is finite. This is essentially the standard fact that the row-rank and column-rank of a matrix are equal. (The equality can fail when the dimensions are not finite.)

\begin{proof}
Clearly (a) implies (b), and if (b) holds then $\V$ is a subspace of $\sspan\{g_i:i=1,\dots,n\}$, so (c) holds.

There remains to show that (c) implies (a). Let $\{g_i:i=1,\dots,n\}$ be a basis for $\V$. Define functions $f_i$ on $X$ by writing for each $x$,
\[
h_x(\,\cdot\,) = \sum_i f_i(x) g_i(\,\cdot\,)
\]
Since the $g_i$ are linearly independent, we can choose $y_1,\dots,y_n$ such that the matrix with $(i,j)$-entry $g_i(y_j)$ is invertible.
Then the functions
\[
h^{y_j}(\,\cdot\,) = \sum_i f_i(\,\cdot\,) g_i(y_j),\ \ j=1,\dots,n
\]
all belong to $\Hh$, and hence so do the $f_i$ since we can solve for them as linear combinations of the $h^{y_j}$.
\end{proof}

\begin{cor}
Let $X$ and $Y$ be topological spaces and assume that $\mathbb{K}$ is equipped with a topology. Let $h\colon X\times Y\to \mathbb{K}$
be separately continuous. The following statements are equivalent.
\begin{enumerate}[\rm(a)]
\item
$h = \sum_{i=1}^n f_i\otimes g_i$ for some $f_i\in \mathbb{K}^X$, $g_i\in \mathbb{K}^Y$, $i=1,\dots,n$.

\smallskip

\item
$h = \sum_{i=1}^n f_i\otimes g_i$ for some $f_i\in C(X)$, $g_i\in C(Y)$.
\end{enumerate}
\end{cor}

Note that when $\mathbb{K}$ is a topological ring and (b) holds, $h$ is in fact jointly continuous since sums and products of continuous functions are then continuous.

\begin{proof}
To show that (a) implies (b), apply the proposition to $h$. The spaces $\V$ and $\Hh$ contains continuous functions since $h$ is separately continuous.
\end{proof}

\begin{prop}\label{p:lin.com.meas}
Suppose $\mathbb{K}=\R$ or $\C$ and $0\leq p\leq \infty$.
Let $(X,\vS,\mu)$ and $(Y,T,\nu)$ be probability spaces with a product space $(X\times Y,\vY,\upsilon)$ so that $[P_0]$, $[C]$ and $[\ov{C}]$ hold. Let $h\in \LL^p(\ups)$.
Write
\begin{align*}
\V & = \sspan\{h_x:h_x\in \LL^p(\nu),\,x\in X\} \\
\Hh & = \sspan\{h^y:h^y\in \LL^p(\mu),\,y\in Y\}
\end{align*}
The following statements are equivalent.
\begin{enumerate}[\rm(a)]
\item
$h =_\ups \sum_{i=1}^n f_i\otimes g_i$ for some $f_i\in \LL^p(\mu)$, $g_i\in \LL^p(\nu)$, $i=1,\dots,n$, $n\in\N$.

\smallskip

\item
$h =_\ups \sum_{i=1}^n f_i\otimes g_i$ for some $f_i\in \Hh$, $g_i\in\V$, $i=1,\dots,n$, $n\in\N$.
\end{enumerate}
\end{prop}

\begin{proof}
We need to show that (a) implies (b). Let $f_i$, $g_i$ be as in (a) and set $k:=\sum_{i=1}^n f_i\otimes g_i$. By Remark \ref{PVS} (i), $k\in \LL^p(\ups)$. By $[C]$, for some $N\in\vS_0$ we have $h_x,k_x\in \LL^p(\nu)$ and $h_x =_\nu k_x$ when $x\notin N$.
By $[\ov{C}]$, for some $M\in T_0$ we have that $h^y,k^y\in \LL^p(\mu)$ and $h^y =_\mu k^y$ when $y\notin M$.

Apply Proposition \ref{p:lin.comb.prod} to $k\restr S$, where $S:=(X\sm N)\times (Y\sm M)$, to conclude that we can write $k\restr S = \sum_{j=1}^m u_i\otimes v_j$ where
\begin{enumerate}[(1)]
\item
each $v_j$ is a linear combination of vertical sections of $k\restr S$ of the form
\[
k_x\restr (Y\sm M) =_\nu h_x\restr (Y\sm M),\ \ x\in X\sm N
\]
and hence $v_j =_\nu \bar{g}_j\restr (Y\sm M)$ for some $\bar{g}_j\in\V$,

\item
and similarly each $u_j =_\mu \bar{f}_j\restr (X\sm N)$ for some $\bar{f}_j\in \Hh$.
\end{enumerate}
Then (see Remark \ref{PVS} (ii))
\[
h\restr S = _\ups k\restr S = \sum_{j=1}^m u_i\otimes v_j =_\ups \sum_{j=1}^m (\bar{f}_i\otimes \bar{g}_j)\restr S
\]
Since $\ups(S)=1$ and the functions involved are measurable (Remark \ref{PVS} (i)), it follows that $h = _\ups \sum_{j=1}^m \bar{f}_j\otimes \bar{g}_j$.
\end{proof}

\begin{cor}\label{c:otimes}
Suppose $\mathbb{K}=\R$ or $\C$ and $0\leq p\leq \infty$. Let $(X,\mathfrak{S},\vS,\mu)$ and $(Y,\mathfrak{T},T,\nu)$ be topological probability spaces with a product space $(X\times Y,\mathfrak{S}\times \mathfrak{T},\vY,\upsilon)$ so that $[P_0]$, $[C]$ and $[\ov{C}]$ hold. Let $h\in \LL^p(\ups)$ be separately continuous.
The following statements are equivalent.
\begin{enumerate}[\rm(a)]
\item
$h =_\ups \sum_{i=1}^n f_i\otimes g_i$ for some $f_i\in \LL^p(\nu)$, $g_i\in \LL^p(\mu)$, $i=1,\dots,n$, $n\in\N$.

\smallskip

\item
$h =_\ups \sum_{i=1}^n f_i\otimes g_i$ for some $f_i\in C^p(X)$, $g_i\in C^p(Y)$, $i=1,\dots,n$, $n\in\N$.
\end{enumerate}
\end{cor}

The next propositions give an analog for sets of the results above.

\begin{prop}
Let $E\sq X\times Y$. Write
\[
\V=\{E_x:x\in X\},\ \ \Hh = \{E^y:y\in Y\},
\]
the collections of all vertical and horizontal sections of $E$, respectively. Write $\Hh^* = \{\bigcap\Hh':\Hh'\sq\Hh,\,\Hh'\ \text{\rm finite}\}$. The following statements are equivalent.
\begin{enumerate}[\rm(a)]
\item
$E = \bigcup_{i=1}^n A_i\times B_i$ for some $A_i\in \Hh^*$, $B_i\in\V$, $i=1,\dots,n$,

\smallskip

\item
$E = \bigcup_{i=1}^n A_i\times B_i$ for some $A_i\sq X$, $B_i\sq Y$, $i=1,\dots,n$,

\smallskip

\item
$\Hh$ and $\V$ are finite.
\end{enumerate}
\end{prop}

\begin{proof}
Clearly (a) implies (b), and (b) implies (c).

There remains to show that (c) implies (a). For this, assuming (c) it suffices to check (a) with the $A_i\times B_i$ ranging over all rectangles contained in $E$ with $A_i\in \Hh^*$, $B_i\in\V$. We verify that these rectangles cover $E$.

Consider a point $(p,q)\in E$. We have $\{p\}\times E_p\sq E$, and if $y\in E_p$ then $p\in E^y\in \Hh$. Thus, $H_0=\bigcap\{H\in \Hh: p\in H\}\in\Hh^*$ satisfies $H_0\sq E^y$, and hence $H_0\times\{y\}\sq E$, for all $y\in E_p$. Thus, $(p,q)\in H_0\times E_p\sq E$.
\end{proof}

\begin{cor}
Let $X$ and $Y$ be topological spaces. Let $U\sq X\times Y$ have open vertical and horizontal sections. The following statements are equivalent.
\begin{enumerate}[\rm(a)]
\item
$U = \bigcup_{i=1}^n A_i\times B_i$ for some $A_i\sq X$, $B_i\sq Y$, $i=1,\dots,n$.

\smallskip

\item
$U = \bigcup_{i=1}^n A_i\times B_i$ for some open sets $A_i\sq X$, $B_i\sq Y$, $i=1,\dots,n$.
\end{enumerate}
\end{cor}

Note that when (b) holds, $U$ is in fact open in $X\times Y$.

\begin{proof}
Follows from the proposition since the elements of $\V$ and $\Hh^*$ are open sets.
\end{proof}

We shall use the tensor product notation of the following definition.

\begin{df}\label{PVD2}
\begin{enumerate}
\item
Let $(X,\vS,\mu)$, $(Y,T,\nu)$, $(X\times Y,\vY,\ups)$ be probability spaces satisfying $[P_1]$. If $V$ and $W$ are linear subspaces of $\LL^p(\mu)$ and $\LL^p(\nu)$, respectively, we denote by $V\otimes_\ups W$, or just $V\otimes W$, the linear subspace of ${L}^p(\ups)$ generated by $\{(f\otimes{g})^{\bullet}:f\in V,g\in W\}$.

\item
For $\gamma\in\LL^p(\mu)\to{\K}^X$ and $\eta\in\LL^p(\nu_x)\to{\K}^Y$, let $\gamma\otimes_{p,\ups}\eta$, or just $\gamma\otimes\eta$, denote the set of maps $\pi:\LL^p(\ups)\to{\K}^{X\times Y}$ such that $\pi(f\otimes{g})=\gamma(f)\otimes\eta(g)$ for all $f\in\LL^p(\mu)$, $g\in \LL^p(\nu)$.
\end{enumerate}
\end{df}

\begin{lem}\label{PV10}
Let $(X,\vS,\mu)$, $(Y,T,\nu)$, $(X\times{Y},\vY,\ups)$ be probability spaces satisfying $[P_1]$.
Let $V$ and $W$ be linear subspaces of $\LL^p(\mu)$ and $\LL^p(\nu)$, respectively. If $\langle f_i^{\bullet}\rangle_{i\in I}$ is a basis for $V^\bullet$ and $\langle g_j^{\bullet}\rangle_{j\in J}$ a basis for $W^\bullet$, then $\langle (f_i\otimes{g}_j)^{\bullet}\rangle_{(i,j)\in I\times J}$ is a basis for the linear subspace $V\otimes W$ of ${L}^p(\ups)$.
\end{lem}

Recall that $V^\bullet$ denotes the image of $V$ in $L^p(\mu)$ under the quotient map $f\mapsto f^\bullet$ and similarly for the other measures.

\begin{rem}\label{r:tensor}
(a) When $\ups=\mu\otimes\nu$, $V=\LL^p(\mu)$, $W=\LL^p(\nu)$, and $p\geq 1$, the lemma describes the standard realization of a tensor product of the Banach spaces $\LL^p(\mu)$ and $\LL^p(\nu)$.

(b) Lemma \ref{PV10} shows that the bilinear map $V^\bullet\times W^\bullet \to V\otimes W$ given by $(f^\bullet,g^\bullet)\mapsto (f\otimes g)^\bullet$ is a tensor product of $V^\bullet$ and $W^\bullet$. (Cf.\ \cite{tr}, pages 403--404. The map $(f^\bullet,g^\bullet)\mapsto (f\otimes g)^\bullet$ is well-defined by Remark \ref{PVS} (ii) and $[P_1]$ and its bilinearity is clear.)
\end{rem}

\begin{proof}
We wish to show that $\langle (f_i\otimes{g}_j)^{\bullet}\rangle_{(i,j)\in I\times J}$ is a basis for $V\otimes W$.

For linear independence, suppose $\sum_{i,j}c_{i,j}(f_i\otimes{g}_j)^{\bullet}=0^{\bullet}$, where $c_{i,j}\in \mathbb{K}$, $\la c_{i,j}\ra$ of finite support.
The function $h:= \sum_{i,j}c_{i,j}(f_i\otimes g_j)$ is $\vS\otimes T$-measurable. Therefore, $M:=\{h\not=0\}\in\vS\otimes T$ and by $[P_1]$, $\mu\otimes\nu(M) = \ups(M)=0$. By the Fubini theorem there exists a set $Q\in T_0$, such that $M^y\in \vS_0$ for every $y\notin Q$.
When $y\notin Q$, set $a_i(y):=\sum_{j}c_{i,j}g_j(y)$. We have for $x\notin M^y\in \vS_0$,  $\sum_{i}f_i(x)\cdot a_i(y)=0$ implying $\sum_{i} a_i(y)f_i^{\bullet}=0^\bullet$.
Since $\langle f_i^{\bullet}\rangle$ is a basis for $V^\bullet$ we get $a_i(y)=0$ for each $i$ and $y\notin Q$, giving $\sum_{j}c_{i,j}g_j^{\bullet}=0^\bullet$ for every $i$ implying
$c_{i,j}=0$ for every $(i,j)$, since $\langle g_j^{\bullet}\rangle$ is a basis for $W^\bullet$. This shows that $\langle (f_i\otimes{g}_j)^{\bullet}\rangle$ is linear independent.

To see that $\langle (f_i\otimes{g}_j)^{\bullet}\rangle$ generates $V\otimes W$, given $f\in V$ and $g\in W$, write $f^{\bullet}=\sum_{i}a_if_i^{\bullet}$, where $a_i\in{\K}$, $\la a_i\ra$ of finite support, and write $g^{\bullet}=\sum_{j}b_jg_j^{\bullet}$, where $b_i\in{\K}$, $\la b_i\ra$ of finite support. There exists sets $N\in \vS_0$, $P\in T_0$ so that $f(x)=\sum_{i}a_if_i(x)$ when $x\notin N$, and $g(y) = \sum_{j}b_jg_j(y)$ when $y\notin P$. Then
$(f\otimes{g})(x,y)=f(x)g(y)=\sum_{i,j}a_ib_jf_i(x)g_j(y)=\sum_{i,j}a_ib_j(f_i\otimes{g}_j)(x,y)$
when $(x,y)\notin (N\times Y)\cup(X\times P)$ and this set belongs to $\vY_0$ by $[P_1]$, so
$(f\otimes{g})^{\bullet}=\sum_{i,j}a_ib_j(f_i\otimes{g}_j)^{\bullet}$. Hence,
$\langle (f_i\otimes{g}_j)^{\bullet}\rangle$ is a basis for the linear subspace $V\otimes{W}$.
\end{proof}

The next proposition is covered by \cite{mms15}, Remark 2.9, when $p=\infty$, when $\mu$ and $\nu$ are complete and $\ups$ is their completed product. Our proof adapts the proof given there to our present context. The proposition also follows directly from the upcoming Theorem \ref{PV30} by taking $\mathfrak{S}=\{\emptyset,X\}$ and $\mathfrak{T}=\{\emptyset,Y\}$. We present it anyway because its proof presents the idea for the proof of Theorem \ref{PV30} in a simper setting.

\begin{prop}\label{PV20}
Let $(X,\vS,\mu)$, $(Y,T,\nu)$, $(X\times{Y},\vY,\ups)$ be probability spaces satisfying $[P_0]$ and $[C]$.
For every
$\gamma\in V^p(\mu)$ and $\eta\in V^p(\nu)$ there exists a $\pi\in V^p(\ups)\cap(\gamma\otimes\eta)$ satisfying $[\pi(f)]_x\in\LL^p(\nu)$ for every $x\in X$.
If $[\ov{C}]$ also holds, then we may require also that $[\pi(f)]^y\in\LL^p(\mu)$ for every $y\in Y$.
\end{prop}

\begin{proof}
Choose linear bases
$\langle f_i^{\bullet}\rangle_{i\in I}$ for $L^p(\mu)$ and
$\langle g_j^{\bullet}\rangle_{j\in J}$ for $L^p(\nu)$ so that $\gamma(f_i)=f_i$ and $\eta(g_j)=g_j$ for all $i$ and $j$. (See Proposition \ref{p:exist.of.v.lift}.) By Lemma \ref{PV10} the subspace
$\LL^p(\mu)\otimes \LL^p(\nu)$ of $\LL(\ups)$
has the basis $\langle (f_i\otimes{g}_j)^{\bullet}\rangle_{(i,j)\in I\times J}$.

Extend this to a basis $\langle h_k^{\bullet}\rangle _{k\in K}$ for $L^p(\ups)$ with $I\times J\sq K$ and $h_{i,j} = f_i\otimes{g}_j$ for $(i,j)\in I\times J$, where for $k\in K\setminus I\times J$ by Corollary \ref{c:sec} we may choose $h_k$ in such a way, that
$[h_k]_x\in\LL^p(\nu)$ for every $x\in X$. If $[\ov{C}]$ holds we may also arrange
$[h_k]^y\in\LL^p(\mu)$ for every $y\in Y$. For all $h_{i,j} = f_i\otimes g_j$, $(i,j)\in I\times J$, these section properties holds automatically, since $[f_i\otimes g_j]_x=f_i(x)g_j$ and $[f_i\otimes g_j]^y=g_j(y)f_i$.

Let $\pi$ be the vector lifting for $\LL^p(\ups)$ such that $\pi(h_k)=h_k$ for all $k\in K$.
By Example \ref{ex:PV15:1} (b,c), for $h\in \LL^p(\ups)$ we have $[\pi(h)]_x\in\LL^p(\nu)$ for $x\in X$, and if $[\ov{C}]$ holds, $[\pi(h)]^y\in\LL^p(\mu)$ for $y\in Y$.

\begin{clm}\label{clm:pi.gamma.eta}
$\pi(f\otimes{g})=\gamma(f)\otimes\eta(g)$ for every $f\in\LL^p(\mu)$ and $g\in\LL^p(\nu)$.
\end{clm}

\begin{proof}
Write $f^{\bullet}=\sum_i a_if_i^{\bullet}$, $a_{i}\in{\K}$, $\la a_i\ra$ of finite support, and write $g^{\bullet}=\sum_j b_jg_j^{\bullet}$, $b_{j}\in{\K}$, $\la b_j\ra$ of finite support. Since
$(f\otimes{g})^{\bullet}=\sum_{i,j} a_ib_j(f_{i}\otimes g_{j})^{\bullet}$ for $(x,y)\in X\times Y$, we get
\begin{align*}
\ts \pi(f\otimes{g})(x,y) & = \ts \sum_{i,j}a_{i}b_{j}(f_{i}\otimes g_{j})(x,y) = \left(\sum_{i\phantom{j}}\!\! a_{i}f_{i}(x)\right)\left(\sum_jb_{j}g_{j}(y)\right) \\
& = \gamma(f)(x)\cdot\eta(g)(y)=(\gamma(f)\otimes\eta(g))(x,y),
\end{align*}
i.e., $\pi(f\otimes g)=\gamma(f)\otimes\eta(g)$.
\end{proof}

This completes the proof.
\end{proof}

\begin{lem}\label{PVR}
Let $(X,\mathfrak{S},\vS,\mu)$ and $(Y,\mathfrak{T},T,\nu)$ be topological probability spaces satisfying $(N_{C,\mu})$ and
$(N_{C,\nu})$, respectively. For every topological probability space
$(X\times Y,\mathfrak{S}\times\mathfrak{T},\Upsilon,\upsilon)$ satisfying $[C]$, we have $(N_{C,\upsilon})$.
\end{lem}

\begin{proof} From $h\in C(X\times Y)$, it follows that $h_x\in C(Y)$ for every $x\in X$ and $h^y\in C(X)$ for every $y\in Y$. If $h=_\ups0$ then by $[C]$ there is an $M\in\vS_0$ with $h_x=_\nu0$ for $x\notin M$. Then for all $x\notin M$, $h_x=0$ by $(N_{C,\nu})$. Thus, $h^y=_\mu 0$ for all $y\in Y$  and hence $h^y=0$ for all $y\in Y$ by $(N_{C,\mu})$. This gives $h=0$, so $(N_{C,\upsilon})$ is satisfied.
\end{proof}

\begin{thm}\label{PV30}
Let $(X,\mathfrak{S},\vS,\mu)$, $(Y,\mathfrak{T},T,\nu)$ and $(X\times{Y},\mathfrak{S}\times\mathfrak{T},\vY,\ups)$ be topological probability spaces satisfying $[P_0]$ and $[C]$. Assume that $(X,\mathfrak{S},\vS,\mu)$ and $(Y,\mathfrak{T},T,\nu)$ satisfy $(N_{C,\mu})$ and $(N_{C,\nu})$, respectively. Then
$(X\times Y,\mathfrak{S}\times\mathfrak{T},\vY,\ups)$ satisfies $(N_{C,\upsilon})$, and for every
$\gamma\in V^p(\mathfrak{S},\mu)$ and $\eta\in V^p(\mathfrak{T},\nu)$ there exists a
$\pi\in(\gamma\otimes\eta)\cap V^p(\mathfrak{S}\times\mathfrak{T},\ups)$ so that for $f\in\LL^p(\ups)$ we have $[\pi(f)]_x\in\LL^p(\nu)$ for all $x\in X$. If $[\ov{C}]$ also holds, then we may require also that $[\pi(f)]^y\in\LL^p(\mu)$ for every $y\in Y$.
\end{thm}

\begin{proof}
$(X\times Y,\mathfrak{S}\times\mathfrak{T},\Upsilon,\upsilon)$ satisfies $(N_{C,\upsilon})$ by  Lemma \ref{PVR}.

Note that $V^p(\mathfrak{S},\mu), V^p(\mathfrak{T},\nu)\not=\emptyset$ by Proposition \ref{V30}. By the assumptions $(N_{C,\mu})$ and $(N_{C,\nu})$, $C^p(X)$ and $C^p(Y)$ are linearly isomorphic to their images in $L^p(\mu)$ and $L^p(\nu)$, respectively. Choose a basis $\la f_i\ra_{i\in K}$ for $C^p(X)$ and a basis $\la g_j\ra_{j\in L}$ for $C^p(Y)$. Then $\gamma(f_i)=f_i$ and $\eta(g_j)=g_j$ for each $i\in K$, $j\in L$ since $\gamma$ and $\eta$ are strong.

Extend $\langle f_i^{\bullet}\rangle_{i\in K}$ to a basis $\langle f_i^{\bullet}\rangle_{i\in I}$ for $L^p(\mu)$, $K\sq I$, and extend $\langle g_j^{\bullet}\rangle_{j\in L}$ to a basis $\langle g_j^{\bullet}\rangle_{j\in J}$ for $L^p(\nu)$, $L\sq J$, choosing the representative functions so that $\gamma(f_i)=f_i$ and $\eta(g_j)=g_j$ also for $i\in I\sm K$, $j\in J\sm L$.

By Lemma \ref{PV10}, $\langle(f_k\otimes{g}_l)^\bullet\rangle_{(k,l)\in K\times L}$ is a linear basis for ${C}^p(X)\otimes C^p(Y)$.
Let $C: = C^p(X\times Y)^{\bullet}$. This is a linear subspace of $L^p(\ups)$ which by $(N_{C,\upsilon})$ is linearly isomorphic to $C^p(X\times Y)$ (the isomorphism being $h^\bullet\mapsto h$ for each $h\in C^p(X\times Y)$).

\begin{clm}\label{clm:1}
$C\cap(\LL^p(\mu) \otimes \LL^p(\nu)) = C^p(X)\otimes C^p(Y)$.
\end{clm}

\begin{proof}
The inclusion from right to left is clear. For the other direction, suppose $h\in C^p(X\times Y)$ and for some functions $f_i\in\LL^p(X)$ and $g_i\in\LL^p(Y)$, $i=1,\dots,n$, we have $h=_\ups \sum_{i=1}^n f_i\otimes g_i$. By Corollary \ref{c:otimes}, we can take $f_i\in C^p(X)$, $g_i\in C^p(Y)$, so $h^\bullet \in C^p(X)\otimes C^p(Y)$.
\end{proof}

Extend $\langle (f_k\otimes g_l)^\bullet\rangle_{(k,l)\in K\times L}$ to a basis $\langle h_r^\bullet\rangle_{r\in R}$ for $C$ with $K\times L\sq R$, each $h_r\in C^p(X\times Y)$, and $h_{k,l} = f_k\otimes{g}_l$ for every $(k,l)\in{K}\times{L}$. Then $[h_r]_x\in C^p(Y)\sq\LL^p(\nu)$ for $x\in X$ and $[h_r]^y\in C^p(X)\sq\LL^p(\mu)$ for $y\in Y$. Also $\langle h_r\rangle_{r\in R}$ is a basis for $C^p(X\times Y)$ since the latter is isomorphic to $C$.

Again by  Lemma \ref{PV10} it follows that $\langle(f_i\otimes{g}_j)^\bullet\rangle_{(i,j)\in{I}\times{J}}$ is a basis for $\LL^p(\mu)\otimes\LL^p(\nu)$.
Write $V$ for the linear subspace of $L^p(\upsilon)$ generated by $C$ and $\LL^p(\mu)\otimes\LL^p(\nu)$. By Claim \ref{clm:1}, $V$ has the basis $\langle(f_i\otimes{g}_j)^\bullet\rangle_{(i,j)\in I\times{J}}\cup \langle h_r^{\bullet}\rangle_{r\in R}$.
Extend this to a basis $\langle h_r^{\bullet}\rangle_{r\in S}$ for $L^p(\upsilon)$ with $(I\times J)\cup R\sq S$, where $h_{i,j}^\bullet = (f_i\otimes{g}_j)^{\bullet}$ for $(i,j)\in I\times J$. By Corollary \ref{c:sec}, we can choose the new representatives $h_r$ for $r\in S\sm ((I\times{J})\cup R)$ to satisfy $[h_r]_x\in\LL^p(\nu)$ for $x\in X$. If $[\ov{C}]$ holds we may also arrange
$[h_r]^y\in\LL^p(\mu)$ for every $y\in Y$. Together with our earlier choices, we get that these section properties now hold for every $r\in S$.

Let $\pi$ be the unique vector lifting for $\LL^p(\ups)$ such that $\pi(h_r)=h_r$ for each $r\in S$. By Example \ref{ex:PV15:1} (b,c), for $f\in\LL^p(\ups)$ we have $[\pi(f)]_x\in\LL^p(\nu)$ for $x\in X$, and if $[\ov{C}]$ holds, $[\pi(f)]^y\in\LL^p(\mu)$ for $y\in Y$.
By Example \ref{ex:PV15:1} (d), $\pi\in V(\mathfrak{S}\times\mathfrak{T},\upsilon)$ since by construction $\pi$ fixes each basis element of $C^p(X\times Y)$.
As in the proof of Claim \ref{clm:pi.gamma.eta}, we have $\pi(f\otimes{g})=\gamma(f)\otimes\eta(g)$ for every $f\in\LL^p(\mu)$ and $g\in\LL^p(\nu)$. This completes the proof of the theorem.
\end{proof}

\begin{cor}\label{PV29a}
Let $(X,\mathfrak{S},\vS,\mu)$ and $(Y,\mathfrak{T},T,\nu)$ be  $\tau$-additive topological probability spaces with $\nu$ complete, satisfying $(N_{C,\mu})$ and $(N_{C,\nu})$, respectively. Their $\tau$-additive product
$(X\times Y,\mathfrak{S}\times\mathfrak{T},\Upsilon,\upsilon)$ satisfies $(N_{C,\upsilon})$, and for every
$\gamma\in V^p(\mathfrak{S},\mu)$ and $\eta\in V^p(\mathfrak{T},\nu)$ there exists a
$\pi\in(\gamma\otimes\eta)\cap V^p(\mathfrak{S}\times\mathfrak{T},\upsilon)$ for every $f\in\LL^p(\upsilon)$ satisfying
$[\pi(f)]_x\in\LL^p(\nu)$ for $x\in X$. If $\mu$ is also complete, then we can require also that $[\pi(f)]^y\in\LL^p(\nu)$ for $y\in Y$.
\end{cor}

\begin{proof} The properties $[P_0]$ and $[C]$ (and $[\ov{C}]$ if $\mu$ is complete) for the spaces $(X,\mathfrak{S},\vS,\mu)$, $(Y,\mathfrak{T},T,\nu)$ and $(X\times{Y},\mathfrak{S}\times\mathfrak{T},\vY,\ups)$ are consequences of \cite{bms5}, Theorem  2.15, and so the desired result is immediate by Theorem \ref{PV30}
\end{proof}

\begin{ex}\label{PV29b}
Let $(Y,\mathfrak{T},T,\nu)$ be the hyperstonian space
(see \cite{fr3}, 321K, where it is called the Stone space) of a probability space, and let $\s$ be the canonical strong lifting for $\nu$ which selects the unique clopen set from each measure class. By Proposition \ref{V30} the space $(Y,\mathfrak{T},T,\nu)$ satisfies the property $(N_{C,\nu})$, and so by Corollary \ref{PV29a} the $\tau$-additive product $(Y\times{Y},\mathfrak{T}\times\mathfrak{T}, \vY,\ups)$ of $(Y,\mathfrak{T},T,\nu)$ by itself satisfies $(N_{C,\ups})$, and there exists a $\pi\in(\s\otimes\s)\cap{V}^\infty_\R(\mathfrak{T}\times\mathfrak{T},\ups)$ for every $f\in\LL^\infty(\R,\ups)$ satisfying $[\pi(f)]_x, [\pi(f)]^y\in\LL^\infty(\R,\nu)$ for all $x,y\in{Y}$. Could $\pi$ be converted into a strong (linear) lifting on $\LL^\infty(\R,\ups)$?
\end{ex}

\begin{que}\label{q:resp.coor}
Given a family of probability spaces $(X_i,\vS_i,\mu_i)$, $i\in I$, with product $(X,\vS,\mu)$, is there always a vector lifting for $\LL^\infty(\mu)$ which respects coordinates in the sense of \cite{fr3} 346A?
\end{que}

\section{Marginals}
\label{s:v.marg}

Extending Definitions 4.2 and 4.4 of \cite{bms6} (cf.\ Definitions 3.1 and 3.2 of \cite{mms14}) in a natural way, we make the following definitions.

\begin{df}\label{VMD}
Fix sets $X$ and $Y$. Also fix a map $\delta\colon U\to \mathcal{P}(Y)$, where $U\sq \mathcal{P}(Y)$, as well as a family of maps $\eta=(\eta_x)_{x\in X}$, where for each $x\in X$, $\eta_x:A_x \to\K^Y$ for some $A_x\sq \K^Y$.

(a) Define $\delta_{\bullet}(E)\sq X\times Y$ for $E\sq X\times Y$ by
\[
[\delta_{\bullet}(E)]_x = \de(E_x)\ \text{if $E_x\in U$},\ \text{and}\ [\delta_\bullet(E)]_x=\e\ \text{otherwise}.
\]
We call $\de$ a \emph{$2$-marginal} with respect to a family $\mathcal{E}\sq\mathcal{P}(X\times Y)$, if $\de_\bullet(E)\in\mathcal{E}$ for every $E\in\mathcal{E}$.

Interchanging the roles of the two coordinates, we similarly define $\eta^\bullet(E)\sq X\times Y$ for $\eta:V\to\mathcal{P}(X)$, where $V\sq\mathcal{P}(X)$, by setting $[\eta^\bullet(E)]^y = \eta(E^y)$ when $E^y\in V$, $[\eta^\bullet(E)]^y=\e$ otherwise. Call $\eta$ a \emph{$1$-marginal} with respect to $\mathcal{E}\sq\mathcal{P}(X\times Y)$ if $\de^\bullet(E)\in\mathcal{E}$ for every $E\in\mathcal{E}$.

(b) Define $\eta_\bullet(f):X\times{Y}\to\K$ for $f:X\times{Y}\to\K$ by
\[
\bigl(\eta_\bullet(f)\bigr)(x,y):=\eta_x(f_x)(y),\;\mbox{if}\; f_x\in A_x \;\mbox{and}\;\; \bigl(\eta_\bullet(f)\bigr)(x,y):=0\;\mbox{otherwise}.
\]
We call $\eta$ a \emph{$2$-marginal} with respect to a family $\mathcal{F}\sq \K^{X\times Y}$, if $\eta_\bullet(f)\in \mathcal{F}$ for each $f\in\mathcal{F}$.

(c) Similarly, given a single map $\eta: A\to\K^Y$, where $A\sq\K^Y$, define $\eta_\bullet(f)$ for $f:X\times{Y}\to\K$ by
\[
\bigl(\eta_{\bullet}(f)\bigr)(x,y):=\eta(f_x)(y),\;\mbox{if}\; f_x\in A\;\mbox{and}\;\; \bigl(\eta_{\bullet}(f)\bigr)(x,y):=0\;\mbox{otherwise}.
\]
In this latter setting, interchanging the roles of the two coordinates, we similarly define $\theta^\bullet(f):X\times{Y}\to\K$ for $\theta:B\to\K^X$, where $B\sq\K^X$, by setting $\theta^\bullet(f)(x,y) = \theta(f^y)(x)$ when $f^y\in B$, $=0$ otherwise. We call $\theta$ a \emph{$1$-marginal} with respect to $\mathcal{F}\sq \K^{X\times Y}$ if $\theta^\bullet(f)\in\mathcal{F}$ for each $f\in \mathcal{F}$.
\end{df}

The next proposition is a stronger version of \cite{bms6}, Remark 4.6 (e).

\begin{prop}\label{p:slc30d}
Let $p\in[0,\infty]$. Let $(X,\vS,\mu)$, $(Y,T_x,\nu_x)_{x\in X}$, and $(X\times{Y},\vY,\ups)$ be probability spaces satisfying $[C]$ and $\vS\,\dot{\times}\, Y\sq\vY$. Let $T=(T_x)_{x\in X}$, $\nu=(\nu_x)_{x\in X}$. Define the four families
\begin{align*}
A_1 & := \LL^p(\ups) \\
A_2 & := \LL^p(\ups,\nu) := \{f\in\LL^p(\ups): f_x\in\LL^p(\nu_x)\;\;\mbox{for all}\;\; x\in{X}\} \\
A_3 & := \LL^{0,p}_\mu(\vY,\nu) := \{f\in\LL^0(\vY): \text{for some}\ N\in\vS_0,\ x\in N^c\ \text{implies}\ f_x\in\LL^p(\nu_x)\} \\
A_4 & := \LL^{0,p}(\vY,\nu) := \{f\in\LL^0(\vY): f_x\in\LL^p(\nu_x)\;\;\mbox{for all}\;\; x\in{X}\}
\end{align*}
We also write $A_3$, $A_4$ as $\LL^{0,p}_\mu(\ups,\nu)$, $\LL^{0,p}(\ups,\nu)$, respectively.
Let $\rho=(\rho_x)_{x\in X}$, with $\rho_x\colon\LL^p(\nu_x)\to\LL^p(\nu_x)$ satisfying $(l1)$ and $(l2)$. Suppose that one of the following completeness assumptions holds.
\begin{enumerate}
\item[$(\al)$]
$(\vS_0\,\wh{\times}\, Y)[T]\sq \vY$.

\item[$(\beta)$]
$\rho_x(0)=0$ for $x\in X$, and $\mu$ is complete.
\end{enumerate}
Then all implications among the conditions $\rho_\bullet(A_i)\sq A_j$, $i,j=1,2,3,4$, hold, omitting, when $p\not=0$, those implications involving a condition where $i\in \{3,4\}$ and $j\in \{1,2\}$.
\end{prop}

\begin{rem}\label{r:slc30d}
(a) The omitted conditions in the last sentence all say that applying $\rho_\bullet$ to a function with $\LL^p$ sections, will create an $\LL^p$ function on the product, which of course will not hold in general when $p\not=0$.

(b) The conclusion gives in particular that the conditions ``$\rho$ is a $2$-marginal with respect to $A_i$'', $i=1,2,3,4$, are all equivalent.

(c) If $(\al)$ holds, then for any $N\in\vS_0$, the preimage of a Borel set under a function $\rho_\bullet(f)$ intersects $N\times Y$ in a member of $\vY$ since each $[\rho_\bullet(f)]_x$ is either $0$ or $\rho_x(f_x)$ and so is $T_x$-measurable.
\end{rem}

\begin{proof}
Write $H_{ij}$ for the statement $\rho_\bullet(A_i)\sq A_j$.

Since $A_2\sq A_1\sq A_3$ (the second by $[C]$) and $A_2\sq A_4\sq A_3$, in a hypothesis $H_{ij}: \rho_\bullet(A_i)\sq A_j$, we need consider only the case $i=2$, $j=3$, since that only weakens the hypothesis and this is not among the omitted conditions. When $p=0$, in a conclusion $H_{ij}$ we similarly need consider only the case $i=3$, $j=2$ since that only strengthens the conclusion. When $0<p\leq\infty$, avoiding the omitted conditions, we see that we need only consider in the conclusion the cases $(i,j) \in \{(1,2),(3,4)\}$.
Thus, we need only consider the following cases:
\begin{itemize}
\item
when $p=0$, $H_{23}\Rightarrow H_{32}$;

\item
when $p\not=0$, $H_{23}\Rightarrow H_{12}$ and $H_{23}\Rightarrow H_{34}$.
\end{itemize}

Observe that any function of the form $\rho_\bullet(f)$ belongs to $A_4$ as long as it belongs to $\LL^0(\vY)$ since for all $x\in X$, $[\rho_\bullet(f)]_x\in \LL^p(\nu_x)$ by definition of $\rho_\bullet$ and by our assumption that $\rho_x$ maps into $\LL^p(\nu_x)$.
In particular, $\rho_\bullet(f)\in A_3$ if and only if $\rho_\bullet(f)\in A_4$, and similarly $\rho_\bullet(f)\in A_1$ if and only if $\rho_\bullet(f)\in A_2$.
Taking these into account, we see that we need only consider the following cases:
\begin{itemize}
\item
when $p=0$, $H_{24}\Rightarrow H_{31}$;

\item
when $p\not=0$, $H_{24}\Rightarrow H_{11}$ and $H_{24}\Rightarrow H_{33}$.
\end{itemize}

\n Case 1. $p=0$.

\m

We need to show $H_{24}\Rightarrow H_{31}$.
Let $f\in A_3$, with $N$ as in the definition. Note that when $p=0$, $A_4\sq A_1$.
Define $g\restr N^c\times Y = f\restr N^c\times Y$, $g\restr N\times Y \equiv 0$.
Since $\vS_0\dot{\times}Y\sq\vY$,  $g$ is $\vY$-measurable, and by the choice of $N$, we have $g\in A_2$. Thus, by $H_{24}$, $\rho_\bullet(g)\in A_4$.
We have $\rho_\bullet(f)\restr N^c\times Y = \rho_\bullet(g)\restr N^c\times Y$, so $\rho_\bullet(f)$ is $\vY$-measurable on $N^c\times Y$.
If $(\al)$ holds, then it follows immediately that $\rho_\bullet(f)$ is $\vY$-measurable (see Remark \ref{r:slc30d} (c)) and therefore $\rho_\bullet(f)\in A_1$.
If $(\beta)$ holds, then we can take $N = \{x\in X:f_x\notin\LL^0(\nu)\}$ and using $\rho_x(0)=0$ for $x\in X$ we get $\rho_\bullet(f) = \rho_\bullet(g)\in A_4\sq A_1$.

\m

\n Case 2. $p=\infty$.

\m

$H_{24}\Rightarrow H_{11}$: Let $f\in A_1$.
Write $M:=\|f\|_\infty$. By $[C]$, there is an $N\in \vS_0$ such that for $x\in N^c$, $\|f_x\|_\infty\leq M$. Define $g$ as in Case 1. We get $g\in A_2$, so $\rho_\bullet(g)\in A_3$. As in Case 1, it follows that $\rho_\bullet(f)$ is $\vY$-measurable. Since each $\rho_x$ satisfies $(l1)$, we have $\|\rho_x(f_x)\|_\infty\leq M$ for $x\in N^c$. By $[C]$ and since $N\times Y\in\vY_0$, it follows that $\|\rho_\bullet(f)\|_\infty\leq M$, so $\rho_\bullet(f)\in A_1$.

$H_{24}\Rightarrow H_{33}$: Let $f\in A_3$, with $N$ as in the definition.
Define $g$ as in the argument for Case 1, getting that $g$ is $\vY$-measurable and $g_x\in \LL^\infty(\nu_x)$ for all $x\in X$.
For each $n\in\N$, write $g_n = (-n)\vee(g\wedge n)$. Then $g_n\in A_2$, so $\rho_\bullet(g_n)\in A_4$. Since for each $x$, $[g_n]_x =_{\nu_x} g_x$ if $n$ is large enough, and $\rho_x$ satisfies $(l2)$, for each $x$ we have $[\rho_\bullet(g)]_x=[\rho_\bullet(g_n)]_x\in\LL^\infty(\nu_x)$ when $n$ is large enough. It follows that $\rho_\bullet(g) = \lim_{n\to\infty}\rho_\bullet(g_n)$ is $\vY$-measurable. Then, as in Case 1, it follows that $\rho_\bullet(f)$ is $\vY$-measurable. By our observation, it then follows that $\rho_\bullet(f)\in A_4\sq A_3$.

\m

\n Case 3. $0<p<\infty$.

\m

$H_{24}\Rightarrow H_{11}$: Let $f\in A_1$. By $[C]$, there is an $N\in \vS_0$ such that for $x\in N^c$, $f_x\in\LL^p(\nu_x)$.
Define $g$ as in Case 1. We get $g\in A_2$, so $\rho_\bullet(g)\in A_4$. As in Case 1, it follows that $\rho_\bullet(f)$ is $\vY$-measurable. Since each $\rho_x$ satisfies $(l1)$, we have $\rho_x(f_x) =_{\nu_x} f_x$ for $x\in N^c$. By $[C]$, it follows that $\rho_\bullet(f)\in A_1$.

$H_{24}\Rightarrow H_{33}$: Let $f\in A_3$, with $N_0\in \vS_0$ as in the definition, i.e., $x\in N_0^c$ implies $f_x\in\LL^p(\nu_x)$.
By $[C]$, specifically Proposition \ref{p:C.for.fcts} (i) applied to $h=|f|^p$, we can enlarge $N_0$ to $N\in\vS_0$ to get that on $N^c$, the function $x\mapsto \int |f_x|^p\,d\nu_x$ is $\vS$-measurable.

Under $(\beta)$, we take specifically $N_0=\{x\in X: f_x\notin\LL^p(\nu_x)\}$. We do not need to enlarge $N_0$ since measurability on $N^c$ implies measurability on $N_0^c$ by completeness of $\mu$. So under $(\beta)$, we take $N=N_0=\{x\in X: f_x\notin\LL^p(\nu_x)\}$.

The function $x\mapsto \int |f_x|^p\,d\nu_x$ is finite-valued on $N^c$ by the choice of $N$.
For each $n\in\N$, let $E_n=\{x\in N^c: \int|f_x|^p\,d\nu_x \leq n\}\in\vS$ and define $g_n(x,y)=f(x,y)$ when $x\in E_n$, $g_n(x,y)=0$ when $x\notin E_n$. Since $E_n\times Y\in \vY$, $g_n$ is $\vY$-measurable. By $[C]$ applied to $|g_n|^p$, we have $g_n\in A_1$ and therefore $g_n\in A_2$. By $H_{24}$, we then get $\rho_\bullet (g_n) \in A_4$. For $x\in E_n$, we have $[g_n]_x=f_x$ and therefore $[\rho_\bullet(g_n)]_x = [\rho_\bullet(f)]_x$. The set $N^c$ is the union of the increasing sequence of the $E_n$. We therefore have $\rho_\bullet(f) = \lim_{n\to\infty}\rho_\bullet(g_n)$ on $N^c\times Y$, giving that $\rho_\bullet(f)$ is $\vY$-measurable on $N^c\times Y$. As in Case 1, it then follows that $\rho_\bullet(f)$ is $\vY$-measurable and hence belongs to $A_4\sq A_3$.
\end{proof}

\begin{lem}\label{VMLa}
Let $(X,\vS,\mu)$, $(Y,T,\nu)$ and $(X\times Y,\vY,\ups)$ be probability spaces satisfying $[P_2]$ and $[C]$. Let $(Y,T,\I)\in\T_\s$ with $\I\cap T = T_0$. Then for every
$f\in\LL^0(\upsilon_{\vS_0\ltimes \I})$ there exists a $g\in\LL^0(\mu\otimes\nu)$ such that $f=_{\vS_0\ltimes \I}g$.
\end{lem}

\begin{proof}
Property $[C]$ continues to hold if we extend $\nu$ to make the triple $(X,\vS,\mu)$, $(Y,T_\I,\nu_\I)$, $(X\times Y,\vY,\ups)$. Then by Proposition \ref{p:VMLa:2} (applied with $\vS_0$ and $(Y,T_\I,\nu_\I)$ in the place of $\I$ and $(Y,T_x,\nu_x)_{x\in X}$), $\ups_{\vS_0\ltimes \I}$ is defined.
By $[P_2]$ and Corollary \ref{c:ext.i}, $\ups = \mu\otimes_{\vY_0}\nu$.

Apply Lemma \ref{P10i} (ii) twice, first getting a $\vY$-measurable function $g_0\colon X\times Y\to\R$ satisfying $f =_{\vS_0\ltimes\I} g_0$ and then
getting a $\vS\otimes T$-measurable function $g\colon X\times Y\to\R$ satisfying $g_0 =_{\vY_0} g$.
Noting that $\vY_0\sq \vS_0\ltimes T_0\sq \vS_0\ltimes \I$, we get $f =_{\vS_0\ltimes\I} g$.
\end{proof}

\begin{cor}\label{c:VMLa}
Let $(X,\vS,\mu)$, $(Y,T,\nu)$ and $(X\times Y,\vY,\ups)$ be probability spaces satisfying $[P_2]$ and $[C]$. For every
$f\in\LL^0(\upsilon_{\mathfrak N})$ there exists a $g\in\LL^0(\mu\otimes\nu)$ with $f=_{\mathfrak N}g$.
\end{cor}

\begin{proof}
Apply the lemma with $\I$ being the completion of $T_0$.
\end{proof}

\begin{cor}\label{c:VMLb}
Denote by $(X\times Y,\mathfrak{S}\times{\mathfrak T},\Upsilon,\upsilon)$ the $\tau$-additive product of the $\tau$-additive topological probability spaces $(X,{\mathfrak S},\vS,\mu)$ and $(Y,\mathfrak{T},T,\nu)$.  For every
$f\in\LL^p(\upsilon_{\mathfrak N})$ there exists a $g\in\LL^p(\mu\otimes\nu)$ with $f=_{\mathfrak N}g$.
\end{cor}

\begin{proof}
By \cite{bms5}, Theorem 2.15 all assumptions of Corollary \ref{c:VMLa} are satisfied and so we may apply it to get the desired result.
\end{proof}

The next lemma is analogous to the fact that on complete measure spaces, functions equal almost everywhere to a measurable function are measurable. It provides a useful setting in which this works without completeness assumptions.

\begin{lem}\label{l:eta.skew}
Let $X$ be a set, $\I$ an upwards directed family of subsets of $X$. Let
$(Y,T_x)_{x\in X}$ be measurable spaces and for each $x\in X$, let $J_x$ be an ideal of $T_x$. Write $T=(T_x)_{x\in X}$, $J=(J_x)_{x\in X}$.
Let $\vY$ be a $\s$-algebra on $X\times Y$ such that $\vY\sq \I\ltimes T$.
Let $\Kk:=\I\ltimes J$.
\begin{enumerate}
\item
For each $f\in\LL^0(\vY_{\Kk})$, there is an $N\in\I$ so that $f_x\in\LL^0(T_x)$ when $x\notin N$.

\item
If $h\colon X\times Y\to \mathbb{K}$ satisfies that for some $N\in\I$ and for all $x\in N^c$ we have $h_x\in\LL^0(T_x)$ and $h_x =_{J_x} 0$ then $h\in\LL^0(\vY_{\Kk})$ and $h =_{\Kk} 0$.

\item
If $f\in\LL^0(\vY_{\Kk})$ and $g\colon X\times Y\to\mathbb{K}$ satisfies that for some $N\in\I$ and for all $x\in N$, $f_x$, $g_x\in\LL^0(T_x)$, and $g_x =_{J_x} f_x$, then $g\in\LL^0(\vY_{\Kk})$ and $g =_{\Kk} f$.
\end{enumerate}
\end{lem}

\begin{proof}
Since $\vY\sq \I\ltimes T$ and $\Kk$ is an ideal of $\I\ltimes T$, we have $(X\times Y,\vY,\Kk)\in\T_\s$ and thus $\vY_\Kk$ is defined.

(i) From $\vY\sq \I\ltimes T$ and $\Kk\sq \I\ltimes T$ we get $\vY_\Kk\sq\I\ltimes T$. Apply Proposition \ref{p:C.for.fcts:0} (i).

(ii) We may assume that $h$ is real-valued. We have $\{h<r\}\in \Kk$ when $r\leq 0$ since $[\{h<r\}]_x = \{h_x<r\}\in J_x$ for $x\in N^c$. When $r>0$, $\{h<r\}$ is the complement of $\{h\geq r\}$ which similarly belongs to $\Kk$.
Thus, $h\in\LL^0(\vY_{\Kk})$.
For $h =_{\Kk} 0$, we are given $\{h_x\not=0\}\in J_x$ when $x\in N^c$, so $\{h\not=0\}\in\Kk$ and hence $h =_{\Kk} 0$.

(iii) The function $h=g-f$ satisfies the assumptions of (ii), so $h\in\LL^0(\vY_{\Kk})$ and hence $g = f + h \in\LL^0(\vY_{\Kk})$, and since $h =_{\Kk} 0$ we get $g =_{\Kk} f$.
\end{proof}

It has been observed already in \cite{mms14}, Theorem 5.1, that given a complete probability space $(Y,T,\nu)$ not all liftings are $1$-marginals or $2$-marginals with respect the completed product $T\,\wh\otimes\,T$ of $T$ by itself.
In addition, under the assumption of CH, it has been shown in \cite{bms6}, Example 6.20, that  there are complete topological probability spaces $(X,\mathfrak{S},\vS,\mu)$ and $(Y,\mathfrak{T},T,\nu)$ and a lifting on the second space which is not 2-marginal with respect to the Radon product $\mu\otimes_R\nu$ or even with respect to the product $\mu\otimes_{\mathfrak{N}_2}\nu$.
In the case of vector liftings, we have the same phenomenon that this does not happen if $\mathfrak{N}_2$ is replaced by $\mathfrak{N}$, as shown in Examples \ref{ex:eta.skew.1} and \ref{ex:eta.skew:2} after the following lemmas.

\begin{lem}\label{l:eta.skew:m}
Let $X$ be a set, $\I$ an upwards directed family of subsets of $X$. Let
$(Y,T_x)_{x\in X}$ be measurable spaces and for each $x\in X$, let $J_x$ be an ideal of $T_x$. Write $T=(T_x)_{x\in X}$, $J=(J_x)_{x\in X}$.
Let $\vY$ be a $\s$-algebra on $X\times Y$ such that $\vY\sq \I\ltimes T$.
Let $\Kk:=\I\ltimes J$.
If $\eta=(\eta_x)_{x\in X}$, where each $\eta_x\colon\LL^0(T_x)\to \LL^0(T_x)$ satisfies $(l1)$, then $\eta_\bullet$ satisfies $(l1)$ on $\LL^0(\vY_{\Kk})$ and hence $\eta$ is a $2$-marginal with respect to $\LL^0(\vY_{\Kk})$.
\end{lem}

\begin{proof}
Use Lemma \ref{l:eta.skew}. By (i) of that lemma, for each $f\in\LL^0(\vY_{\Kk})$, the function $g=\eta_\bullet(f)$ satisfies the assumptions of (iii) of that lemma.
\end{proof}

\begin{lem}\label{l:eta.skew:n}
Let $(X,\vS,\mu)$, $(Y,T_x,\nu_x)$ for $x\in X$, and $(X\times{Y},\vY,\ups)$ be probability spaces. Write $T=(T_x)_{x\in X}$, $T_0=(T_{x,0}:x\in X)$, where $T_{x,0}$ is the null ideal of $\nu_x$, and $\nu=(\nu_x)_{x\in X}$. Let $\I\sq \vS_0$ be a $\s$-ideal of $\vS$, $\Kk:=\I\ltimes T_0$. Suppose that $\vY\sq \I\ltimes T$ and $\Kk\cap\vY\sq \vY_0$.

Assume that $[C]_\I$ holds and $\vS_0\,\dot{\times}\,Y\sq(\vY_\Kk)_0$. If $\eta=(\eta_x)_{x\in X}$, where each $\eta_x\colon\LL^p(\nu_x)\to \LL^p(\nu_x)$ satisfies $(l1)$, then
$\eta_\bullet$ satisfies $(l1)$ on $\LL^{0,p}_\mu(\vY_\Kk,\nu)$.
\end{lem}

\begin{rem}\label{r:null.ideals}
(a) Since $\vY\sq \I\ltimes T$ and $\Kk$ is an ideal of $\I\ltimes T$, we have $(X\times Y,\vY,\Kk)\in\T_\s$. Since we also have $\Kk\cap\vY\sq \vY_0$, $\ups_\Kk$ is defined.

(b) $\eta$ is a $2$-marginal with respect to both $\LL^{0,p}_\mu(\vY_\Kk,\nu)$ and the smaller family $\LL^{p}(\ups_\Kk)$.

(c) When $\I=\vS_0$,
$\vS_0\,\dot{\times}\,Y\sq (\vY_\Kk)_0$ holds automatically since $\vS_0\,\dot{\times}\,Y\sq \I\ltimes T_0 = \Kk\sq (\vY_\Kk)_0$.
\end{rem}

\begin{proof}
Case 1. $p=0$. This case follows from Lemma \ref{l:eta.skew:m} since $\LL^{0,0}_\mu(\vY_\Kk,\nu)\sq \LL^{0}(\vY_\Kk)$.

Case 2. $0<p<\infty$. Let $f\in \LL^{0,p}_\mu(\vY_\Kk,\nu)$. By Proposition \ref{p:VMLa:2}, $[C]_\I$ holds for $\ups_\Kk$. By Proposition \ref{p:C.for.fcts:0} (i) and Proposition \ref{p:C.for.fcts} (i), there is an $N\in\I$ such that $f_x\in \LL^0(\nu_x)$ for $x\in N^c$, and $x\mapsto \int |f_x|^p\,d\nu_x$ is $\vS$-measurable on $N^c$. Since $f\in \LL^{0,p}_\mu(\vY_\Kk,\nu)$, the set $A=\{x\in N^c:\int |f_x|^p\,d\nu_x=\infty\}$ belongs to $\vS_0$.
Define $f_1(x,y)=0$ when $x\in A$, $f_1(x,y)=f(x,y)$ otherwise.
Then $f_1$ is $\vY_\Kk$-measurable since $\vS_0\,\dot{\times}\,Y\sq \vY_\Kk$.
For $x\in A$, we have $[\eta_\bullet(f_1)]_x = 0 = (f_1)_x$ and $[\eta_\bullet(f)]_x = 0$. For $x\in N^c\sm A$, we have $[\eta_\bullet(f_1)]_x = \eta(f_x) =_{\nu_x} f_x = (f_1)_x$. Thus, $g=\eta_\bullet(f_1)$ satisfies the assumptions of Lemma \ref{l:eta.skew} (iii) with respect to $f_1$ and hence $\eta_\bullet(f_1)$ is $\vY_\Kk$-measurable. Since $\eta_\bullet(f)$ and $\eta_\bullet(f_1)$ agree on $N^c\times Y$, and since $\mathcal{P}(N\times Y)\sq \I\,\wh{\times}\,Y\sq \I\ltimes T_0\sq \vY_\Kk$, it follows that $\eta_\bullet(f)$ is $\vY_\Kk$-measurable as well. By assumption, $A\times Y\in (\vY_\Kk)_0$, so from $[C]$ it follows that $\eta_\bullet(f)=_{\ups_\Kk} f$.

Case 3. $p=\infty$. Proceed as in Case 2, replacing $\int|f_x|^p\,d\nu_x$ by $\|f_x\|_\infty$, using Proposition \ref{p:C.for.fcts} (ii) for the measurability of $x\mapsto \|f_x\|_\infty$.
\end{proof}

\begin{ex}\label{ex:eta.skew.1}
Let $(X,\vS,\mu)$, $(Y,T_x,\nu_x)_{x\in X}$, and $(X\times{Y},\vY,\ups)$ be probability spaces satisfying $[C]$. Write $\nu=(\nu_x)_{x\in X}$.
If $\eta=(\eta_x)_{x\in X}$, where each $\eta_x\colon\LL^p(\nu_x)\to \LL^p(\nu_x)$ satisfies $(l1)$,
then $\eta$ is a $2$-marginal with respect to $\LL^{0,p}_\mu(\vY_{\mathfrak{N}(\mu,\nu)},\nu)$ and hence also with respect to $\LL^{p}(\vY_{\mathfrak{N}(\mu,\nu)})$.

\begin{proof}
Use Lemma \ref{l:eta.skew:n} with $\I=\vS_0$. (See Remark \ref{r:null.ideals} (c).)
\end{proof}
\end{ex}

\begin{ex}\label{ex:eta.skew:2}
Let $(X,\vS,\mu)$, $(Y,T,\nu)$ be probability spaces and let $(X\times Y,\vS\otimes T,\mu\otimes \nu)$ be the product space.
If $\eta\colon\LL^p(\nu)\to \LL^p(\nu)$ satisfies $(l1)$,
then $\eta$ is a $2$-marginal with respect to $\LL^{0,p}_\mu(\vS\otimes_{\{\e\}\ltimes T_0} T,\nu)$ and hence also with respect to
$\LL^{p}(\mu\otimes_{\{\e\}\ltimes T_0}\nu)$.

\begin{proof}
Use Lemma \ref{l:eta.skew:n} with $\I=\{\e\}$.
\end{proof}
\end{ex}

\begin{lem}\label{VM10a}
Let $X$ be a set, $(Y,T_x,\nu_x)_{x\in X}$ probability spaces, $f\colon X\times Y\to \K$, $\eta=(\eta_x)_{x\in X}$ with $\eta_x:\LL^p(\nu_x)\rightarrow\K^Y$.
\begin{enumerate}
\item
For $x\in X$, $[\eta_{\bullet}(f)]_x$ equals $\eta_x(f_x)$ if $f_x\in\LL^p(\nu_x)$ and equals $0$ otherwise.

\item
If each $\eta_x$ is homogenous and $f=g\otimes h$ for some $g\in\K^X$ and $h\in \LL^p(\nu)$, then $\eta_{\bullet}(f)(x,y) = g(x)\eta_x(h)(y)$ for each $(x,y)\in X\times Y$. In particular, if we have a single space $(Y,T,\nu)$ and a single $\eta:\LL^p(\nu)\rightarrow\K^Y$, where $\eta$ is homogeneous, then $\eta_{\bullet}(f) = g\otimes\eta(h)$.%
\footnote{So in general, $\eta_{\bullet}$ does not satisfy $(l2)$.}

\item
If $\K=\R$, when $\eta_x(g)\geq{0}$ $(x\in X)$ for all nonnegative $g\in \LL^p(\nu)$, $f\geq 0$ implies $\eta_\bullet(f)\geq{0}$.

\item
If we also have $g\colon X\times Y\to\K$ and $f_x,g_x\in\LL^p(\nu_x)$ for every $x\in X$ then the following hold.
\begin{enumerate}[\rm(a)]
\item
If each $\eta_x$ is homogeneous then $\eta_{\bullet}(af) = a\eta_\bullet(f)$ $(a\in\K)$.

\item
If each $\eta_x$ is additive then $\eta_\bullet(f+g)=\eta_{\bullet}(f)+\eta_{\bullet}(g)$.

\item
If $p=0$ or $\infty$ and each $\eta_x$ is multiplicative then
$\eta_\bullet(fg)=\eta_{\bullet}(f)\eta_{\bullet}(g)$.

\item
If $\eta_x(f_x)=f_x$ for all $x\in X$ then
$\eta_\bullet(f)=f$.
\end{enumerate}
\end{enumerate}
\end{lem}

\begin{proof}
(i) is true by definition of $\eta_\bullet$.

(ii) We have $f_x = [g\otimes h]_x = g(x)h\in\LL^p(\nu)$ for every $x\in X$, so by homogeneity of $\eta_x$, $[\eta_{\bullet}(f)]_x = \eta(f_x) = g(x)\eta_x(h)$ giving $\eta_{\bullet}(f)(x,y) = g(x)\eta_x(h)(y)$.

(iii) holds by (i).

(iv) is clear from the definitions. In (c), that $p=0$ or $\infty$ ensures that each $\LL^p(\nu_x)$ is closed under multiplication.
\end{proof}

\begin{lem}\label{l:VM10a.2}
Let $(X,\vS,\mu)$, $(Y,T,\nu)$ and $(X\times Y,\vY,\ups)$ be probability spaces satisfying $[C]$, and let $\eta:\LL^p(\nu)\rightarrow\K^Y$.
If $\eta$ satisfies $(l2)$ then for $f,g\in\LL^p(\ups)$ with $f =_\ups g$ we have for all $y\in{Y}$ the equality $[\eta_\bullet(f)]^y=_{\wh{\mu}}[\eta_\bullet(g)]^y$.
\end{lem}

\begin{proof}
By $[C]$, there is a set $N\in\vS_0$ such that for $x\in N^c$ we have $f_x,g_x\in\LL^p(\nu)$ and $f_x=_\nu g_x$.
By $(l2)$ for $\eta$ we get $[\eta_\bullet(f)]_x=\eta(f_x)=\eta(g_x)=[\eta_\bullet(g)]_x$ for $x\in N^c$. Thus, $\{[\eta_\bullet(f)]^y\not=[\eta_\bullet(g)]^y\}\sq{N}$ for all $y\in Y$, implying $\{[\eta_\bullet(f)]^y\not=[\eta_\bullet(g)]^y\}\in\wh{\vS}_0$ for all $y\in{Y}$.
\end{proof}

\begin{ex}\label{VM10b}
Denote by $(X\times Y,\mathfrak{S}\times{\mathfrak T},\Upsilon,\upsilon)$ the $\tau$-additive product of the $\tau$-additive topological probability spaces $(X,{\mathfrak S},\vS,\mu)$ and $(Y,\mathfrak{T},T,\nu)$, where $\nu$ is complete.  If $\eta\in{V}^p(\nu)$ then Lemma \ref{l:VM10a.2} applies to $\eta$, and in Lemma \ref{VM10a}, clauses (ii) and (iv)(a,b,d) apply to $\eta$, and, when $\K=\R$, so does (iii) if additionally $\eta\in\mathcal{G}(\nu)$, i.e., $\eta$ is a linear lifting, and so does (iv)(c) if $\eta\in\vL^\infty(\nu)$.
(See \cite{bms5}, Theorem 2.15, for the hypotheses of Lemma \ref{l:VM10a.2}.)
\end{ex}

The next Lemma extends \cite{mms14}, Proposition 3.1 as well as \cite{bms4}, Lemma 4.4. Compare \cite{bms6}, Lemma 4.7.

\begin{lem}\label{VM19}
Let $(X,\vS,\mu)$, $(Y,T_x,\nu_x)_{x\in X}$, $(X\times{Y},\vY,\ups)$ be probability spaces so that $[C]$ holds. Let $\nu=(\nu_x)_{x\in X}$ and $\eta=(\eta_x)_{x\in X}$, where $\eta_x\colon \LL^p(\nu_x) \to \LL^p(\nu_x)$. Fix a linear subspace $\M\sq\LL^{0,p}_\mu(\vY,\nu)$.
\begin{enumerate}
\item
If $\eta$ is a $2$-marginal for $\M$ and there is an $N\in\vS_0$ such that $x\notin N$ implies $\eta_x$ satisfies $(l1)$ on $\LL^{p}(\nu_x)$, then $\eta_\bullet$ satisfies $(l1)$ on $\M$.

\item
If for each $x$, $\eta_x(0)=0$ and $(l1)$, $(l2)$ hold for $\eta_x$, then $\eta_x([\eta_\bullet(f)]_{x})=[\eta_\bullet(f)]_{x}$ for any function $f$ on $X\times Y$.

\item
If $\eta$ is a $2$-marginal for $\M$, for each $x$, $\eta_x\in V^p(\nu_x)$, and $\vS_0\,\dot{\times}\, Y\sq\vY$, then there is a $\psi\in V(\M)$ such that $\eta_x([\psi(f)]_{x}) = [\psi(f)]_{x}$ for all $f\in\M$, $x\in X$.

\item
Suppose $\K=\R$ and $p=\infty$. If $\eta$ is a $2$-marginal for $\M$,  $\eta_x\in \vL^\infty(\nu_x)$ for each $x$, $\vS_0\,\dot{\times}\, Y\sq\vY$, and there exists a $\theta \in \vL^\infty(\ups\restr\vY[T])$ then there is a $\psi\in\Lambda^\infty(\ups)$ such that $\eta_x([\psi(f)]_{x}) = [\psi(f)]_{x}$ for all $f\in\LL^{\infty}(\ups)$, $x\in X$.
\end{enumerate}
\end{lem}

\begin{rem}\label{r:Mo}
Under CH, by a theorem of Mokobodzki (cf.\ \cite{bms6}, Theorem 2.7), the assumption in (iv) that there is a lifting for $\ups\restr \vY[T]$ holds if the measure algebra of $\ups$ has cardinality at most $\aleph_2$.
\end{rem}

\begin{proof}
(i) Let $f\in \M$. Then $f\in\LL^{0,p}(\vY,\nu)$, so there is an $N\in\vS_0$ such that $x\notin N$ implies $f_x\in\LL^p(\nu_x)$ and $[\eta_\bullet(f)]_x = \eta_x(f_x)=_{\nu_x}f_x$. By $[C]$, $\eta_\bullet(f)=_\ups f$.

(ii) If $f_x\notin\LL^p(\nu_x)$, then by definition $[\eta_\bullet(f)]_x=0$ and the conclusion is clear.
If $f_x\in\LL^p(\nu_x)$, then by $(l1)$, $\eta_x(f_x)=_{\nu_x}f_x$ and then by $(l2)$, $\eta_x(\eta_x(f_x)) = \eta_x(f_x)$. The conclusion follows since by definition $[\eta_\bullet(f)]_x=\eta_x(f_x)$.

(iii) If $\eta_x\in V^p(\nu_x)$, take any $\theta\in V(\M)$ whose values are in $\LL^{0,p}(\vY,\nu)$ (cf.\ Example \ref{ex:PV15:1} (b)). Define
$\psi(f):=\eta_\bullet(\theta(f))$ for each $f\in\M$.
Since $\theta$ and $\eta_\bullet$ both satisfy $(l1)$, so does $\psi$, and
$\psi$ satisfies $(l2)$ since $\theta$ does.
Linearity of $\psi$ follows from linearity of $\theta$ and the $\eta_x$, and the fact that $\theta$ takes values in $\LL^{0,p}(\vY,\nu)$ (see Lemma \ref{VM10a} (iv) (a,b)). Since $\theta(1)=1$ and $\eta_\bullet(1)=1$, we get $\psi(1)=1$, so $\psi\in V(\M)$. Finally, for all $f\in\M$ and $x\in{X}$ we have $\eta_x([\psi(f)]_x) = [\psi(f)]_x$ by (ii).

(iv) If $\K=\R$, $p=\infty$, $\eta_x\in \vL^\infty(\nu_x)$ for each $x$, and we fix $\theta\in\Lambda^\infty(\ups\restr \vY[T])$, then by Corollary \ref{c:sec}, the inclusion $\LL^\infty(\ups\restr \vY[T])\to \LL^\infty(\ups)$ induces a bijection $L^\infty(\ups\restr \vY[T])\to L^\infty(\ups)$ which is clearly an algebra isomorphism, so it suffices to define $\psi$ on $\LL^\infty(\ups\restr \vY[T])$. Notice that $[\theta(f)]_x\in\LL^\infty(\nu_x)$ when $x\in X$, $f\in \LL^\infty(\ups\restr\vY[T])$ since $[\theta(f)]_x\in\LL^0(\nu_x)$ by Corollary \ref{c:C.for.fcts:0}, and $|\theta(f)|\leq\|f\|_\infty$ (\cite{it}, (1) page 35). Thus, $\theta$ takes values in $\LL^{0,\infty}(\vY,\nu)$.
Take $\psi(f) := \eta_\bullet(\theta(f))$, $f\in\LL^\infty(\ups)$, and proceed as in (iii) (using also Lemma \ref{VM10a} (iv)(c)).
\end{proof}

There is a converse of sorts to Lemma \ref{VM19} (iii).

\begin{lem}\label{VM19a}
Let $(X,\vS,\mu)$, $(Y,T_x,\nu_x)_{x\in X}$, $(X\times{Y},\vY,\ups)$ be probability spaces satisfying $[C]$. Let $\nu=(\nu_x)_{x\in X}$ and $\eta=(\eta_x)_{x\in X}$, where $\eta_x\colon \LL^p(\nu_x) \to \LL^p(\nu_x)$ satisfies $(l2)$ for each $x$.
Let $\M\sq\LL^{0,p}_\mu(\vY,\nu)$ satisfy $\M+\NN(\mu)=\M$.
Suppose $(\vS_0\,\wh{\times}\, Y)[T]\sq\vY$ and there exists a function $\psi\colon \M\rightarrow\M$ satisfying $(l1)$ so that
\[
\eta_x([\psi(f)]_{x})=[\psi(f)]_{x},\ \text{for all $f\in\M$ with $x\notin N_f$, for some $N_f\in\vS_0$.}
\]
\begin{enumerate}
\item
Then $\eta$ is a $2$-marginal with respect to $\M$ and $\eta_\bullet$ satisfies $(l1)$ on $\M$.

\item
When we have a single space $(Y,T,\nu)$ and a single map $\eta\colon \LL^p(\nu) \to \LL^p(\nu)$, if $1\otimes g\in\M$ for all $g\in\LL^p(\nu)$ then $\eta$ satisfies $(l1)$.
\end{enumerate}
\end{lem}

\begin{rem}
When $\M=\LL^p(\ups)$, the condition $X\,\dot{\times}\,T\sq\vY$ is sufficient to ensure that $1\otimes g\in\M$ for all $g\in\LL^p(\nu)$.
\end{rem}

\begin{rem}\label{r:VM19a}
Take $\K=\R$ and suppose $(X,\vS,\mu)$, $(Y,T,\nu_x)_{x\in X}$, $(X\times Y,\vS\otimes T,\ups)$ are probability spaces so that $(\nu_x)_{x\in X}$ is a product r.c.p.\ on $T$ for $\ups$ with respect to $\mu$.
According to \cite{mu23} Theorem 3.11, for every equi-admissible family $(\tau_x)_{x\in{X}}$ of densities $\tau_x\in\vartheta(\nu_x)$ (see \cite{mu23} for the definition) there exist a family $\eta=(\eta_x)_{x\in{X}}$ of liftings $\eta_x\in\vL^\infty(\wh\nu_x)$ admissibly generated by $\tau_x$ (see \cite{mu23} for the definition), and a lifting $\varphi:\LL^\infty(\ups_\mathfrak{N})\to\LL^\infty(\wh{\ups})$ such that
\[
[\varphi(f)]_x=\eta_x\bigl([\varphi(f)]_x\bigr)\;\;\mbox{for every}\;\; f\in\LL^\infty(\ups_\mathfrak{N})\;\;\mbox{and}\;\; x\in{X}.
\]
By Lemma \ref{VM19a} applied to $(X,\vS,\mu)$, $(Y,T,\wh{\nu}_x)_{x\in X}$, $(X\times Y,\vS\,\wh{\otimes}_\ups\, T,\wh{\ups})$, where $\vS\,\wh{\otimes}_\ups\, T$ is the domain of $\wh{\ups}$, $\eta$ is a 2-marginal with respect to $\LL^\infty(\wh{\ups})$.
\end{rem}

\begin{proof}
(i) Fix $f\in\M$. Since $\psi(f)=_{\ups}f$, using $[C]$ and enlarging the given set $N_f\in\vS_0$ if necessary, we can assume that $[\psi(f)]_x,\,f_x\in\LL^p(\nu_x)$ and $[\psi(f)]_x=_{\nu_x}f_x$ for every $x\notin N_f$. For $x\notin N_f$ we have, using $(l2)$ for $\eta_x$,
\[
[\psi(f)]_x = \eta_x([\psi(f)]_x)=\eta_x(f_x)=[\eta_{\bullet}(f)]_x,
\]
so $\psi(f)\restr(N_f^c\times Y)=\eta_{\bullet}(f)\restr(N_f^c\times Y)$. Also $\eta_{\bullet}(f)\restr N_f\times Y$ is $\vY$-measurable since $(\vS_0\,\wh{\times}\, Y)[T]\sq\vY_0$. Since $\psi(f)\in\M$ and $\M+\NN(\mu)=\M$, we have $\eta_{\bullet}(f)\in\M$ and $\eta_\bullet(f) =_\ups \psi(f) =_\ups f$, so $\eta_\bullet$ satisfies $(l1)$.

(ii) Now assume that we have a single space $(Y,T,\nu)$ and a single map $\eta\colon \LL^p(\nu) \to \LL^p(\nu)$, and $1\otimes g\in\M$ for all $g\in\LL^p(\nu)$.
To see that $\eta$ satisfies $(l1)$, fix $g\in\LL^p(\nu)$ and let $f=1\otimes g$. Note that
$\{\psi(1\otimes{g})\not=1\otimes{g}\}\in\vY_0\sq\vS_0\ltimes T_0$. Choose $x\notin N_f$ such that $\{[\psi(1\otimes{g})]_x\not=[1\otimes{g}]_x = g\}\in T_0$. Since $x\notin N_f$, as in (i) we have $[\psi(1\otimes{g})]_x = [\eta_\bullet(1\otimes{g})]_x = \eta(g)$, yielding $\{\eta(g)\not=g\}\in T_0$, or $\eta(g) =_{\nu} g$ and hence $\eta$ satisfies $(l1)$.
\end{proof}

The next example shows that, assuming sets of reals of cardinality less than $\mathfrak{c}$ have measure zero, the restriction to having a single $\eta$ cannot be removed in (ii).

\begin{ex}\label{ex:VM19}
Let each of $(X,\vS,\mu)$ and $(Y,T,\nu)$ be $[0,1]$ with Lebesgue measure, and let $F^\bullet$ denote the class of $F\in T$ in the measure algebra of $\nu$. Let $(X\times Y,\vY,\ups)$ be the Lebesgue measure space on $[0,1]^2$. There are $\eta_x\colon \LL^0(\nu) \to \LL^0(\nu)$, $x\in X$, satisfying $(l2)$ so that $\eta=(\eta_x)_{x\in X}$ is a $2$-marginal for $\LL^0(\ups)$, and $\eta_\bullet$ satisfies $(l1)$ on $\LL^0(\ups)$, but none of the maps $\eta_x$ satisfy $(l1)$.

Let $E_x$, $x\in X$ be non-null Borel sets so that for each Borel set $A\sq X\times Y$, $\{x\in X:(E_x)^\bullet = (A_x)^\bullet\}$ has cardinality less than $\mathfrak{c}$.%
\footnote{
List $[0,1]$ bijectively as $\{x_\al:\al<\mathfrak{c}\}$. List the Borel subsets of $X\times Y$ as $\{B_\al:\al<\mathfrak{c}\}$. Recursively define $E_{x_\al} = [0,t_{\al}]$ so that at each stage, $(E_{x_\al})^\bullet$ is different from any of the $<\mathfrak{c}$ classes $([B_\beta]_{x_\al})^\bullet$, $\beta<\al$.}

\begin{clm}
For any $f\in \LL^0(\ups)$, there is an $N\in\vS_0$ such that $\{x\in N^c:f_x\in\LL^0(T),\,f_x=_{\nu}\chi_{E_x}\}$ has cardinality $<\mathfrak{c}$.
\end{clm}

\begin{proof}
Let $g\in \LL^0(\ups)$ be Borel so that $f =_\ups g$. For some set $N\in\vS_0$, $x\in N^c$ implies $f_x\in\LL^0(\nu)$, $f_x =_{\nu} g_x$. Let $A=g^{-1}(1)$. Then $S=\{x\in X:(E_x)^\bullet = (A_x)^\bullet\}$ has cardinality less than $\mathfrak{c}$. If $x\notin N$ and $f_x =_{\nu}\chi_{E_x}$, then $g_x =_{\nu}\chi_{E_x}$ so $(A_x)^\bullet = (E_x)^\bullet$ giving $x\in S$.
\end{proof}

Suppose that sets of size $<\mathfrak{c}$ have measure zero, which is true for example if CH holds.
Fix a $\de\colon\LL^0(T)\to \LL^0(T)$ which is a $2$-marginal for $\LL^0(\ups)$ (for example, see Proposition 7.24 along with \cite{bms4} Example 4.8) and let $\eta_x\colon \LL^0(\nu) \to \LL^0(\nu)$ agree with $\de$ except on functions $h=_{\nu}\chi_{E_x}$, setting $\eta_x(h)=0$ for such functions. If $f\in \LL^0(\ups)$, then the set
\[
N_f := \{x\in X:f_x\notin\LL^0(\nu),\ \text{or}\ f_x\in\LL^0(\nu)\ \text{and}\ f_x=_{\nu}\chi_{E_x}\}
\]
has measure zero by the claim.
Since $\eta_\bullet(f)$ and $\de_\bullet(f)$ agree on $N_f^c\times Y$, $\eta_\bullet(f)\in\LL^0(\ups)$ and $\eta_\bullet(f) =_\ups \de_\bullet(f) =_\ups f$, so $(l1)$ holds for $\eta_\bullet$. It follows that each $\eta_x$ satisfies $(l2)$ and $\eta$ is a $2$-marginal for $\LL^0(\ups)$ but no $\eta_x$ satisfies $(l1)$ since $\eta_x(\chi_{E_x})=0$.
\end{ex}

The following example extends that included in Remark 7.20. 

\begin{ex}\label{disex}
Let $(X,\vS)$, $(Y,\wt{T})$ be measurable spaces. Let $\ups$ be a probability measure on $\vS\otimes\wt{T}$, and $\mu$ the marginal measure of $\ups$ on $\vS$. Suppose that the marginal measure $\wt\nu$ of $\ups$ on $\wt{T}$ is inner regular with respect to a countably compact class $\mathcal{K}\sq\wt{T}$ (see \cite{fr4} 451A for definition) which is closed under finite unions and countable intersections. By \cite{pa} Theorem 3.5 (see also \cite{fr4} Theorem 452M) there exists a family $(Y,T_x,\nu_x)_{x\in{X}}$ of complete probability spaces satisfying the property $[C]$ for $\vY=\vS\otimes\wt{T}$ such that for every $x\in{X}$ the inclusion $\mathcal{K}\sq{T}_x$ holds and $\nu_x$ is inner regular with respect to $\mathcal{K}$. So, we may apply \cite{mu23} Theorem 3.8, in order to find a family
$(\eta_x)_{x\in{X}}$ of liftings $\eta_x\in\vL^\infty(\nu_x)$ and a lifting $\varphi\in\vL^\infty(\wh\ups)$ such that
\[
[\varphi(f)]_x=\eta_x([\varphi(f)]_x)\quad\mbox{for every}\quad x\in{X}\quad\mbox{and}\quad f\in\mathcal{L}^\infty(\vS\wh\otimes_{\ups}\wt{T}).
\]
By Lemma 7.18 applied to the probability spaces $(X,\vS,\mu), (Y,T_x, \nu_x)_{x\in{X}}$, and $(X\times{Y}, \vS\,\wh\otimes_{\ups}\,\wt{T},\wh{\ups})$, the family $\eta:=\langle\eta_x\rangle_{x\in{X}}$ is a 2-marginal with respect to $\mathcal{L}^\infty(\wh{\ups})$.\\
Note that the existence of a family $(\nu_x)_{x\in{X}}$ of probability measures $\nu_x$ as above with domain $\wt{T}$ is guaranteed, if either $\wt{T}$ is the Baire $\s$-algebra with respect to a compact Hausdorf topology on $Y$, or $\wt{T}$ is the Borel $\s$-algebra with respect to an analytic Hausdorff topology on $Y$, or $(Y,\wt{T})$ is a standard Borel space (see \cite{fr4} 452N).
\end{ex}

From Lemmas \ref{VM19} and \ref{VM19a}, we easily get the following corollary.

\begin{cor}\label{VM19b}
Let $(X,\vS,\mu)$, $(Y,T_x,\nu_x)_{x\in X}$, $(X\times{Y},\vY,\ups)$ be probability spaces satisfying $[C]$. Suppose that $(\vS_0\,\wh{\times}\, Y)[T]\sq\vY$.
Let $\eta=(\eta_x)_{x\in X}$ with each $\eta_x\colon \LL^p(\nu_x) \to \LL^p(\nu_x)$ satisfying $\eta_x(0)=0$, $(l1)$, and $(l2)$. The following statements are equivalent.
\begin{enumerate}
\item
$\eta$ is $2$-marginal with respect to $\LL^p(\ups)$;
\item
$\eta_\bullet(f)=_{\ups}f$ for every $f\in\LL^p(\ups)$;
\item
there exists a map $\varphi: \LL^p(\ups)\rightarrow\LL^p(\ups)$ satisfying $(l1)$, such that $\eta_x([\varphi(f)]_{x})=[\varphi(f)]_{x}$ for all $f\in\LL^p(\ups)$ and $x\in X$;
\item
there exists a map $\varphi: \LL^p(\ups)\rightarrow\LL^p(\ups)$ satisfying $(l1)$, such that for each $f\in\LL^p(\ups)$ there exists a set $N_f\in\vS_0$ with $\eta_x([\varphi(f)]_{x})=[\varphi(f)]_{x}$ for all $x\notin N_f$.
\end{enumerate}
\end{cor}

\begin{prop}\label{p:2.marg}
Let $(X,\vS,\mu)$, $(Y,T,\nu)$, $(X\times Y,\vY,\ups)$ be probability spaces satisfying $[C]$ and $\wh{\vS}_0\,\dot{\times}\,Y\sq \vY$, and suppose that there is a $\theta\colon T\to T$ satisfying $(L1)$ which is a $2$-marginal for $\vY$. Then there is a $\de\colon \LL^0(\nu)\to \LL^0(\nu)$ satisfying $(l1)$ which is a $2$-marginal for $\LL^0(\ups)$. If $Y$ has a topology $\mathfrak{T}\sq T$ with respect to which $\theta$ is strong, then we may take $\de$ to be strong, i.e., to fix the elements of $C(Y,\mathfrak{T})$.
\end{prop}

\begin{proof}
We may assume that $\K=\R$. By $[C]$ and $\wh{\vS}_0\,\dot{\times}\,Y\sq \vY$, it follows from Proposition 4.5 of \cite{bms6} that $\theta_0\colon\LL^0(\ov{\R},\nu)\to \LL^0(\ov{\R},\nu)$ is a $2$-marginal with respect to $\LL^0(\ov{\R},\ups)$, where $\theta_0(f)(y) = \sup\{r\in\Q: y\in\theta\{r<f\}\}$. By Proposition 3.4 (iv) of \cite{bms6}, $\theta_0$ satisfies $(l1)$.
Let $\de\colon\LL^0(\nu)\to \LL^0(\nu)$ be the restriction of $\theta_0$ to $\LL^0(\nu)$, with values modified to replace $\pm\infty$ by $0$, i.e., $\de(f)(x)=\theta_0(f)(x)$ when $\theta_0(f)(x)\not=\pm\infty$, $\de(f)(x)=0$ otherwise.

We check that $\de$ is a $2$-marginal with respect to $\LL^0(\ups)$. If $f\in\LL^0(\ups)$ then $(\theta_0)_\bullet(f)$ is $\vY$-measurable. By Lemma \ref{VM19} (i) (which works in the same way for functions into $\ov{\R}$), $(\theta_0)_\bullet(f) =_\ups f$ and hence $A=\{(\theta_0)_\bullet(f)=\pm\infty\}\in\vY_0$. We have
$\de_\bullet(f)\restr A^c = (\theta_0)_\bullet(f)\restr A^c$ and $\de_\bullet(f)\restr A\equiv 0$, so $\de_\bullet(f)$ is $\vY$-measurable.

For the last sentence, Proposition 3.6 of \cite{bms6} implies that $\theta_0$ is strong and hence $\de$ is strong.
\end{proof}

\begin{que}
Is there a vector lifting for Lebesgue measure on $[0,1]$ which is a $2$-marginal for $\LL^0(\ups)$, where $\ups$ is Lebesgue measure on $[0,1]^2$?
\end{que}

In the next example, we show that it is possible that a primitive lifting for $\LL^p$ which is a $2$-marginal for $\LL^q$ for all $q>p$ might not be a $2$-marginal for $\LL^p$.

\begin{ex}\label{ex:p.not.q}
Let $(X,\vS,\mu)$ and $(Y,T,\nu)$ be the Lebesgue measure space on $(0,1)$, and let $(X\times Y,\vY,\ups)$ be the Lebesgue measure space on $(0,1)^2$.
For any $p\in[0,\infty)$, there is a primitive lifting for $\LL^p(\nu)$ which is not a $2$-marginal with respect to $\LL^p(\ups)$ but whose restriction to $\LL^q(\nu)$ is a $2$-marginal with respect to $\LL^q(\ups)$ for each $q>p$.

\begin{proof}
It suffices to do the case $\K=\R$. Since there are strong liftings for $\nu$ which are $2$-marginals for $\vY$ by \cite{bms4} Example 4.8, Proposition \ref{p:2.marg} gives a  $\de\colon\LL^0(\nu)\to \LL^0(\nu)$ which is a strong $2$-marginal with respect to $\LL^0(\ups)$.

Fix a positive continuous function $h$ on $(0,1)$ which belongs to $\LL^p(\nu)$  but not to $\LL^q(\nu)$ for any $q>p$. (One possibility: if $p=0$ take $h(x)=e^{1/x}$; if $0<p<\infty$ take $h(x) = [x\log^2(2/x)]^{-1/p}$.) Define $f(x,y)=xh(y)$. We have $f\in\LL^p(\ups)$. Note that the sections $f_x$ are distinct. They are also continuous, so $\de(f_x)=f_x$ since $\de$ is strong.
Fix a set $E\sq (0,1)^2$ which is the graph of a function such that each section $E^y$ has outer measure one. (See Example \ref{ex:[C]}.) Let $\rho\colon\LL^p(\nu)\to \LL^p(\nu)$ agree with $\de$ except for setting $\rho(f_x)(y)=0$ when $(x,y)\in E$. When $q>p$, for any function $g\in\LL^q(\ups)$ we have $\de_\bullet(g)\in \LL^q(\ups)$. For some $N\in\vS_0$, $g_x\in\LL^q(\nu)$, and hence $\rho(g_x)=\de(g_x)$, for $x\in N^c$. Thus, $\rho_\bullet(g)$ agrees with $\de_\bullet(g)$ on $N^c\times Y$ and therefore $\rho_\bullet(g)\in \LL^q(\ups)$. However, for $g=f$, the zero set of $\rho_\bullet(f)$ is $E$ and hence $\rho_\bullet(f)$ is not $\vY$-measurable.
By \cite{bms6}, Proposition 3.4, $\de$ satisfies $(l2)$ and fixes constant functions; hence $\de$ and $\rho$ are primitive liftings for $\mathcal{L}^p(\nu)$.
\end{proof}
\end{ex}

The following lemma corresponds to \cite{bms4}, Lemma 4.9.%
\footnote{The claim after the statement of Lemma 4.9 in \cite{bms4} that the assumptions imply $\nu$ is complete is incorrect. Completeness of $\nu$ should be added to the hypothesis. Alternatively, write  $\mathfrak{N}(\mu,\nu)$ instead of $\mathfrak{N}$ as we have done here.}

\begin{lem}\label{VM40aa}
Let $(X,\vS,\mu)$, $(Y,T,\nu)$, $(X\times Y,\vY,\ups)$ be probability spaces satisying $[P_2]$ and $[C]$, and let $(X\times Y,\vY,\I)\in \T_\s$ with $\vS_0\,\wh{\times}\,Y\sq \mathcal{I} \sq \Kk:=\mathfrak{N}(\mu,\nu) = \vS_0\ltimes T_0$.
If $\eta:\LL^p(\nu)\rightarrow\K^Y$ satisfies $(l2)$ and $\eta_\bullet(g)\in \LL^p(\ups_\I)$ when $g\in \LL^p(\mu\otimes\nu)$, then
\begin{enumerate}
\item
$\eta_\bullet(f)\in\LL^p(\ups_\I)$ for each
$f\in\LL^p(\ups_\Kk)$.

\item
$\eta$ is a $2$-marginal with respect to $\LL^p(\ups_J)$ for any $J$ with $(X\times Y,\vY,J)\in \T_\s$, $\mathcal{I} \sq J\sq \Kk$.
\end{enumerate}
\end{lem}

\begin{rem}
(a) $\Kk$ is an ideal of $\vS_0\ltimes T$ and $\vY\sq \vS_0\ltimes T$, so $(X\times Y,\vY,\Kk)$ belongs to $\T_\s$.

(b) $\I\sq \Kk$ are $\s$-ideals of the $\s$-algebras $\vY_\I \sq \vY_\Kk$, respectively.
\end{rem}

\begin{proof}
For each ideal $\mathcal{H}=\I,\Kk$, and in (ii) also $J$, since $\mathcal{H}\sq \Kk$ and $[C]$ holds, the condition $\mathcal{H}\cap \vY\sq \vY_0$ is satisfied, so $\ups_\mathcal{H}$ is defined.

(i) Fix $f\in\LL^p(\ups_\Kk)$. Choose $g\in\LL^p(\mu\otimes\nu)$ such that $f =_\Kk g$ by Lemma \ref{VMLa}. By $[C]$ for $\ups_\Kk$ (Proposition \ref{p:VMLa:2}), fix $N\in \vS_0$ such that $f_x,g_x\in\LL^p(\nu)$ and $\{f_x\not=g_x\}\in T_0$ for all $x\in N^c$. For $x\in N^c$ we have $f_x =_\nu g_x$, so $\eta(f_x)=\eta(g_x)$ by $(l2)$. It follows that
$\eta_\bullet(f)\restr(N^c\times Y)=\eta_\bullet(g)\restr(N^c\times Y)$.
By assumption, $\eta_\bullet(g)\in\LL^p(\ups_\I)$. Then because $\vS_0\,\wh{\times}\,Y\sq \mathcal{I}$ it follows that $\eta_\bullet(f)\in\LL^p(\ups_\I)$.

(ii) Follows from (i) since $\LL^p(\ups_\I)\sq \LL^p(\ups_J) \sq \LL^p(\ups_\Kk  )$.
\end{proof}

\begin{prop}\label{VM40c}
Let $(X,\vS,\mu)$, $(Y,T,\nu)$ and $(X\times Y,\vY,\ups)$ be probability spaces satisfying $[P_0]$ and $[C]$. Let $\Kk:=\mathfrak{N}(\mu,\nu)=\vS_0\ltimes T_0$.
For every $\gamma\in V^p(\mu)$ and  $\eta\in V^p(\nu)$ there exists a $\varphi\in V^p(\ups_{\Kk})\cap(\gamma\otimes\eta)$ such that $\eta([\varphi(f)]_x)=[\varphi(f)]_x$ for every $f\in\LL^p(\ups_{\Kk})$ and $x\in X$.
\end{prop}

\begin{proof}
By Proposition \ref{p:VMLa:2}, $\ups_{\Kk}$ is defined and $[C]$ holds for the triple $(X,\vS,\mu)$, $(Y,T,\nu)$, $(X\times Y,\vY_{\Kk},\ups_{\Kk})$.
By Proposition \ref{PV20} choose
$\pi\in(\gamma\otimes\eta)\cap V^p(\ups)$ such that for every $f\in\LL^p(\ups)$,
$[\pi(f)]_x\in\LL^p(\nu)$ for all $x\in X$. Define
$\psi:=\eta_{\bullet}\circ\pi$, i.e.,
$[\psi(f)]_x:=\eta([\pi(f)]_x)$ for $f\in\LL^p(\ups)$ and $x\in X$.
Clearly $\psi(1)=1$, and
\[
\eta([\psi(f)]_x)=[\psi(f)]_x\;\mbox{for}\; f\in\LL^p(\upsilon)\;\mbox{and}\; x\in X.
\]
By Lemma \ref{l:eta.skew:m} the map $\eta$ is a 2-marginal with respect to $\LL^p(\ups_{\Kk})$, so $\psi(f)=\eta_{\bullet}(\pi(f))\in\LL^p(\ups_{\Kk})$ and $\psi(f)=_{\Kk}\pi(f)=_{\ups}f$, implying  $\psi(f)=_{\Kk}f$ since $\vY_0\sq{\Kk}$. Thus, we have
\[
\psi\colon\LL^p(\upsilon)\to \LL^p(\ups_\Kk),\ \ \text{$\psi(f)=_\Kk f$.}
\]
Since $\pi$ satisfies $(l2)$, so does $\psi = \eta_\bullet\circ \pi$.
By Lemma \ref{VM10a} (iv)(a,b), $\psi$ is linear since $\pi$ is linear and $[\pi(f)]_x\in\LL^p(\nu)$ for every $x\in X$.
For $f\in\LL^p(\mu)$, $g\in\LL^p(\nu)$, we have, using Lemma \ref{VM10a} (ii),
\[
\psi(f\otimes{g}) = \eta_{\bullet}(\pi(f\otimes{g})) =
\eta_{\bullet}(\gamma(f)\otimes\eta(g)) = \gamma(f)\otimes\eta(\eta(g)) = \gamma(f)\otimes\eta(g),
\]
i.e. we have $\psi\in\gamma\otimes\eta$.

Thus, $\psi$ has all the desired properties of $\varphi$ except that it is defined only on $\LL^p(\upsilon)$ rather than on all of $\LL^p(\ups_\Kk)$.
Using Lemma \ref{P10i} (ii), we now readily extend the domain of $\psi$ to all of $\LL^p(\upsilon_{\Kk})$. Since $\ups_\Kk$ extends $\ups$, if $g_1,g_2\in\LL^p(\ups)$ satisfy $g_1 =_\Kk g_2$ then $g_1 =_\ups g_2$. Hence, we have a linear isomorphism $i\colon L^p(\ups_\Kk)\to L^p(\ups)$ where for each $g\in \LL^p(\ups)$, we map $g^\bullet\in L^p(\ups_\Kk)$ to $g^\bullet\in L^p(\ups)$. The composition
\[
\LL^p(\ups_\Kk) \stackrel{q}{\longrightarrow} L^p(\ups_\Kk) \stackrel{i}{\longrightarrow} L^p(\ups) \stackrel{\ov{\psi}}{\longrightarrow} \LL^p(\ups)
\]
is the desired extension $\varphi$ of $\psi$, where $q(f)=f^\bullet$ and $\ov{\psi}(g^\bullet)=\psi(g)$. Clearly $\varphi(1)=1$, and $\varphi$ is a composition of linear maps and hence linear. Since $q$ satisfies $(l2)$, so does $\varphi$. Invariance of the sections under $\eta$ holds since for $f\in\LL^p(\ups_\Kk)$, $[\varphi(f)]_x = [\psi(g)]_x$ for any $g\in\LL^p(\ups)$ in the $\ups_\Kk$-class of $f$.
Also for such $f$ and $g$, $\varphi(f) = \psi(g) =_\Kk g =_\Kk f$, so $(l1)$ holds.
The property $\varphi\in \gamma\otimes\eta$ holds since this property only deals with functions $f\otimes g$ that belong to $\LL^p(\ups)$ on which $\varphi$ and $\psi$ agree.
\end{proof}

Taking $\ups=\mu\otimes\nu$ with $\nu$ complete gives the following.

\begin{cor}\label{VM30}
Let  $(X,\vS,\mu)$ and $(Y,T,\nu)$ be probability spaces with $\nu$ complete. For every
$\gamma\in V^p(\mu)$ and $\eta\in V^p(\nu)$ there exists a $\varphi\in V^p(\mu\otimes_\mathfrak{N}\nu)\cap(\gamma\otimes\eta)$,  such that $\eta([\varphi(f)]_x)=[\varphi(f)]_x$ for every $f\in\LL^p( \mu\otimes_\mathfrak{N}\nu)$ and $x\in X$.
\end{cor}

\begin{thm}\label{VM50}
Let $(X,\mathfrak{S},\vS,\mu)$, $(Y,\mathfrak{T},T,\nu)$ and $(X\times Y,\mathfrak{S}\times\mathfrak{T},\vY,\ups)$ be topological probability spaces satisfying $[P_0]$ and $[C]$. Let $\Kk:=\mathfrak{N}(\mu,\nu)=\vS_0\ltimes T_0$. Assume $(N_{C,\mu})$, $(N_{C,\nu})$ hold, and
\begin{equation}\label{eq:cont}
\text{$h_x\in C^p(Y)$ for all $h\in C^p(X\times Y)$ and $x\in X$.}
\end{equation}
For every $\gamma\in V^p(\mathfrak{S},\mu)$ and  $\eta\in V^p(\mathfrak{T},\nu)$ there exists a $\varphi\in V^p(\mathfrak{S}\times \mathfrak{T}, \ups_\Kk)\cap(\gamma\otimes\eta)$ such that $\eta([\varphi(f)]_x)=[\varphi(f)]_x$ for every $f\in\LL^p(\ups_\Kk)$ and $x\in X$.
\end{thm}

\begin{rem}\label{r:cont}
Here are some conditions under which (\ref{eq:cont}) holds:
(1) $p=0$;
(2) $p=\infty$ (by Lemma \ref{PVR}, $(N_{C,\upsilon})$ holds, so if $h\in C(X\times Y)$ and $\|h\|_\infty<\infty$, then $h$ is bounded by  Remark \ref{V20E} (iii));
(3) any $p$, if $Y$ is pseudocompact, i.e., the elements of $C(Y)$ are all bounded.
\end{rem}

\begin{proof} By Theorem \ref{PV30} choose a $\pi\in(\gamma\otimes\eta)\cap{V}^p(\mathfrak{S}\times\mathfrak{T},\ups)$ for every $f\in\LL^p(\ups)$ satisfying $[\pi(f)]_x\in\LL^p(\ups)$ for all $x\in{X}$, and repeat the construction of $\varphi$ from the proof of Proposition \ref{VM40c} using this $\pi$. From (\ref{eq:cont}) and Lemma \ref{VM10a} (iv)(d) we have for $h\in C^p(X\times Y)$ that $\eta_{\bullet}(h)=h$. Along with the fact that $\pi$ is $C^p(X\times{Y})$-strong this implies that $\varphi(h)=\psi(h)=\eta_{\bullet}(\pi(h))=\eta_{\bullet}(h)=h$,  i.e. $\varphi$ is $C^p(X\times{Y})$-strong.
\end{proof}

\begin{cor}\label{VM60}
Let $(X\times Y,\mathfrak{S}\times{\mathfrak T},\Upsilon,\upsilon)$ be the
$\tau$-additive product of the $\tau$-additive topological probability spaces $(X,{\mathfrak S},\vS,\mu)$ and $(Y,\mathfrak{T},T,\nu)$ with $\nu$ complete. Assume that $(N_{C,\mu})$, $(N_{C,\nu})$, and {\rm(\ref{eq:cont})} hold. For every $\gamma\in V^p({\mathfrak S},\mu)$, $\eta\in V^p({\mathfrak T},\nu)$ there exists a $\varphi\in V^p(\upsilon_{\mathfrak N})\cap(\gamma\otimes\eta)$ such that $\eta([\varphi(f)]_x)=[\varphi(f)]_x$ for every $f\in\LL^p(\ups_\mathfrak{N})$ and $x\in X$.
\end{cor}

\begin{proof}
Since the assumptions of Theorem \ref{VM50} are satisfied by \cite{bms5}, Theorem 2.15,  we may apply Theorem \ref{VM50} in order to get the result.
\end{proof}

\section{Measurable modifications of stochastic processes}
\label{s:meas.stoch.proc}

In this section we prove that the existence of a measurable $p$-integrable vector lifting modification for stochastic processes is equivalent to the existence of a marginal vector lifting in one of the factors (see Theorem \ref{slc30}) and, in case of modifications respecting continuous functions to the existence of strong marginals (see Corollary \ref{s50AR}). In Theorem \ref{t:mm1} and Corollary \ref{c:mm1} we provide a characterization of stochastic processes which have a measurable modification.

Because our main objects of study here are liftings, in this section we assume throughout that $\K=\R$. The results not mentioning liftings work for complex scalars as well (by the same arguments).
Throughout this section, $p\in[0,\infty] $ and we fix a measurable space $(\Gamma,\B)$, where $\Gamma$ is a Tychonoff space and $\B$ is its $\s$-algebra of Baire sets.

Given a measurable space $(X,\vS)$, write $\LL^\infty_\Gamma = \LL^\infty_\Gamma(\vS)$ for the space of all bounded $\vS$-$\B$-measurable maps from $X$ into $\Gamma$, where we say $f\colon X\to \Gamma$ is \emph{bounded} if $f(X)\sq K$ for some compact $K\sq \Gamma$.

\begin{rem}\label{r:meas.ble}
Measurability of $f\colon X\to \Gamma$ is equivalent to $\vS$-measurability of each composition $h\circ f$, for $h\in C(\Gamma)$,
since $\B$ is the Baire $\s$-algebra.
\end{rem}

\begin{df}\label{co}
Fix a probability space $(X,\vS,\mu)$ and $\rho\in\Lambda^\infty(\mu)$. Then $\rho^\prime:\LL^\infty_\Gamma\to\LL^\infty_\Gamma$, the lifting of $\LL^\infty_\Gamma$ \emph{associated to} $\rho$, is the unique map satisfying
\begin{enumerate}
\item[(a)]
$h\circ\rho^\prime(f)=\rho(h\circ f)$ for $f\in\LL^\infty_\Gamma$ and $h\in C(\Gamma)$,
\item[(b)]
$\rho^{\prime}(f)=^*_\mu f$ for $f\in\LL^\infty_\Gamma$,
\item[(c)]
$\rho^\prime(f)=\rho^\prime(g)$ if $f=^*_\mu{g}$ for $f,g\in \LL^\infty_\Gamma$.
\end{enumerate}
where $f =_\mu^* g$ means $h\circ f =_\mu h\circ g$ for all $h\in C(\Gamma)$.
\end{df}

\begin{rem}\label{r:co}
(i)
It follows by \cite{bms6}, Remark 6.5(a) that, property (a) already uniquely determines $\rho^\prime$ and hence necessarily implies (b) and (c).

(ii)
For the proof that $\rho^\prime$ exists, see
Theorem IV.7 p.\ 52 of \cite{it}.

(iii)
Again by \cite{bms6}, Remark 6.5(d), in (b) we cannot strengthen $=^*_\mu$ to $=_\mu$.
\end{rem}

\begin{df}\label{mf10}
Let $(Y,T)$, $(X\times Y,\Upsilon)$ be measurable spaces. For (c)--(e) below, assume we also have a family of measures $\nu:=(\nu_x)_{x\in X}$ on $T$.

A family $(Q_x)_{x\in X}$ of $T$-$\B$-measurable functions $Q_x\colon Y\to \Gamma$ will be called a \emph{stochastic process} (or just a \emph{process}) \emph{over} $(Y,T)$ \emph{with values in} $\Gamma$. We say that the function $Q\colon X\times Y\to \Gamma$ given by $Q(x,y)=Q_x(y)$ and the process $(Q_x)_{x\in X}$ are \emph{associated} to each other. Conveniently, $(Q_x)_{x\in X}$ is the family of vertical sections of $Q$, so there is no conflict with our notation for vertical sections.  The process is called
\begin{enumerate}
\item[(a)]
\emph{bounded} if $Q_x\in\LL_\Gamma^{\infty}(T)$ for every $x\in X$,

\item[(b)]
$\vY$-$\B$-\emph{measurable} if $Q$ is $\vY$-$\B$-measurable,

\item[(c)]
an  \emph{$\LL^p$-process} if $\Gamma=\R$ and $Q_x\in\LL^p(\nu_x)$ for each $x\in X$,

\item[(d)]
a \emph{modification} of $(R_x)_{x\in X}$ if $Q_x=_{\nu_x}{R}_x$ for every $x\in X$,

\item[(e)]
a \emph{weak modification} of $(R_x)_{x\in X}$ if $Q_x=^*_{\nu_x}R_x$ for every $x\in X$,
\end{enumerate}
where $(R_x)_{x\in X}$ is another stochastic process over $(Y,T)$ with values in $\Gamma$.
\end{df}

The next theorem does for liftings for $\LL^p(\nu)$ something analogous to what Theorem 6.10 of \cite{bms6} did for liftings for the measure algebra of $\nu$.

\begin{thm}\label{slc30}
Let $(X,\vS,\mu)$, $(Y,T,\nu_x)_{x\in X}$, and $(X\times Y,\vY,\ups)$ be probability spaces satisfying $[C]$, $\vS\,\dot{\times}\,Y\sq\vY$. For $\rho:=(\rho_x)_{x\in X}$ with $\rho_x\colon \LL^p(\nu_x)\to \LL^p(\nu_x)$ satisfying $(l1)$ and $(l2)$ with either $(\al)$ $(\vS_0\,\wh{\times}\,Y)[T]\sq\vY$ or $(\beta)$ $\rho(0)=0$ and $\mu$ is complete, consider the following statements. In {\rm (iii)}, assume $\rho_x\in\Lambda^\infty(\nu_x)$ for each $x$.
\begin{enumerate}
\item
$\rho$ is a $2$-marginal with respect to
$\LL^p(\ups)$.

\item
$(\rho_x(Q_x))_{x\in X}$ is a $\vY$-measurable $\LL^p$-process for every $\vY$-measurable $\LL^p$-process $(Q_x)_{x\in X}$.

\item
$(\rho'_x(Q_x))_{x\in X}$ is $\vY$-$\B$-measurable for every bounded $\vY$-$\B$-measurable process $(Q_x)_{x\in X}$ over $(Y,T)$ with values in $\Gamma$, for any choice of $\Gamma$.
\end{enumerate}
Then {\rm(i)} is equivalent to {\rm (ii)}. When $p=\infty$ and $\rho_x\in\Lambda^\infty(\nu_x)$ for $x\in X$, {\rm(i)}--{\rm(iii)} are all equivalent.
\end{thm}

\begin{rem}\label{r:slc30}
(a) When $p\in(0,\infty)$, Theorem \ref{slc30} applies to vector liftings $\rho_x\in V^p(\nu_x)$ but none of these are linear liftings (Proposition \ref{p:o-p}).

(b) In (iii) $\bigl(\rho'_x(Q_x)\bigr)$ is a process by definition of $\rho^\prime$, and it is a weak modification of $(Q_x)$ by Definition \ref{co} (b). (It may not be a modification of $(Q_x)$ by \cite{bms6}, Remark 6.11.)

(c) For the case $p=\infty$ and $\rho_x\in\Lambda^\infty(\nu_x)$, the proof of (ii) $\Leftrightarrow$ (iii) below makes use of none of the assumptions on $\ups$.
\end{rem}

\begin{proof}
(ii) is equivalent to stating that $\rho$ is a $2$-marginal for $\LL^{0,p}(\ups,\nu)$, so the equivalence of (i) and (ii) follows from Proposition \ref{p:slc30d}.
For the rest of the proof, let $p=\infty$ and $\rho_x\in\Lambda^\infty(\nu_x)$ for each $x$.

(ii) $\Rightarrow$ (iii) This part follows as for the proof of (i) $\Rightarrow$ (iv) in Theorem 6.10 of \cite{bms6}. We repeat the argument for completeness.
By (ii), $\rho$ is a $2$-marginal with respect to
$\LL^{0,\infty}(\ups,\nu)$.
Let $Q\colon X\times Y\to \Gamma$ be $\vY$-$\B$-measurable with $Q_x\in\LL^\infty_\Gamma(T)$ for all $x\in X$. Let $g\in C(\Gamma)$. Then $g\circ Q\colon X\times Y\to\R$ is $\vY$-measurable, and each section $[g\circ Q]_x = g\circ Q_x$ belongs to $\LL^\infty(\nu_x)$ since $Q_x$ is bounded. Thus, $g\circ Q\in\LL^{0,\infty}(\ups,\nu)$ and hence $\rho_\bullet (g\circ Q)\in \LL^{0,\infty}(\ups,\nu)$.
The map
\begin{align*}
(x,y) \mapsto{} & \rho_\bullet (g\circ Q)(x,y) = [\rho_\bullet (g\circ Q)]_x(y) \\
& = \rho([g\circ Q]_x)(y) = \rho(g\circ Q_x)(y) = g(\rho^\prime(Q_x)(y))
\end{align*}
is therefore $\vY$-measurable for all $g\in C(\Gamma)$, and therefore
$(x,y)\mapsto \rho^\prime(Q_x)(y)$ is $\vY$-$\B$-measurable.

(iii) $\Rightarrow$ (ii) We apply (iii) with $\Gamma=\R$. For $f\in\LL^\infty(\nu_x)$, by \cite{it}, (2) page 35, we have $g\circ\rho(f)=\rho_x(g\circ{f})$ for any $g\in C(\R)$. From Remark \ref{r:co} (i), it then follows that in this case we have $\rho'_x = \rho_x$.
Thus, when $p=\infty$, the only distinction between (ii) and (iii) with $\Gamma=\R$ is that in (ii), the $Q_x$ are essentially bounded, whereas in (iii) they are bounded.

We want to show that $\rho$ is a $2$-marginal with respect to $\LL^{0,\infty}(\ups,\nu)$. Let $f\in\LL^{0,\infty}(\ups,\nu)$. We use an idea from the proof of Proposition \ref{p:slc30d}. Define $f_n=(-n)\vee (f\wedge n)$ for $n\in\N$. Then $f_n\in\LL^{0}(\ups)$ and each section $[f_n]_x$ is bounded. By (iii), the process $(\rho_x([f_n]_x))_{x\in X}$ is $\vY$-measurable. For each $x$, when $n$ is large enough we have $f_x =_{\nu_x} [f_n]_x$ , so $[\rho_\bullet(f)]_x = [\rho_\bullet(f_n)]_x$. Thus, $\rho_\bullet(f) = \lim_{n\to\infty}\rho_\bullet(f_n)$ is $\vY$-measurable, yielding $\rho_\bullet(f)\in\LL^{0,\infty}(\ups,\nu)$, as desired.
\end{proof}

\begin{cor}\label{s50AR}
Let $(X,\vS,\mu)$, $(Y,T,\nu_x)_{x\in X}$, and $(X\times Y,\vY,\ups)$ be probability spaces satisfying $[C]$, $\vS\,\dot{\times}\,Y\sq\vY$, and $\vS_0\,\wh{\times}\,Y\sq\vY$. Suppose we have a topology $\mathfrak{T}\sq T$ on $Y$ with $X\,\dot{\times}\,\mathfrak{T}\sq\vY$.
For $\rho:=(\rho_x)_{x\in X}$ with $\rho_x\colon \LL^p(\nu_x)\to \LL^p(\nu_x)$ satisfying $(l1)$ and $(l2)$, consider the following statements. In {\rm (iii)}, assume $\rho_x\in\Lambda^\infty(\nu_x)$ for each $x$.
\begin{enumerate}
\item
$\rho$ is a $C^p(Y)$-strong $2$-marginal with respect to $\LL^p(\ups)$.

\item
$(\rho_x(Q_x))_{x\in X}$ is a $\vY$-measurable $\LL^p$-process for every $\vY$-measurable $\LL^p$-process $(Q_x)_{x\in X}$
and $Q_x=\rho_x(Q_x)$ when $Q_x$ is continuous.

\item
$(\rho'_x(Q_x))_{x\in X}$ is $\vY$-$\B$-measurable for every bounded $\vY$-$\B$-measurable  process $(Q_x)_{x\in X}$ with values in $\Gamma$, for any choice of $\Gamma$, and $\rho'_x(Q_x) = Q_x$ when $Q_x$ is continuous.
\end{enumerate}
Then {\rm(i)} is equivalent to {\rm(ii)}. If $p=\infty$ and $\rho_x\in\Lambda^\infty(\nu_x)$ for $x\in X$, {\rm(i)}--{\rm(iii)} are all equivalent.
\end{cor}

\begin{proof}
The equivalence of the statements (i) and (ii) is immediate by Theorem \ref{slc30}. The assumption $X\,\dot{\times}\,\mathfrak{T}\sq\vY$ ensures for (ii) $\Rightarrow$ (i) that if $g\in C^p(Y)$ then $Q=1\otimes g$ is the associated function of a $\vY$-measurable $\LL^p$-process.
Now assume $p=\infty$ and $\rho_x\in\Lambda^\infty(\nu_x)$ for each $x$.

For (i) $\Rightarrow$ (iii), by Theorem \ref{slc30} we get that $(\rho_x(Q_x))_{x\in{X}}$ is $\vY$-measurable for every bounded $\vY$-$\B$-measurable process $(Q_x)_{x\in{X}}$ with values in $\Gamma$. The equality $\rho'_x(Q_x) = Q_x$ in (iii) follows from $g\circ\rho'_x(Q_x) = \rho_x(g\circ Q_x) = g\circ Q_x$ for $g\in C(\Gamma)$.
The converse implication (iii) $\Rightarrow$ (ii) is clear by Theorem \ref{slc30}, taking $\Gamma=\R$, which as pointed out before makes $\rho'_x = \rho_x$.
\end{proof}

\begin{ex}\label{s50ar1}
Let $(X,\vS,\mu)$ and $(Y,T,\nu)$ be probability spaces.

(a) Assume that $\nu$ is complete. Every $\rho\in A\mathcal{G}(\nu)$ is a $2$-marginal with respect to $\LL^\infty(\mu\,\wh\otimes\,\nu)$ by \cite{mms15}, Theorem 2.6 along with Corollary \ref{VM19b}. It follows that the clauses \rm{(i)} and \rm{(ii)} of Theorem \ref{slc30} hold. See \cite{mms12} for the definition of $A\mathcal{G}(\nu)$.

(b) If $\nu$ is complete and $\rho\in\Lambda^\infty(\nu)$ is an admissibly generated lifting (see \cite{mms2}, page 291 for the definition), then we get the existence of a lifting $\pi\in\Lambda^\infty(\mu\,\wh\otimes\,\nu)$ with $\rho$-invariant sections: $[\pi(f)]_x=\rho([\pi(f)]_x)$ for each $f\in\LL^\infty(\mu\,\wh\otimes\,\nu)$ and each $x\in{X}$, by \cite{mms2}, Theorem 2.13. By Corollary \ref{VM19b}, this property implies that $\rho$ is a $2$-marginal with respect to $\LL^\infty(\mu\,\wh\otimes\,\nu)$, and so the clauses of Theorem \ref{slc30} all hold for
$(X,\vS,\mu)$, $(Y,T,\nu)$, $(X\times Y,\vS\,\wh{\otimes}\, T, \mu\,\wh{\otimes}\,\nu)$.

(c) Let $(X\times{Y},\vY,\ups)$ be a probability space satisfying $[C]$ and $\vS\,\dot{\times}\,Y\sq\vY$.
If $\eta:\LL^p(\nu)\to \LL^p(\nu)$ satisfies $(l1)$ and $(l2)$ then
$\eta$ is a $2$-marginal with respect to $\LL^p(\ups_{\mathfrak{N}(\mu,\nu)})$ by Example \ref{ex:eta.skew.1}.
It follows that the clauses of Theorem \ref{slc30} hold for the spaces $(X,\vS,\mu)$, $(Y,T,\nu)$, $(X\times Y,\vY_{\mathfrak{N}(\mu,\nu)}, \ups_{\mathfrak{N}(\mu,\nu)})$ which satisfy $[C]$ by Proposition \ref{p:VMLa:2} and satisfy $\vS_0\,\wh{\times}\,Y\sq\mathfrak{N}(\mu,\nu)\sq \vY_{\mathfrak{N}(\mu,\nu)}$. (Clause (iii) requires $p=\infty$, $\eta\in\vL^\infty(\nu)$.)
\end{ex}

For part (a) of the next example, cf.\ \cite{bms6}, Example 6.14(a).

\begin{ex}\label{s50ar2}
Let $(X,\mathfrak{S},\vS,\mu)$ and $(Y,\mathfrak{T},T,\nu)$ be topological probability spaces.

(a) If $\mu,\nu$ are $\tau$-additive then from \cite{bms5}, Theorem 2.15 we get a $\tau$-additive product space $(X\times Y,\mathfrak{S}\times \mathfrak{T}, \vY,\ups)$.
By the proof of (v) of that theorem, $[C]$ holds if we replace $(\vY,\ups)$ by $(\ov{\vY},\ov{\ups})$, where $\ov{\vY}:=\mathfrak{B}(X\times Y)$, $\ov{\ups}:=\ups\restr\ov{\vY}$. Extending by the ideal $\mathfrak{N}$ to get $\ov{\vY}_\mathfrak{N},\ov{\ups}_\mathfrak{N}$, we preserve $[C]$ by
Proposition \ref{p:VMLa:2}.
If $\rho\in{V}^p(\mathfrak{T},\nu)$, then similarly to what we did in Example \ref{s50ar1} (c), we may apply Example \ref{ex:eta.skew.1} in order to get that $\rho$ is a $2$-marginal with respect to $\LL^p(\ov{\ups}_\mathfrak{N})$. Then clauses (i) and (ii) of Corollary \ref{s50AR} hold for $(X,\vS,\mu)$, $(Y,T,\nu)$, $(X\times Y,\ov{\vY}_\mathfrak{N},\ov{\ups}_\mathfrak{N})$. Clause (iii) also holds when $p=\infty$, $\rho\in\vL^\infty(\nu)$.

(b) If $\nu$ is complete and $\rho\in\Lambda^\infty(\nu)$ is a strong admissibly generated lifting,
then there exists a $\mathfrak{S}\times\mathfrak{T}$-strong lifting $\pi\in\Lambda^\infty(\mu\,\wh\otimes\,\nu)$ with $\rho$-invariant sections as in Example \ref{s50ar1} (b), and such that the product topology $\mathfrak{S}\times\mathfrak{T}$ consists of measurable sets by \cite{mms19}, Theorem 4.1.%
\footnote{Recall that for topological probability spaces $(X,\mathfrak{S},\vS,\mu)$ and $(Y,\mathfrak{T},T,\nu)$, in general we don't have $\mathfrak{S}\times\mathfrak{T}\sq\Sigma\otimes T$ (see \cite{fr3}, Lemma 346K and \cite{gg}).}
We then get, as in Example \ref{s50ar1} (b), that $\rho\in\Lambda^\infty(\nu)$ is a $2$-marginal with respect to $\LL^\infty(\mu\,\wh\otimes\,\nu)$. But since $\pi$ is $\mathfrak{S}\times\mathfrak{T}$-strong, it follows by \cite{fr4}, 453C, that $\pi$ is $C(X\times{Y})\cap\LL^\infty(\mu\,\wh\otimes\,\nu)$-strong. Thus, the clause \rm{(i)} of Corollary \ref{s50AR} holds, and so the clauses of that Corollary all hold for $(X,\mathfrak{S},\vS,\mu),\ \ (Y,\mathfrak{T},T,\nu),\ \ (X\times Y,\mathfrak{S}\times\mathfrak{T},\vS\,\wh{\otimes}\, T, \mu\,\wh{\otimes}\,\nu)$.
\end{ex}

Cohn \cite{co72}, Theorem 3 (see also Chung and Doob \cite{cd}) has given a necessary and sufficient condition  for a process to have a measurable modification. The result of Cohn has been extended in \cite{bms6}, Theorem 8.8. Hoffmann-J{\o}rgensen  \cite{hj} has given another necessary and sufficient condition, which only depends on the 2-dimensional marginal distributions of the process. Musia{\l} \cite{mu23}, Proposition 4.1 states another characterization of processes possessing a measurable modification. The next theorem is inspired by that statement. See Corollary \ref{c:mm1} for the specialization to probability measures, and Remark \ref{r:mu} for how our result relates to the one in \cite{mu23}. We begin with a technical fact.

\begin{lem}\label{l:ups.T}
Let $X$ be a set, $\I$ an upwards directed family of subsets of $X$. Let
$(Y,T_x)_{x\in X}$ be measurable spaces and for each $x\in X$, let $J_x$ be an ideal of $T_x$. Write $T=(T_x)_{x\in X}$, $J=(J_x)_{x\in X}$.
Let $\vY$ be a $\s$-algebra on $X\times Y$ such that $\vY\sq \I\ltimes T$.
\begin{enumerate}
\item
If $(\I\,\wh{\times}\, Y)[T]\sq \vY$ then $\vY_{\I\ltimes J}[T] = \vY[T]_{\{\e\}\ltimes J}$.
\end{enumerate}
If $(Y,T_x,K_x)\in \T_\s$, $x\in X$, and $J_x = K_x\cap T_x$, $K=(K_x)_{x\in X}$ then
\begin{enumerate}
\setcounter{enumi}{1}
\item
$\vY_{\I\ltimes K}[T] = \vY_{\I\ltimes J}[T]$.
\end{enumerate}
\end{lem}

\begin{proof}
By Proposition \ref{p:skew2}, $(X\times Y,\vY,\mathcal{D})\in\T_\s$, and so $\vY_\mathcal{D}$ is defined, for $\mathcal{D} = \I\ltimes J$, $\I\ltimes K$, and similarly $(X\times Y,\vY[T],\{\e\}\ltimes J)\in\T_\s$ so that $\vY[T]_{\{\e\}\ltimes J}$ is defined.

(i) Since $\vY[T]_{\{\e\}\ltimes J} \sq \vY_{\I\ltimes J}$ and for any $E\in \{\e\}\ltimes J$ and $x\in X$ we have $E_x\in J_x\sq T_x$, we have $\vY[T]_{\{\e\}\ltimes J} \sq \vY_{\I\ltimes J}[T]$.
Conversely, let $A\in \vY_{\I\ltimes J}[T]$. Then for some $B\in \vY$ and $E\in \I\ltimes J$, we have $A = B\,\triangle\,E$. For some $N\in \I$, we have for all $x\in N^c$ that $B_x\in T$ and $E_x\in J_x$. It follows that $E\cap(N^c\times Y)\in \{\e\}\ltimes J$. We also have $B\cap(N^c\times Y)\in \vY$ since $N\times Y\in(\I\,\wh{\times}\, Y)[T]\sq \vY$, and hence $B\cap(N^c\times Y)\in \vY[T]$. Therefore, $A\cap(N^c\times Y) = (B\,\triangle\,E)\cap(N^c\times Y) \in \vY[T]_{\{\e\}\ltimes J}$.
By assumption, for all $x\in X$, $A_x\in T_x$, so
$A\cap(N\times Y)\in (\I\,\wh{\times}\, Y)[T]\sq \vY[T]\sq \vY[T]_{\{\e\}\ltimes J}$. Thus, $A\in\vY[T]_{\{\e\}\ltimes J}$.

(ii) Since $J_x\sq K_x$ for all $x$, we have $\vY_{\I\ltimes J}[T]\sq \vY_{\I\ltimes K}[T]$.
Conversely, let $A\in \vY_{\I\ltimes K}[T]$. Then for some $B\in \vY\sq \I\ltimes T$ and $E\in \I\ltimes K$, we have $A = B\,\triangle\,E$. For some $N\in \I$, we have for all $x\in N^c$ that $B_x\in T_x$ and $E_x\in K_x$. But we also have by assumption $A_x\in T_x$, so $E_x = A_x\,\triangle\,B_x\in T_x$ and therefore $E_x\in J_x$. It follows that $E\in \I\ltimes J$ and thus $A\in\vY_{\I\ltimes J}$.
\end{proof}

\begin{thm}\label{t:mm1}
Let $X$ be a set, $\I$ an upwards directed family of subsets of $X$. For $x\in X$, let $(Y,T_x,K_x)\in\T_\s$, $J_x:=K_x\cap T_x$. Write
\begin{center}
$T=(T_x)_{x\in X}$, $J=(J_x)_{x\in X}$, $K=(K_x)_{x\in X}$, $T_K = ((T_x)_{K_x})_{x\in X}$.
\end{center}
Let $\vY$ be a $\s$-algebra on $X\times Y$ such that $\vY\sq \I\ltimes T$.
Let $\F_T = \{f\in \R^{X\times Y}:f_x\in\LL^0(T_x),\,x\in X\}$. For $f\in\F_T$,
consider the following statements, where $\Omega$ is a $\s$-algebra with $\vY\sq\Omega\sq\vY_{\I\ltimes K}$.
\begin{enumerate}
\item[\rm(a)]
$f_x =_{J_x} g_x$ for all $x\in X$, for some $\vY$-measurable $g\in \F_T$.

\item[\rm(b)]
$f$ is $\vY[T]_{\{\e\}\ltimes T_0}$-measurable.

\item[\rm(c)]
$f$ is $\vY_{\I\ltimes J}$-measurable.

\item[\rm(d)]
$f$ is $\vY_{\I\ltimes K}$-measurable.

\item[\rm(e)]
$f_x =_{K_x} g_x$ for all $x\in X$, for some $\Omega$-measurable $g\in \F_{T_K}$.
\end{enumerate}
The following implications hold: {\rm(a)} $\Leftrightarrow$ {\rm(b)} $\Rightarrow$ {\rm(c)} $\Leftrightarrow$ {\rm(d)}, {\rm(a)} $\Rightarrow$ {\rm(e)} $\Rightarrow$ {\rm(d)}. If $(\I\,\wh{\times}\, Y)[T]\sq \vY$, then {\rm(c)} $\Rightarrow$ {\rm(b)} and therefore all statements are equivalent.
\end{thm}

\begin{proof}
The implications {\rm(b)} $\Rightarrow$ {\rm(c)} $\Rightarrow$ {\rm(d)} are clear. For the converses of these, if (d) holds, then by Corollary \ref{c:C.for.fcts:0}, $f$ is $\vY_{\I\ltimes K}[T]$-measurable, and hence, by Lemma \ref{l:ups.T} (ii), is $\vY_{\I\ltimes J}[T]$-measurable, giving (c).
If $(\I\,\wh{\times}\, Y)[T]\sq \vY$ and (c) holds, then by Corollary \ref{c:C.for.fcts:0}, $f$ is $\vY_{\I\ltimes J}[T]$-measurable, and hence, by Lemma \ref{l:ups.T} (i), is $\vY[T]_{\{\e\}\ltimes T_0}$-measurable, giving (b). Next we check that {\rm(a)} and {\rm(b)} are equivalent.

(a) $\Rightarrow$ (b) Assume that $g\in\F_T$ is $\vY$-measurable and $f_x=_{J_x}g_x$ for $x\in X$. By Lemma \ref{l:eta.skew} (iii) (taking $\I$, $\vY$ there to be $\{\e\}$, $\vY[T]$, respectively),
$f$ is $\vY[T]_{\{\e\}\ltimes J}$-measurable.

(b) $\Rightarrow$ (a)
Suppose $(f_x)_{x\in X}$ is a $\vY[T]_{\{\e\}\ltimes J}$-measurable process. By Lemma \ref{P10i} (ii), there is a real-valued $\vY[T]$-measurable function $g$ on $X\times Y$ such that $g =_{\{\e\}\ltimes J} f$. For each $x\in X$, $g_x$ is $T$-measurable by Proposition \ref{p:C.for.fcts:0} (i) (with $\I=\{\e\}$), giving $g\in\F_T$.
From $g =_{\{\e\}\ltimes J} f$ we get that for each $x\in X$, $g_x =_{J_x} f_x$, as required.

The implication {\rm(a)} $\Rightarrow$ {\rm(e)} is obvious since $\vY\sq\Omega$, so there remains to show {\rm(e)} $\Rightarrow$ {\rm(d)}. For this we apply {\rm(a)} $\Rightarrow$ {\rm(c)} to the structures obtained by replacing $T_x$, $J_x$ and $\vY$ with $(T_x)_{K_x}$, $K_x$ and $\Omega$, respectively. For the hypothesis corresponding to $\vY\sq\I\ltimes T$, we have $\Omega\sq \vY_{\I\ltimes K}\sq \I\ltimes T_K$. We have $f\in\F_T\sq \F_{T_K}$, so (a) $\Rightarrow$ (c) in this instance gives that (e) implies that $f$ is $\Omega_{\I\ltimes K}$-measurable, and therefore is $\vY_{\I\ltimes K}$-measurable since $\Omega_{\I\ltimes K} \sq (\vY_{\I\ltimes K})_{\I\ltimes K} = \vY_{\I\ltimes K}$. Thus, (d) holds.
\end{proof}

\begin{rem}\label{r:eta.mod}
Assume that $(\I\,\wh{\times}\, Y)[T]\sq \vY$ holds and there exists an $\eta=(\eta_x)_{x\in X}$, where each $\eta_x\colon\LL^0(T_x)\to \LL^0(T_x)$ satisfies $(l1)$ and $(l2)$ (modulo $K_x$) and $\eta$ is a $2$-marginal for  $\LL^0(\vY)$.
Then starting with $f\in\F_T$ being $\vY_{\I\ltimes K}$-measurable as in (d), we directly get $g\in \F_T$ as in (a) by taking $g=\eta_\bullet(f)$.

\begin{proof}
By $(l1)$ for the $\eta_x$, $g\in\F_T$ and $f_x =_{K_x} g_x$ for all $x\in X$. There remains to show $g$ is $\vY$-measurable.
By Lemma \ref{P10i} (ii), there is an $h\in\LL^0(\vY)$ such that $f =_{\I\ltimes K} h$.
There is an $N\in \I$ such that for $x\in N^c$, we have $h_x\in\LL^0(T)$ and $f_x =_{K_x} h_x$. By redefining $h$ to be zero on $N\times Y$, we get $h\in \F_T$. The function $\eta_\bullet(h)$ is $\vY$-measurable since $\eta$ is a $2$-marginal for  $\LL^0(\vY)$. By $(l2)$, when $x\in N^c$, $\eta_x(f_x)=\eta_x(h_x)$. Hence $\eta_\bullet(f)$ and $\eta_\bullet(h)$ have the same restriction to $N^c\times Y$. By $(\I\,\wh{\times}\, Y)[T]\sq \vY$, it follows that $g=\eta_\bullet(f)$ is $\vY$-measurable.
\end{proof}
\end{rem}

\begin{cor}\label{c:mm1}
Let $(X,\vS,\mu)$, $(Y,T,\nu_x)_{x\in X}$, and $(X\times Y,\vY,\ups)$ be probability spaces satisfying $[C]$. Write
Consider the following statements.
\begin{enumerate}
\item[\rm(a)]
$(Q_x)_{x\in{X}}$ has a $\vY$-measurable  modification $(U_x)_{x\in{X}}$ over $(Y,T)$.

\item[\rm(b)]
$(Q_x)_{x\in{X}}$ is $\vY[T]_{\{\e\}\ltimes T_0}$-measurable.

\item[\rm(c)]
$(Q_x)_{x\in{X}}$ is $\vY_{\mathfrak{N}(\mu,\nu)}$-measurable.

\item[\rm(d)]
$(Q_x)_{x\in{X}}$ is $\vY_{\mathfrak{N}}$-measurable.

\item[\rm(e)]
$Q_x =_{\wh{\nu}_x} U_x$, $x\in X$, for some $\wh{\vY}$-measurable $U\colon X\times Y\to\R$, where $U_x$ is $\wh{\nu}_x$-measurable.
\end{enumerate}
The following implications hold: {\rm(a)} $\Leftrightarrow$ {\rm(b)} $\Rightarrow$ {\rm(c)} $\Leftrightarrow$ {\rm(d)}, {\rm(a)} $\Rightarrow$ {\rm(e)} $\Rightarrow$ {\rm(d)}. If $(\vS_0\,\wh{\times}\, Y)[T]\sq \vY$, then {\rm(c)} $\Rightarrow$ {\rm(b)} and therefore all statements are equivalent.
\end{cor}

\begin{proof}
This follows immediately from the theorem, taking $\Omega = \wh{\vY}$.
\end{proof}

\begin{rem}\label{r:d.e}
The equivalence of (d) and (e) Corollary \ref{c:mm1} holds under the milder assumption that $\vS_0\,\dot{\times}\, Y\sq \vY$. Since $\vY_\mathfrak{N} = \wh{\vY}_\mathfrak{N}$, for the equivalence of (d) and (e), we can apply Theorem \ref{t:mm1} to the spaces $(X,\vS,\mu)$, $(Y,\wh{T}_x,\wh{\nu}_x)_{x\in X}$, $(X\times Y,\wh{\vY},\wh{\upsilon})$, where $\wh{T}_x$ is the domain of $\wh{\nu}_x$. ($[C]$ holds by Proposition \ref{p:C.for.fcts} (v).) Now the assumption $(\I\,\wh{\times}\, Y)[T]\sq \vY$ holds since $\vS_0\,\wh{\times}\, Y\sq \wh{\vY}$ follows from $\vS_0\,\dot{\times}\, Y\sq \vY$.
Some assumption is needed because if we take $(X,\vS,\mu)=(Y,T,\nu)=[0,1]$ with Lebesgue measure, $(X\times Y,\vY,\ups)=$ the trivial probability having $X\times Y$ as its only nonempty measurable set then all the measures are complete. The function $Q=\chi_{\{0\}\times Y}$ is $\vY_\mathfrak{N}$-measurable since $\{0\}\times Y\in\mathfrak{N}$, but if $U_x =_{\nu} Q_x$ for all $x$, then $U$ must take both the values $0$ and $1$ and hence is not $\vY$-measurable.
\end{rem}

\begin{rem}\label{r:eta.mod.2}
For $0\leq p\leq\infty$, we could stipulate in clause (a) of Corollary \ref{c:mm1} that $(U_x)_{x\in X}$ is an $\LL^p$-process and in clauses (b)--(d) that $(Q_x)_{x\in X}$ is an $\LL^p$-process, since a modification of an $\LL^p$-process is an $\LL^p$-process. Then, similar to Remark \ref{r:eta.mod}, if $(\vS_0\,\wh{\times}\, Y)[T]\sq \vY$ holds and there exists an $\eta=(\eta_x)_{x\in X}$, where each $\eta_x\colon\LL^p(T_x)\to \LL^p(T_x)$ satisfies $(l1)$ and $(l2)$ and $\eta$ is a $2$-marginal for  $\LL^{0,p}(\vY)$, then for any process $(Q_x)_{x\in X}$ which is  $\vY_\mathfrak{N}$-measurable, we get a $\vY$-measurable modification $(U_x)_{x\in X}$ by taking $U_x = \eta_\bullet(Q_x)$. In the proof of Proposition 4.1 of \cite{mu23}, liftings for $\LL^\infty$ were used to show that (d) implies (e).
\end{rem}

\begin{rem}\label{r:mu}
In Proposition 4.1 of \cite{mu23}, (using our notation) a product r.c.p.\ $(\nu_x)_{x\in X}$ for a measure $\ups$ is given (see Example \ref{ex:prod.r.c.p.}), and a process on $(Y,T)$ is defined to be measurable if it is $\wh{\ups}$-measurable.
The proof that a $\vY_\mathfrak{N}$-measurable process has a measurable modification produces a family of variables $(U_x)_{x\in X}$ where $U_x$ is $\wh{\nu}_x$-measurable.
By our definition, this is not a process because the $U_x$ are not all measurable with respect to the same $\s$-algebra. Nevertheless, it fits into the framework of Corollary \ref{c:mm1} as statement (e) (see Remark \ref{r:d.e}), and was the motivation for including statement (e).
\end{rem}

\begin{cor}\label{mm2}
Let $(X,\vS,\mu)$, $(Y,T,\nu_x)_{x\in X}$, and $(X\times Y,\vY,\ups)$ be probability spaces satisfying $[C]$. Let $(Q_x)_{x\in{X}}$ be a process over $(Y,T)$ with values in $\R$.
The following statements are equivalent:
\begin{enumerate}
\item
$(Q_x)_{x\in{X}}$ has a $\vY_\mathfrak{N}$-measurable modification $(U_x)_{x\in{X}}$ over $(Y,T)$.

\item
$(Q_x)_{x\in{X}}$ is $\vY_\mathfrak{N}$-measurable.
\end{enumerate}
\end{cor}

\begin{proof}
For (i) implies (ii), let $U$ be a $\vY_\mathfrak{N}$-measurable modification of $(Q_x)_{x\in{X}}$. Then $U$ is $\vY_\mathfrak{N}[T]$-measurable by
Corollary \ref{c:C.for.fcts:0}. The restriction of $\ups_\mathfrak{N}$ to $\vY_\mathfrak{N}[T]$ satisfies $[C]$.
(By Proposition \ref{p:VMLa:2} (i), $\ups_{\mathfrak{N}(\mu,\nu)}$ satisfies $[C]$, and hence so does its restriction to $\vY_{\mathfrak{N}(\mu,\nu)}[T] = \vY_{\mathfrak{N}}[T]$ (Lemma \ref{l:ups.T} (ii)) which is the same as the restriction of $\ups_\mathfrak{N}$ to that $\s$-algebra.) We may therefore apply the theorem to that restriction in the place of $\ups$ to conclude that
$Q$ is $(\vY_{\mathfrak{N}}[T])_\mathfrak{N}$-measurable. Since $(\vY_{\mathfrak{N}}[T])_\mathfrak{N} \sq (\vY_{\mathfrak{N}})_\mathfrak{N} = \vY_\mathfrak{N}$, (ii) follows.

(ii) $\Rightarrow$ (i) is immediate by taking $U=Q$.
\end{proof}

\section{Fubini type products}
\label{s:fubini}

Let $(X,\vS,\mu)$, $(Y,T,\nu)$, $(X\times Y,\vY,\ups)$ be fixed probability spaces, and let $p\in[0,\infty]$.
Also fix $\de\colon\LL^p(\mu)\to\R^X$ and $\eta\colon\LL^p(\nu)\to\R^Y$.
In this section we work with real scalars assuming throughout that $\K=\R$.

\begin{df} \label{A}
Define
$\de\odot\eta\colon \LL^p(\ups)\rightarrow\R^{X\times{Y}}$ by the formula $\de\odot\eta := \de^\bullet\circ \eta_\bullet$, i.e.,
\[
(\de\odot\eta)(f)(x,y):= \de([\eta_\bullet(f)]^y)(x)\;\;\mbox{if}\;[\eta_\bullet(f)]^y\in\LL^p(\mu)\;\mbox{and}\;(\gamma\odot\delta)(f)(x,y):=0\;\mbox{otherwise}
\]
for all $f\in\LL^p(\ups)$ and $(x,y)\in{X}\times{Y}$.
\end{df}

\begin{rem}
We can consider the composition $\eta\odot_t\de := \eta_\bullet\circ \de^\bullet$ instead of $\de^\bullet\circ \eta_\bullet$.
It is easy to derive from results about $\de\odot\eta$ corresponding results for the product $\eta\odot_t\de$. By \cite{mms15}, Theorem
3.2, when $\ups=\mu\,\wh{\otimes}\,\nu$ we have $\de\odot\eta\not=\eta\odot_t\de$ if the spaces
$(X,\vS,\mu)$ and $(Y,T,\nu)$ are not purely atomic, i.e. practically in all
interesting cases.
\end{rem}

\begin{df}[Cf.\ \cite{mms14}, Definition 3.2] \label{B1}
We say that $\eta$ \emph{generates $\LL^p(\mu)$-sections}, if $[\eta_\bullet(f)]^y\in\LL^p(\mu)$ for all $f\in\LL^p(\ups)$ and $y\in Y$.
If $\eta$ is a primitive lifting for $\LL^p(\nu)$ which is a 2-marginal with respect to $\LL^p(\ups)$ and generates $\LL^p(\mu)$-sections, then we say $\eta$ is $\mu$-\emph{smooth}.
\end{df}

The next result corresponds to \cite{mms14}, Lemma 3.1, concerning primitive liftings for sets, and has a similar proof, which we enclose for completeness.

\begin{lem}\label{00}
Assume that $[C]$ holds, $\mu$ is complete, $(\vS_0\,\wh{\times}\,Y)[T]\sq \vY$, and $\eta$ is a primitive lifting for $\LL^p(\nu)$ which is a $2$-marginal with respect to $\LL^p(\ups)$. Then $\eta$ is $\mu$-smooth if and only if
there exists a map $\varphi: \LL^p(\ups)\rightarrow\LL^p(\ups)$ satisfying $(l1)$, such that for all $f\in\LL^p(\ups)$, there exists an $N_f\in\vS_0$ such that for all $x\notin N_f$, $\eta([\varphi(f)]_{x})=[\varphi(f)]_{x}$, and $[\varphi(f)]^y\in\LL^p(\mu)$ for every $f\in\LL^p(\ups)$ and $y\in Y$. In this statement we can also take each $N_f$ to be empty.
\end{lem}

\begin{proof}
Consider a map $\varphi\colon \LL^p(\ups)\to\LL^p(\ups)$ satisfying $(l1)$ and such that for all $f\in\LL^p(\ups)$ there exists a set $N_f\in\vS_0$ with $\eta([\varphi(f)]_x) = [\varphi(f)]_x$ for all $x\notin{N}_f$. There is such a map by Corollary \ref{VM19b} (and we can ask that $N_f=\e$ for all $f$). By  $[C]$ and $(l1)$, there is an $M_f\in \vS_0$ such that $f_x,[\varphi(f)]_x\in\LL^p(\nu_x)$ and $f_x =_{\nu} [\varphi(f)]_x$ for all $x\notin M_f$. For $y\in{Y}$ and $x\notin N_f\cup M_f$ we get
\[
[\eta_\bullet(f)]^y(x) = \eta(f_x)(y) = \eta([\varphi(f)]_x)(y) = [\varphi(f)]_x(y) = [\varphi(f)]^y(x).
\]
Since $\mu$ is complete, it follows that $\{[\eta_\bullet(f)]^y \not= [\varphi(f)]^y\}\in\vS_0$. Thus, $[\varphi(f)]^y\in\LL^p(\mu)$ if and only if $[\eta_\bullet(f)]^y\in\LL^p(\mu)$.
\end{proof}

\begin{lem}\label{01}
The map $\de\odot\eta:\LL^p(\ups)\to\R^{X\times{Y}}$
has the following properties. In {\rm(iii)}, {\rm(v)}, {\rm(vi)}, assume that $[C]$ holds. In {\rm(vi)}, {\rm(vii)}, assume also that $\vS_0\,\dot{\times}\,Y\sq\vY$ holds.
\begin{enumerate}
\item
$\de([\eta_\bullet(f)]^y)=[(\de\odot\eta)(f)]^y$ for all $f\in\LL^p(\ups)$ and all $y\in{Y}$ with $[\eta_\bullet(f)]^y\in\LL^p(\mu)$.

\item
$[(\de\odot\eta)(f)]^y=\de([(\de\odot\eta)(f)]^y)$ for all $f\in\LL^p(\ups)$ and all $y\in{Y}$.

\item
If $\mu$ is complete, $\de\odot\eta$ satisfies $(l2)$.

\item
If $\de$ and $\eta$ fix the constant function with value $k$ on $X$ and $Y$, respectively, then $\de\odot\eta$ fixes the constant function with value $k$ on $X\times Y$.

\item
If $\de$ and $\eta$ are homogeneous then $\de\odot\eta\in\de\otimes\eta$.

\item
If $\de\in V^p_\R(\mu)$, $\eta\in V^p_\R(\nu)$, and $\eta$ generates $\LL^p(\mu)$-sections, then $\de\odot\eta$ is linear: $(\de\odot\eta)(af+g)=a(\de\odot\eta)(f)+(\de\odot\eta)(g)$
for all $a\in\R$ and $f,g\in\LL^p( \ups)$.

\item
If $\de\in\Lambda^\infty(\mu)$, $\eta\in\Lambda^\infty(\nu)$, and $\eta$ generates $\LL^\infty(\mu)$-sections, then $\de\odot\eta$ is multiplicative: $(\de\odot\eta)(fg)=(\de\odot\eta)(f)(\de\odot\eta)(g)$
for all $f,g\in\LL^\infty(\ups)$.

\item
If $\de$ and $\eta$ are positive then so is $\de\odot\eta$.
\end{enumerate}
\end{lem}

Recall from Definition \ref{E5} that positivity of a primitive lifting means that it  takes nonnegative values on nonnegative functions.

\begin{proof}
(i) follows by Definition \ref{A}.

(ii) If $f\in\LL^p(\ups)$  and $y\in{Y}$ then
$[(\de\odot\eta)(f)]^y$ equals either $\de\left([\eta_\bullet(f)]^y\right)$ if $[\eta_\bullet(f)]^y\in\LL^p(\mu)$, or
$0$ otherwise. Either way we get $[(\de\odot\eta)(f)]^y=
\de\left([(\de\odot\eta)(f)]^y\right)$.

(iii) By Lemma \ref{l:VM10a.2}, from $f=_{\ups}g$ and completeness of $\mu$ we get that for all $y$, $[\eta_\bullet(f)]^y =_\mu [\eta_\bullet(g)]^y$, i.e., $\mu([\eta_\bullet(f)]^y \,\triangle\, [\eta_\bullet(g)]^y)=0$, but the sets $[\eta_\bullet(f)]^y$, $[\eta_\bullet(g)]^y$ might not be measurable. If these sets are measurable, it follows that $[(\de\odot\eta)(f)]^y = \de([\eta_\bullet(f)]^y) = \de([\eta_\bullet(g)]^y) = [(\de\odot\eta)(g)]^y$. Otherwise we have $[(\de\odot\eta)(f)]^y = 0 = [(\de\odot\eta)(g)]^y$.

(iv) Obvious from the definitions.

(v) If $g\in\LL^p(\mu)$ and $h\in\LL^p(\nu)$, then
$[\eta_\bullet(g\otimes h)]^y = [g\otimes\eta(h)]^y = g\eta(h)(y) \in \LL^p(\mu)$ by Lemma \ref{VM10a} (ii), so using (i) we get
$[(\de\odot\eta)(g\otimes{h})]^y = \de(\left[\eta_\bullet(g\otimes{h})\right]^y) = \eta(h)(y)\de(g) = [\de(g)\otimes\eta(h)]^y$.

(vi) Let $a\in\R$.
Since $\de\odot\eta$ satisfies $(l2)$ by (iii), it follows from Proposition \ref{p:C.for.fcts} (iv) and Corollary \ref{c:sec} (here we use $\vS_0\,\dot{\times}\,Y\sq\vY$) that we may assume that $f_x,g_x\in\LL^p(\nu)$ for all $x\in X$.
Then from Lemma \ref{VM10a} (iv) it follows that $\eta_\bullet(af+g) = a\eta_\bullet(f) + \eta_\bullet(g)$. Since $\eta$ generates $\LL^p(\mu)$-sections, we then get for any $y\in Y$ that $[(\de\odot\eta)(af+g)]^y = \de([\eta_\bullet(af+g)]^y) = a\de([\eta_\bullet(f)]^y) + \de([\eta_\bullet(g)]^y) = a[(\de\odot\eta)(f)]^y + [(\de\odot\eta)(g)]^y$.

(vii) As in (vi), we may assume that $f_x,g_x\in \LL^\infty(\nu)$ for all $x\in X$. For all $f\in\LL^\infty(\ups)$ and all $y\in Y$ applying Lemma \ref{VM10a} (iv)(c) and the assumption that $\eta$ generates $\LL^\infty(\mu)$-sections, we get
\begin{align*}
[(\de\odot\eta)(fg)]^y
& = [\de([\eta_\bullet(fg)]^y) = \de([\eta_\bullet(f)]^y[\eta_\bullet(g)]^y) \\
& = \de([\eta_\bullet(f)]^y)\de([\eta_\bullet(g)]^y) = [(\de\odot\eta)(f)]^y[(\de\odot\eta)(g)]^y
\end{align*}

(viii) Let $f\in\LL^p(\ups)$ with $f\geq{0}$ and let $y\in{Y}$. Since $\eta$ is positive, Lemma \ref{VM10a} (iii) gives $\eta_\bullet(f)\geq 0$. Then for $y\in{Y}$, $[\eta_\bullet(f)]^y\geq 0$, so if $[\eta_\bullet(f)]^y\in\LL^p(\mu)$ we get $[(\de\odot\eta)(f)]^y=\de([\eta_\bullet(f)]^y)\geq{0}$ since $\de$ is positive. If $[\eta_\bullet(f)]^y\notin\LL^p(\mu)$ then $[(\de\odot\eta)(f)]^y = 0$.
\end{proof}

\begin{rem}\label{C}
\rm{There is now an obvious question, whether when $\de$ and $\eta$ are primitive liftings for $\LL^p(\mu)$ and $\LL^p(\nu)$, respectively, the map $\de\odot\eta$ is a primitive lifting for $\LL^p(\ups)$. As it is shown in \cite{mms14}, Remark 5.1, when $\ups = \mu\,\wh{\otimes}\,\nu$, if $\de\in\Lambda^\infty(\mu)$ and $\eta\in\Lambda^\infty(\nu)$ it may happen that
$\de\odot\eta$ does not satisfy (l1).}
\end{rem}

\begin{df}[Cf.\ \cite{mms14}, Definition 4.2]\label{D}
The pair $(\de,\eta)$ is \emph{$\LL^p(\ups)$-preserving} if $f\in\LL^p(\ups)$ implies $(\de\odot\eta)(f)\in\LL^p(\ups)$.
\end{df}

\begin{prop}\label{10a}
If $\de$ is a $1$-marginal and $\eta$ is a $2$-marginal with respect to $\LL^p(\ups)$, then the pair $(\de,\eta)$ is $\LL^p(\ups)$-preserving.
\end{prop}

\begin{proof} Let $f\in\LL^p(\ups)$ be arbitrary. Then $\eta_\bullet(f)\in\LL^p(\ups)$ by the marginality of $\eta$ and therefore
$(\de\odot\eta)(f) = \de^\bullet(\eta_\bullet(f))\in\LL^p(\ups)$ by the marginality of $\de$.
\end{proof}

\begin{df}[Cf.\ \cite{mms14}, Definition 3.2] \label{Da}
$\eta$ is a \emph{weak $2$-marginal} with respect to $\LL^p(\ups)$ if for each $f\in\LL^p(\ups)$ there is an $M_f\in T_0$ such that $[\eta_\bullet(f)]^y, f^y\in \LL^p(\mu)$ and $[\eta_\bullet(f)]^y=_{\mu}f^y$ for all $y\notin M_f$.
\end{df}

\begin{rem}\label{r:Da}
When $[C]$ and $[\ov{C}]$ hold, and $(\vS_0\,\wh{\times}\, Y)[T]\sq\vY$, a $2$-marginal with respect to $\LL^p(\ups)$ is also a weak $2$-marginal with respect to $\LL^p(\ups)$.

\begin{proof}
By Corollary \ref{VM19b}, $\eta_\bullet$ satisfies $(l1)$ on $\LL^p(\ups)$.
Then for each $f\in\LL^p(\ups)$ we have $\eta_\bullet(f)\in\LL^p(\ups)$ and $\eta_\bullet(f) =_\ups f$. Then by $[\ov{C}]$, there is an $M_f\in T_0$ such that $[\eta_\bullet(f)]^y, f^y\in \LL^p(\mu)$ and $[\eta_\bullet(f)]^y=_{\mu}f^y$ for all $y\notin M_f$.
\end{proof}
\end{rem}

Under some mild assumptions, $\de\odot\eta$ satisfies $(l1)$ if and only if $(\de,\eta)$ is $\LL^p(\ups)$-preserving and $\eta$ is a weak $2$-marginal with respect to $\LL^p(\ups)$. The following proposition is modelled on Proposition 4.3 of \cite{mms14} which establishes that equivalence in the context of liftings for sets with $\ups=\mu\,\wh{\otimes}\,\nu$. As pointed out in the footnote to \cite{bms4}, Remark 5.7 (b), the proof given there has a gap. We take the opportunity in Corollary \ref{c:E} to fill that gap. The new proof also eliminates completeness assumptions on the measures.

\begin{prop}\label{E}
Assume $[\ov{C}]$ holds, $\de$ is a primitive lifting for $\LL^p(\mu)$, $\eta\colon \LL^p(\nu)\to \R^Y$.
\begin{enumerate}
\item
If $(\de,\eta)$ is $\LL^p(\ups)$-preserving and $\eta$ is a weak $2$-marginal with respect to $\LL^p(\ups)$ then $\de\odot\eta$ satisfies $(l1)$.

\item
Assume that $\vS\,\dot{\times}\,Y\sq \vY$. Suppose that either {\rm(a)} $X$ is not an atom, or {\rm(b)}
$X$ is an atom, $[C]$ holds, $X\,\dot{\times}\,T\sq \vY$, and $\eta$ satisfies $(l2)$.
If $\de\odot\eta$ satisfies $(l1)$ then $(\de,\eta)$ is $\LL^p(\ups)$-preserving and $\eta$ is a weak $2$-marginal with respect to $\LL^p(\ups)$.
\end{enumerate}
\end{prop}

\begin{proof} (i)
The argument in \cite{mms14} adapts easily to the present setting.
Assume $(\de,\eta)$ is $\LL^p(\ups)$-preserving and $\eta$ is a weak $2$-marginal with respect to $\LL^p(\ups)$. Write $\varphi:=\de\odot\eta$. Let $f\in\LL^p(\ups)$. Then $\varphi(f)\in\LL^p(\ups)$ and for some $M\in T_0$, $[\eta_\bullet(f)]^y,f^y\in \LL^p(\mu)$ and $[\eta_\bullet(f)]^y=_{\mu}f^y$ when $y\notin M$.
Let $F=\{\varphi(f)\not=f\}$. We have $F\in\vY$. When $y\notin M$,
$[\varphi(f)]^y=\de([\eta_\bullet(f)]^y)= \de(f^y)=_{\mu}f^y$, where we used that $\de$ is a primitive lifting for $\LL^p(\mu)$, giving $\mu(F^y) = \mu\{[\varphi(f)]^y\not=f^y\} = 0$.
It follows from $[\ov{C}]$ that $\ups(F) = 0$. Thus, $\varphi(f)=_{\ups}f$ and $(l1)$ holds.

(ii) Assume that $(l1)$ holds for $\varphi=\de\odot\eta$. Then in particular the pair $(\de,\eta)$ is $\LL^p(\ups)$-preserving.
To prove that $\eta$ is a weak $2$-marginal for $\LL^p(\ups)$, let $f\in\LL^p(\ups)$.

\begin{clm}\label{claim:f}
For each $N\in\vS_0$, $\nu^*(G_N)=0$, where $G_N=\{y\in Y: [\eta_\bullet(f)]^y\restr N\notin \LL^0(\mu\restr N)\}$.
\end{clm}

\begin{proof}
Let $N\in\vS_0$. Consider the function $g$ which agrees with $f$ on $N\times Y$ and is identically equal to $1$ on $N^c\times Y$. This function $g$ belongs to $\LL^p(\ups)$ since $\vS\,\dot{\times}\,Y\sq \vY$.
By $[\ov{C}]$ and $(l1)$ for $\varphi$, let $M\in T_0$ satisfy that $y\notin M$ implies $[\varphi(g)]^y,g^y\in\LL^p(\mu)$ and $[\varphi(g)]^y =_\mu g^y =_\mu 1$.

For $y\in G_N$, we have $[\eta_\bullet(g)]^y\restr N = [\eta_\bullet(f)]^y\restr N\notin \LL^0(\mu\restr N)$ implying that $[\eta_\bullet(g)]^y\notin \LL^p(\mu)$ and hence $[\varphi(g)]^y = 0$, so that $y\in M$. Thus $G_N\sq M$ and the claim follows.
\end{proof}

\begin{clm}\label{claim:c}
$\nu^*(K)=0$, where $K=\{y\in Y:[\eta_\bullet(f)]^y\notin \LL^p(\mu)\}$.
\end{clm}

\begin{proof}
We treat two cases, corresponding to assumptions (a) and (b).

\m

\n (a) $X$ is not an atom.

\m

This is similar to the proof of Claim \ref{claim:f}.
Write $X=A_0\cup A_1$ where $A_0$ and $A_1$ are disjoint sets of positive measure. Then $K=K_0\cup K_1$, where for $i=0,1$, $K_i = \{y\in K : [\eta_\bullet(f)]^y\restr A_i\notin \LL^p(\mu\restr A_i)\}$.
It will suffice to show that $\nu^*(K_i)=0$ for $i=0,1$. We show $\nu^*(K_0)=0$, the case $i=1$ being similar. Consider the function $g$ which agrees with $f$ on $A_0\times Y$ and is identically equal to $1$ on $A_1\times Y$. We have $g\in\LL^p(\ups)$.
By $[\ov{C}]$ and $(l1)$ for $\varphi$, let $M\in T_0$ satisfy that $y\notin M$ implies $[\varphi(g)]^y,g^y\in\LL^p(\mu)$ and $[\varphi(g)]^y =_\mu g^y$ which is not $=_\mu 0$ since it equals $1$ on $A_1$.

For $y\in K_0$, $[\eta_\bullet(g)]^y\restr A_0 = [\eta_\bullet(f)]^y\restr A_0\notin \LL^p(\mu\restr A_0)$ implying that $[\eta_\bullet(g)]^y\notin \LL^p(\mu)$ and hence $[\varphi(g)]^y=0$. Thus, $K_0\sq M$ and hence $\nu^*(K_0)=0$, as desired.

\m

\n (b) $X$ is an atom, $[C]$ holds, $X\,\dot{\times}\,T\sq \vY$, and $\eta$ satisfies $(l2)$.

\m

By Example \ref{ex:atom.factor}, there is a $w\in\LL^0(\nu)$ such that the $(X\,\dot{\times}\,T)$-measurable function $g\colon X\times Y\to\R$ given by $g(x,y)=w(y)$ satisfies $f=_\ups g$.
By $[C]$, there exists a set $P\in \vS_0$, such that $x\notin P$ implies $f_x\in \LL^p(\nu)$, and $f_x =_\nu g_x = w$. Since $\eta$ satisfies $(l2)$,
$[\eta_\bullet(f)]_x = \eta(w)$ when $x\in P^c$, so that for any $y\in Y$, $[\eta_\bullet(f)]^y\restr P^c = \eta(w)(y)$ is constant, in particular, $[\eta_\bullet(f)]^y\restr P^c\in\LL^p(\mu)$. For $y\notin G_P$ we have $[\eta_\bullet(f)]^y\restr P\in \LL^0(\mu)$ and hence $[\eta_\bullet(f)]^y\in \LL^p(\mu)$. Thus, $K\sq G_P$ and Claim \ref{claim:c} follows by Claim \ref{claim:f}.
\end{proof}

With this claim, the proof proceeds by adapting the argument in \cite{mms14}.
From $[\ov{C}]$ and $(l1)$ we get $M\in T_0$, such that $y\notin M$ implies $[\varphi(f)]^y, f^y\in\vS$, and $[\varphi(f)]^y=_{\mu}f^y$. By Claim \ref{claim:c}, we can also take $M$ so that $y\notin M$ implies $[\eta_\bullet(f)]^y\in \LL^p(\mu)$.
For $y\in M^c$ we have $[\eta_\bullet(f)]^y=_{\mu}\de([\eta_\bullet(f)]^y)=
[\varphi(f)]^y=_{\mu}f^y$,
and therefore $\eta$ is a weak $2$-marginal with respect to $\LL^p(\ups)$.
\end{proof}

We similarly have the following corollary of the preceding proof, giving a version of the proposition for maps on the family of measurable sets instead of on $\LL^p$. Except for the new clause (c) in (ii), it follows by a straightforward modification to the proof of Proposition \ref{E}.

For set maps $\de\in \mathcal{P}(X)^\vS$, $\eta\in\mathcal{P}(Y)^T$ the notions in this section have meanings analogous to the ones given above. For $E\sq X\times Y$, $\de^\bullet(E)=\{(x,y)\in X\times Y: x\in \de(E^y),\,E^y\in\vS\}$ and
$\eta_\bullet(E)=\{(x,y)\in X\times Y: y\in \de(E_x),\,E_x\in T\}$. Then the pair $(\de,\eta)$ is $\vY$-preserving if $E\in\vY$ implies $(\de\odot\eta)(f):=\de^\bullet(\eta_\bullet(E))\in\vY$.
We say $\eta$ is a weak $2$-marginal with respect to $\vY$ if for each $E\in\vY$ there is an $M_E\in T_0$ such that $[\eta_\bullet(E)]^y, E^y\in \vY$ and $[\eta_\bullet(E)]^y=_{\mu}E^y$ for all $y\notin M_E$.

\begin{cor}\label{c:E}
Assume $[\ov{C}]$ holds, $\vS\,\dot{\times}\,Y\sq \vY$, $\de$ is a primitive lifting for $\mu$, and $\eta\in \mathcal{P}(Y)^T$.
\begin{enumerate}
\item
If $(\de,\eta)$ is $\vY$-preserving and $\eta$ is a weak $2$-marginal with respect to $\vY$ then $\de\odot\eta$ satisfies $(L1)$.

\item
Suppose that one of the following holds: {\rm(a)}
$X$ is not an atom; {\rm(b)}
$X$ is an atom, $[C]$ holds, $X\,\dot{\times}\,T\sq \vY$, and $\eta$ satisfies $(L2)$;
{\rm(c)} $X$ is an atom, $X\,\dot{\times}\,T\sq \vY$, $\eta$ satisfies $(V)$ and $\mu$ is complete.
If $\de\odot\eta$ satisfies $(L1)$ then $(\de,\eta)$ is $\vY$-preserving and $\eta$ is a weak $2$-marginal with respect to $\vY$.
\end{enumerate}
\end{cor}

\begin{proof} Repeat the proof of Proposition \ref{E}, making appropriate changes to the notation: in the whole proof, replace $\LL^p(\ups)$ by $\vY$; in (i), replace $f$ by a set $E\in\vY$, and take $F=\varphi(E)\,\triangle\,E$; in (ii), fix a set $E\in\vY$, replacing $f$ by $E$ in the rest of the proof, replace the restriction symbol $\restr$ by the intersection symbol $\cap$, replace $\LL^p(\mu)$, $\LL^p(\mu\restr N)$ and $\LL^p(\mu\restr A_i)$ by $\vS$ and $\LL^p(\nu)$ by $T$ (for any exponent $p$), and replace the constant functions $0$ and $1$ by the sets $\e$ and $X$.
Claim \ref{claim:f} now reads as follows.

\begin{clm}\label{claim:f.sets}
For each $N\in\vS_0$, $\nu^*(G_N)=0$, where $G_N=\{y\in Y: [\eta_\bullet(E)]^y\cap N\notin \vS\}$.
\end{clm}

In the proof of Claim \ref{claim:f}, replace $g$ by $F=E\cup (N^c\times Y)$. Claim \ref{claim:c} now reads as follows.

\begin{clm}\label{claim:c.sets}
$\nu^*(K)=0$, where $K=\{y\in Y:[\eta_\bullet(E)]^y\notin \vS\}$.
\end{clm}

In the proof of Claim \ref{claim:c}, in (a), replace $g$ by $F=E\cup (A_1\times Y)$. In (b), and also for (c), the fact that $X$ is an atom and $[C]$ holds with $X\,\dot{\times}\,T\sq \vY$, gives us (from Example \ref{ex:atom.factor} again)
an $A_E\in T$ and an $N_E\in T_0$ such that $E =_\ups F:=X\times A$ ($F$ replaces $g$) and $y\notin N_E$ implies $E^y\in \vS$ and $\mu(E^y)=\chi_{A_E}(y)$.
The value $\eta(w)(y)$ becomes $P^c$ or $\e$ depending on whether $y\in \eta(A_E)$.

\m

That leaves the proof of Claim \ref{claim:c.sets} under the assumptions in (c), namely that $X$ is an atom, $\eta$ satisfies $(V)$ and $\mu$ is complete. For that, we will make use of the following claim.

\begin{clm}\label{claim:a}
If $y\in A_E\sm N_E$ then $[\eta_\bullet(E)]^y \in \vS$ and $\mu([\eta_\bullet(E)]^y) = 1$.
\end{clm}

\begin{proof}
When $y\in A_E\sm N_E$, $[\varphi(E)]^y =_\mu E^y$, so $[\varphi(E)]^y$ has measure one since $y\in A_E\sm N_E$. It follows that $[\eta_\bullet(E)]^y\in\vS$ and $\de([\eta_\bullet(E)]^y) = [\varphi(E)]^y$ has measure one. Therefore $\mu([\eta_\bullet(E)]^y)=1$.
\end{proof}

Since $\eta$ satisfies $(V)$, so does $\eta_\bullet$, and then we get $[\eta_\bullet(E)]^y\cap [\eta_\bullet(E^c)]^y = \e$ for all $y$. By Claim \ref{claim:a} applied to $E^c$, $y\in A_{E^c}\sm N_{E^c}$ implies $\mu([\eta_\bullet(E^c)]^y) = 1$, and then by completeness of $\mu$ we get $\mu([\eta_\bullet(E)]^y) = 0$, in particular $[\eta_\bullet(E)]^y\in\vS$.
We also have that
$y\in A_E\sm N_E$ implies $[\eta_\bullet(E)]^y\in\vS$ and $\mu([\eta_\bullet(E)]^y) = 1$, so it follows that $(A_E\sm N_E)\cap (A_{E^c}\sm N_{E^c}) = \e$.
We have by $[\ov{C}]$ that $\ups(E)= \ups(X\times A_E) = \nu(A_E)$ and similarly for $E^c$, so the set $(A_E\sm N_E)\cup (A_{E^c}\sm N_{E^c})$ has $\nu$-measure one and is disjoint from $K$. Claim \ref{claim:c.sets} follows.

\m

Using Claim \ref{claim:c.sets}, the proof is completed as in the proof of Proposition \ref{E}, making appropriate changes as above.
\end{proof}

The following example shows that the assumption that $\eta$ satisfies $(L2)$ in (ii)(b) cannot be deleted.

\begin{ex}\label{ex:L2.needed}
Let $X=Y=[0,1]$, $\mu=$ Lebesgue measure restricted to Borel sets of measure $0$ or $1$, $\nu=$ Lebesgue measure, $\ups=\mu\,\wh{\otimes}\,\nu$.

Note that each $\mu\otimes\nu$-measurable set $E$ has the property that for some measure zero set $P\sq X$, $E\cap (P^c\times Y) = P^c\times A$ for some measurable set $A$ (since these sets form a $\s$-algebra containing the products $U\times V$, where $U$ is $\mu$-measurable and $V$ is $\nu$-measurable). The measure zero sets for $\ups$ are therefore the sets which are contained in $(P\times Y)\cup (X\times Q)$ for some measure zero sets $P,Q$.

Let $u\colon X\to [0,1/2)$ and $h\colon X\to [1/2,1]$ be functions such that $u$ is one-to-one and has measure zero range, and each fiber $h^{-1}(y)$ of $h$ is a Bernstein subset of $X$, $y\in [1/2,1]$. Let $\de$ be the unique lifting for $\mu$, and let $\eta_0$ be any lifting for Lebesgue measure satisfying $\eta_0([0,1/2)) = [0,1/2)$, and for measurable sets $E\sq[0,1]$, let $\eta(E) =\eta_0(E)$ unless $E$ has the form $E = [0,1/2)\sm \{u(x)\}$, in which case set $\eta([0,1/2)\sm \{u(x)\})=[0,1/2)\cup \{h(x)\}$.

Let $E\sq X\times Y$ be $\ups$-measurable. By \cite{bms4}, Proposition 5.2 (iii), we have $E =_\ups L:= X\times A$ for some measurable set $A$. The set $E\,\triangle\,(X\times A)$ is contained in $(P\times Y)\cup (X\times Q)$ for some measure zero sets $P,Q$, so $E\cap (P^c\times Q^c) = P^c\times (A\cap Q^c)$. If $A=_\nu[0,1/2)$ fails, then we also do not have $E_x =_\nu [0,1/2)$ for any $x$ and hence $\eta_\bullet(E)\cap(P^c\times Y) = P^c\times \eta(A)$. This gives for any $y\in \eta(A)$, $[\eta_\bullet(E)]^y \cap P^c = P^c$ and therefore $[\eta_\bullet(E)]^y$ has measure one and $[\varphi(E)]^y = \de([\eta_\bullet(E)]^y) = X$.
When $y\notin \eta(A)$, we get $[\eta_\bullet(E)]^y \cap P^c = \e$ and therefore $[\eta_\bullet(E)]^y$ has measure zero and $[\varphi(E)]^y = \de([\eta_\bullet(E)]^y) = \e$. It follows that $\varphi(E) = X\times \eta(A)$ and hence $\varphi(E) =_\ups E$.

If $A=_\nu[0,1/2)$ holds, then for some $N\in\vS_0$, we have for $x\in N^c$ that $E_x\in T$ and $E_x =_\nu A$, so $[\eta_\bullet(E)]_x = \eta(E_x) = [0,1/2)$ or $[0,1/2)\cup \{w(x)\}$ for some $w(x)\in [1/2,1]$. Only countable many fibers $w^{-1}(y)$ of the (partially defined) function $w$ can be measurable of positive measure. For any other value of $y$, $[\eta_\bullet(E)]^y$ is of measure zero or not measurable, so $[\varphi(E)]^y=\e$. It follows that $\varphi(E)\cap (N^c\times Y)$ equals the union of $N^c\times [0,1/2)$ and countably many sets of measure zero of the form $\de(w^{-1}(y))\times\{y\}$ and hence equals $E$ modulo $\ups$.

However, identifying $u$ with its graph $\{(x,u(x)):x\in X\}$, for the set $E=(X\times [0,1/2))\sm u$, all the sections $[\eta_\bullet(E)]^y = h^{-1}(y)$ are nonmeasurable for $y\in [1/2,1]$, so $\eta$ is not a weak $2$-marginal for $\vY$.
\end{ex}

\begin{prop}\label{F}
Assume that $\mu$ is complete and that $[C]$ and $[\ov{C}]$ hold.
Let $\de\in P^p(\mu)$, $\eta\in P^p(\nu)$. Assume that $(\de,\eta)$ is $\LL^p(\ups)$-preserving.
If $\eta$ is a weak $2$-marginal with respect to $\LL^p(\ups)$, then
$\de\odot\eta\in P^p(\ups)$.
\end{prop}

\begin{proof} The product $\de\odot\eta$ satisfies condition
$(l2)$ by Lemma \ref{01} (iii), and $(\de\odot\eta)(f)=f$ for $f$ constant, by Lemma \ref{01} (iv), so we need only to show condition $(l1)$, which follows by Proposition \ref{E}.
\end{proof}

\begin{prop}\label{G}
Assume that $\mu$ is complete and that $[C]$ and $[\ov{C}]$ hold.
Let $\de\in{V}_\R^\infty(\mu)$, $\eta\in {V}_\R^\infty(\nu)$. If $(\de,\eta)$ is $\LL^\infty(\ups)$-preserving and $\eta$ is a weak $2$-marginal with respect to $\LL^\infty(\ups)$ and generates $\LL^\infty(\mu)$-sections, then $\de\odot\eta\in{V}^\infty_\R(\ups)$.
\end{prop}

\begin{proof} By Proposition \ref{F} we get $\de\odot\eta\in{P}^\infty(\ups)$ and by Lemma \ref{01} \rm{(vi)} the product $\de\odot\eta$ is a linear map.
\end{proof}

\begin{cor}\label{H}
Assume that $\mu$ is complete and that $[C]$ and $[\ov{C}]$ hold.
If $(\de,\eta)\in\mathcal{G}(\mu)\times\mathcal{G}(\nu)$
is $\LL^\infty(\ups)$-preserving and $\eta$ is a weak $2$-marginal with respect to $\LL^\infty(\ups)$ and generates $\LL^\infty(\mu)$-sections, then $\de\odot\eta\in\mathcal{G}(\ups)$.
\end{cor}

\begin{proof} By  Proposition \ref{G} we get  $\de\odot\eta\in{V}^\infty_\R(\ups)$; hence $\de\odot\eta\in\mathcal{G}(\ups)$ by Lemma \ref{01} \rm{(viii)}.
\end{proof}

\begin{que}\label{Q}
Let $(\de,\eta)\in{V}^\infty_\R(\mu)\times{V}^\infty_\R(\nu)$ preserve measurability, and let $\eta$ be a weak 2-marginal with respect to $\LL^\infty(\mu\,\wh\otimes\,\nu)$. Assume that $\de\odot\eta\in{V}^\infty_\R(\mu\,\wh\otimes\,\nu)$. Does $\eta$ generate $\mu$-measurable sections in $\LL^\infty(\mu)$?
What if we assume also that $(\de,\eta)\in\mathcal{G}(\mu)\times\mathcal{G}(\nu)$ and $\de\odot\eta\in\mathcal{G}(\mu\,\wh\otimes\,\nu)$?
\end{que}

The next result is complementary to Theorem 2.6 from \cite{mms15}.

\begin{thm}\label{K}
Assume that $\mu$ and $\nu$ are complete.
Let $\de\in\mathcal{G}(\mu)$, $\eta\in{A}\mathcal{G}(\nu)$. If $(\de,\eta)$ is $\LL^\infty(\ups)$-preserving and $\eta$ generates $\LL^\infty(\mu)$-sections,  then $\pi:=\de\odot\eta\in\mathcal{G}(\mu\,\wh\otimes\,\nu)$ and the following hold.
\begin{enumerate}
\item
$\pi\in\de\otimes\eta$.
\item
$[\pi(f)]^y=\de([\pi(f)]^y)\;\;\mbox{for all}\;\; f\in\LL^\infty(\mu\,\wh\otimes\,\nu)\;\;\mbox{and all}\;\; y\in{Y}$.
\end{enumerate}
\end{thm}

\begin{proof} Since $\eta\in{A}\mathcal{G}(\nu)$ and $\mu$ and $\nu$ are complete, it follows by \cite{mms15}, Proposition 2.4
that there exists $\varphi\in\mathcal{G}(\mu\,\wh{\otimes}\,\nu)$ such that for every $f \in\LL^\infty(\mu\,\wh{\otimes}\,\nu)$ there exists a set $N_f \in \vS_0$ such that
$[\varphi(f)]_x = \eta([\varphi(f)]_x)$ for every $x\in X\sm N_f$.
By Lemma \ref{VM19a}, taking $\ups=\mu\,\wh{\otimes}\,\nu$ and $\mathcal{M}=\LL^\infty(\ups)$, we get that $\eta_\bullet$ satisfies $(l1)$ and hence $\eta$ is a weak $2$-marginal with respect to $\LL^\infty(\mu\,\wh\otimes\,\nu)$. It then follows by Corollary \ref{H} that $\pi:=\de\odot\eta\in\mathcal{G}(\mu\,\wh\otimes\,\nu)$, and so applying Lemma \ref{01}, \rm{(ii)}, \rm{(v)}, we obtain the validity of the statements \rm{(i)} and \rm{(ii)}.
\end{proof}

The next result corresponds to \cite{mms14}, Proposition 5.1, concerning liftings for sets, and has a similar proof, which we enclose for completeness.

\begin{prop}\label{000}
Assume that $\mu$ is complete.
If $(\de,\eta)\in\Lambda^\infty(\mu)\times\Lambda^\infty(\nu)$ is $\LL^\infty(\ups)$-preserving and $\eta$ is $2$-marginal with respect to $\LL^\infty(\mu\,\wh\otimes\,\nu)$ the the following statements are equivalent
\begin{enumerate}
\item
$\eta$ is $\mu$-smooth;
\item
$\de\odot\eta\in\Lambda^\infty(\mu\,\wh\otimes\,\nu)$.
\end{enumerate}
\end{prop}

\begin{proof}
For (i) $\Rightarrow$ (ii), by Corollary \ref{H} we get $\de\odot\eta\in\mathcal{G}(\mu\,\wh\otimes\,\nu)$, implying along with Lemma \ref{01} (vii) that $\de\odot\eta\in\Lambda^\infty(\mu\,\wh\otimes\,\nu)$.

For (ii) $\Rightarrow$ (i),
let $f\in\LL^\infty(\mu\,\wh\otimes\,\nu)$ and let $g=1-f$. By Proposition \ref{p:C.for.fcts} (ii), there is an $N\in\vS_0$ such that for $x\notin N$. $f_x,g_x\in\LL^\infty(\nu)$. When $x\notin N$, we have $f_x + g_x = 1$, so $\eta(f_x) + \eta(g_x) = \eta(1)=1$. Thus, we have for any $y\in Y$ the equation $[\eta_\bullet(f)]^y(x) + [\eta_\bullet(g)]^y(x) = 1$ for $x\notin N$. Since $\mu$ is complete, it follows that $[\eta_\bullet(f)]^y \in\LL^\infty(\mu)$ if and only if $[\eta_\bullet(g)]^y\in\LL^\infty(\mu)$.
Since for any $y\in Y$ we have
\[
1 = [(\de\odot\eta)(1)]^y = [(\de\odot\eta)(f+g)]^y = [(\de\odot\eta)(f)]^y + [(\de\odot\eta)(g)]^y,
\]
the conditions $[\eta_\bullet(f)]^y\in\LL^\infty(\mu)$ and $[\eta_\bullet(g)]^y\in\LL^\infty(\mu)$ do not both fail, and hence they both hold.
Thus, $\eta$ generates $\LL^\infty(\mu)$-sections and therefore is $\mu$-smooth.
\end{proof}

\begin{thm}\label{001}
Let $(X,\vS,\mu)$ be complete and non-atomic and let $(Y,T,\nu)$ be complete, non-atomic and perfect. If the pair $(\de,\eta)\in\Lambda^\infty(\mu)\times\Lambda^\infty(\nu)$ is $\LL^\infty(\mu\,\wh{\otimes}\,\nu)$-preserving and $\eta$ is a $2$-marginal with respect to $\LL^\infty(\mu\,\wh\otimes\,\nu)$ then $\de\odot\eta\in{P}^\infty(\mu\,\wh\otimes\,\nu)$ but $\de\odot\eta\notin\Lambda^\infty(\mu\,\wh\otimes\,\nu)$.
\end{thm}

\begin{proof}
We have that $\eta$ is a weak $2$-marginal with respect to $\LL^\infty(\mu\,\wh\otimes\,\nu)$ by Remark \ref{r:Da}. That $\de\odot\eta\in{P}^\infty(\mu\,\wh\otimes\,\nu)$ then follows from Proposition \ref{F}.
For the last statement suppose we had $\de\odot\eta\in\vL^\infty(\mu\,\wh{\otimes}\,\nu)$.

By Proposition 9.24 
$\eta$ is $\mu$-smooth, i.e., it is a 2-marginal with respect to $\mathcal{L}^\infty(\mu\,\wh{\otimes}\,\nu)$ and generates $\mathcal{L}^\infty(\mu)$-sections.
Since $\eta$ is 2-marginal with respect to $\mathcal{L}^\infty(\mu\,\wh{\otimes}\,\nu)$, the induced lifting on $T$, which we also denote by $\eta$, is a $2$-marginal with respect to $\vS\,\wh{\otimes}\,{T}$. (The induced lifting on $T$ is given by $B\mapsto \eta(B)$ where $\eta(\chi_B)=\chi_{\eta(B)}$. We have $\eta_\bullet(\chi_E) = \chi_{\eta_\bullet(E)}$ for any set $E\sq X\times Y$, as is readily verified.)
Thus, applying \cite{bms4} Lemma 4.11 (see \cite{bms6} Lemma 4.8 for an important correction to the proof%
\footnote{In \cite{bms6} Lemma 4.8, $\xi\in\vartheta(\nu)$ should read $\xi\in\vartheta(I_2)$. In our application here, we have $\xi=\eta$.}
) we obtain the existence of a lifting $\pi\in\vL(\mu\,\wh{\otimes}\,\nu)$ such that
\begin{equation}\label{eq:x}
\eta([\pi(E)]_x)=[\pi(E)]_x\;\;\mbox{for all}\;\; E\in\vS\,\wh{\otimes}\,{T}\;\;\mbox{and}\;\; x\in{X}.
\end{equation}
By standard arguments it then follows that
$\eta([\pi(f)]_x)=[\pi(f)]_x$ for all $f\in \LL^\infty(\mu\,\wh{\otimes}\,\nu)$ and $x\in X$. (From (\ref{eq:x}) this holds for $f=\chi_E$, $E\in\vS\,\wh{\otimes}\,{T}$, and hence for simple functions. Then use the fact that any $f\in \LL^\infty(\mu\,\wh{\otimes}\,\nu)$ is a uniform limit of simple functions and the fact that liftings preserve uniform limits.)

Since $\eta$ generates $\mathcal{L}^\infty(\mu)$-sections, and $[\eta_\bullet(f)]_x = \eta(f_x) = \eta([\pi(f)]_x) = [\pi(f)]_x$ for $\mu$-almost all $x\in X$, it follows that  $[\pi(f)]^y\in\mathcal{L}^\infty(\mu)$ for all $f\in\mathcal{L}^\infty(\mu\,\wh{\otimes}\,\nu)$ and $y\in{Y}$. (Cf.\ the proof of Lemma 9.4.)
Thus, we have
\begin{equation}\label{eq:xx}
[\pi(E)]^y\in\vS\;\;\mbox{for all}\;\; E\in\vS\,\wh{\otimes}\,{T}\;\;\mbox{and}\;\; y\in{Y}.
\end{equation}
So, there exist liftings $\eta\in\vL(\nu)$ and $\pi\in\vL(\mu\,\wh{\otimes}\,\nu)$ satisfying conditions \eqref{eq:x} and \eqref{eq:xx}, which contradicts \cite{mms2} Theorem 4.3.
\end{proof}

\end{document}